\documentclass[11pt,reqno]{amsart}

\usepackage{amsfonts}
\usepackage{eurosym}
\usepackage{amssymb}
\usepackage{amsthm}
\usepackage{amsmath}
\usepackage{amsaddr}
\usepackage[x11names]{xcolor}
\usepackage{bm}
\usepackage{cite}
\usepackage{mathrsfs}
\usepackage[OT1]{fontenc}
\usepackage[left=2.3cm, right=2.3cm, top=3cm]{geometry}
\usepackage{hyperref}
\hypersetup{colorlinks=true, linkcolor=blue, citecolor=red}

\RequirePackage{times}
\newtheorem{theorem}{Theorem}[section]

\newtheorem{lemma}[theorem]{Lemma}

\newtheorem{remark}[theorem]{Remark}

\allowdisplaybreaks[4]

\def\n{\textbf{\textit{n}}}
\def\R3{\mathbb{R}^3}
\def\F2o{\overline{F_2}}

\def \E{\mathcal{E}}
\def \C{\overline{C}}

\def\d{{\rm d}}
\def \l {\langle}
\def \r {\rangle}
\def\V{{\mathbf{V}}}

\def\H{\mathbf{H}}
\def\W{\mathbf{W}}
\def\A{\mathbf{A}}

\def\f{\textbf{\textit{f}}}
\def\g{\textbf{\textit{g}}}

\def\u{\textbf{\textit{u}}}
\def\uu{\textbf{\textit{u}}}
\def\vv{\textbf{\textit{v}}}
\def\ww{\textbf{\textit{w}}}
\def\ddt{\frac{\d}{\d t}}
\def\C{\mathcal{C}}

\def\P{\mathbb{P}}
\def\div{\mathrm{div}\,}
\def\L2{L^2(\Omega)}
\def\curl{\mathrm{curl}\, }

\def \au {\rm}
\def \ti {\it}
\def \jou {\rm}
\def \bk {\it}
\def \no#1#2#3 {{\bf #1} (#3), #2.}
\def \eds#1#2#3 {#1, #2, #3.}
\def \nome#1#2 {{\bf #1}, (#2).}

\begin{document}
\title[Mass-Conserving Allen-Cahn Approximation for Binary Fluids]
{On the Mass-Conserving Allen-Cahn Approximation \\ for Incompressible Binary Fluids}

\author[A. Giorgini, M. Grasselli \& H. Wu]{Andrea Giorgini$^\dagger$, Maurizio Grasselli$^\ddagger$ \& Hao Wu$^\ast$
}

\address{$^\dagger$Department of Mathematics \\
Imperial College London\\
London, SW7 2AZ, UK}
\email{a.giorgini@imperial.ac.uk}

\address{$^\ddagger$Dipartimento di Matematica\\
 Politecnico di Milano\\
Milano 20133, Italy}
\email{maurizio.grasselli@polimi.it}

\address{$^\ast$School of Mathematical Sciences \\
Key Laboratory of Mathematics for Nonlinear
Sciences (Fudan University), Ministry of Education\\
Shanghai Key Laboratory for Contemporary Applied Mathematics\\
Fudan University\\
Shanghai 200433, China}
\email{haowufd@fudan.edu.cn}




\begin{abstract}
This paper is devoted to the global well-posedness of two Diffuse Interface systems modeling the motion of an incompressible two-phase fluid mixture in presence of capillarity effects in a bounded smooth domain $\Omega\subset \mathbb{R}^d$, $d=2,3$.
We focus on dissipative mixing effects originating from the mass-conserving Allen-Cahn dynamics with the physically relevant Flory-Huggins potential. More precisely, we study the mass-conserving Navier-Stokes-Allen-Cahn system for nonhomogeneous fluids and the mass-conserving Euler-Allen-Cahn system for homogeneous fluids. We prove existence and uniqueness of global weak and strong solutions as well as their property of separation from the pure states. In our analysis, we combine the energy and entropy estimates, a novel end-point estimate of the product of two functions, a new estimate for the Stokes problem with non-constant viscosity, and logarithmic type Gronwall arguments.
\end{abstract}

\maketitle


\section{Introduction}
The flow of a two-phase or multi-phase fluid mixture is nowadays one of the most attractive theoretical and numerical problems in Fluid Mechanics (see, e.g., \cite{AMW1998,GKL2018,LIN2012,PS2016} and the references therein). This is mainly due to the complicated interplay between the motion of the moving and deforming free interfaces separating the two fluids (or phases) and the dynamics of surrounding fluids. Besides, the phenomenon of liquid-liquid phase separation has nowadays become a sort of paradigm in Cell Biology (see, e.g., \cite{AD2019,MZ2022,HWF2014,ShB17}).

A natural description of the dynamics of fluid mixtures is based on a free-boundary formulation. Let $\Omega$ be a bounded domain in $\mathbb{R}^d$ with $d=2,3$, and $T>0$. We assume that $\Omega$ is filled by two incompressible fluids that are immiscible (e.g., oil and water). Denote by $\Omega_1=\Omega_1(t)$ and $\Omega_2=\Omega_2(t)$ the subsets of $\Omega$ containing, respectively, the first and the second fluid components for any time $t\geq 0$. Then the  equations of motion are given by
\begin{equation}
\label{FB}
\begin{cases}
\rho_1 \big( \partial_t \uu_1 + \uu_1 \cdot \nabla \uu_1 \big) - \nu_1 \div D \uu_1 +\nabla p_1 =0, \quad &\div \uu_1 =0, \quad \text{ in } \Omega_1\times (0,T),\\
\rho_2 \big( \partial_t \uu_2 + \uu_2 \cdot \nabla \uu_2 \big) - \nu_2 \div D \uu_2 +\nabla p_2 =0, \quad &\div \uu_2 =0, \quad \text{ in } \Omega_2 \times (0,T).
\end{cases}
\end{equation}
Here, $\uu_1$, $\uu_2$ and $p_1$ and $p_2$ are, respectively, the velocities and pressures of the two fluids, while $\rho_1, \rho_2$ and $\nu_1,\nu_2$ are the (constant) densities and viscosities of the two fluids. The symmetric gradient is defined by $D=\frac12 (\nabla +\nabla^t)$. In system \eqref{FB} the effect of the gravity is neglected for the sake of simplicity. Denoting by $\Gamma=\Gamma(t)$ the (moving) interface between the time-varying domains $\Omega_1$ and $\Omega_2$, then system \eqref{FB} can be equipped with the classical free boundary conditions
\begin{equation}
\label{Y-L}
\uu_1=\uu_2, \quad \big( \nu_1 D \uu_1 -\nu_2 D \uu_2 \big) \cdot \n_\Gamma = (p_1-p_2+\sigma H)\n_\Gamma,  \quad \text{ on }\Gamma \times (0,T),
\end{equation}
together with the no-slip boundary conditions
\begin{equation}
\label{FB-ns}
\uu_1=\mathbf{0}, \quad \uu_2=\mathbf{0} \quad \text{ on }\partial \Omega \times (0,T).
\end{equation}
The vector $\n_\Gamma$ in \eqref{Y-L} is the unit normal vector of the interface from $\partial \Omega_1(t)$, and $H$ denotes the mean curvature of the interface $\Gamma$ (i.e., $H= - \div \n_\Gamma$). In this setting, the free interface $\Gamma(t)$ is assumed to move with the velocity given by
\begin{equation}
\label{inter-vel}
V_{\Gamma(t)}= \uu \cdot \n_{\Gamma(t)}.
\end{equation}
The coefficient $\sigma>0$ in \eqref{Y-L} stands for the surface tension, which introduces a discontinuity in the normal stress proportional to the mean curvature of the interface. Since
$$\ddt \mathcal{H}^{d-1}(\Gamma(t))= - \int_{\Gamma(t)} H V_\Gamma \,  \d \mathcal{H}^{d-1},$$
where $\mathcal{H}^{d-1}$ denotes the $d-1$-dimensional Hausdorff measure,
the (formal) energy identity for system \eqref{FB}-\eqref{inter-vel} reads as follows
\begin{equation}
\label{FB-energy}
\ddt \Big\lbrace  \sum_{i=1,2} \int_{\Omega_i(t)} \frac{\rho_i}{2}|\uu_i|^2 \, \d x + \sigma \mathcal{H}^{d-1}(\Gamma(t)) \Big\rbrace
+  \sum_{i=1,2} \int_{\Omega_i(t)} \nu_i |D \uu_i|^2 \, \d x
=0.
\end{equation}
Concerning the mathematical analysis, we refer to \cite{A2007,DS1995,Plo1993,PS2010-1,PS2016,T1995} for the study of classical and varifold solutions to the system \eqref{FB}-\eqref{inter-vel} with suitable initial conditions.

The twofold Lagrangian and Eulerian nature of system \eqref{FB}-\eqref{Y-L} has led to the breakthrough idea to reformulate it in the Eulerian description by interpreting the effect of the surface tension as a singular force term localized at the free interface, see, e.g., the review paper \cite{SS2003}. Introduce the so-called level set function $\phi: \Omega\times(0,T) \rightarrow \mathbb{R}$ such that
$$
\phi>0 \quad \text{ in } \Omega_1\times (0,T), \quad \phi<0 \quad \text{ in }\Omega_2 \times (0,T), \quad  \phi=0 \quad \text {on } \Gamma \times (0,T),
$$
namely, the interface $\Gamma$ is identified as the zero level set of $\phi$.
We consider the Heaviside type function
\begin{equation}
\label{Heav}
K(\phi)=
\begin{cases}
1 \quad &\phi>0,\\
0 \quad &\phi=0,\\
-1 \quad &\phi<0,
\end{cases}
\end{equation}
and denote by $\uu$ the velocity field such that $\uu=\uu_1$ in $\Omega_1\times (0,T)$ and   $\uu=\uu_2$ in $\Omega_2\times (0,T)$.
It was shown in \cite[Section 2]{CHMS1994} that the system \eqref{FB}-\eqref{Y-L} is formally equivalent to
\begin{equation}
\label{FB2}
\begin{cases}
\rho(\phi) \big( \partial_t \uu + \uu \cdot \nabla \uu\big) - \div (\nu(\phi)D\uu) + \nabla P=\sigma [H(\phi) \nabla \phi] \delta (\phi),\\
\div \uu=0,\\
\partial_t \phi +\uu\cdot \nabla \phi = 0,
\end{cases}
\quad \text{in }\Omega\times(0,T),
\end{equation}
together with the boundary condition \eqref{FB-ns}. In \eqref{FB2}, we have
$$
\rho(\phi)= \rho_1 \frac{1+K(\phi)}{2}+ \rho_2 \frac{1-K(\phi)}{2}, \
\nu(\phi)=\nu_1 \frac{1+K(\phi)}{2}+ \nu_2 \frac{1-K(\phi)}{2}, \
H(\phi)= \div \left( \frac{\nabla \phi}{|\nabla \phi|} \right).
$$
Besides, $\delta$ is the Dirac distribution and $\nabla \phi$ is oriented as $\n_\Gamma$ on $\Gamma$. The equation \eqref{FB2}$_3$ represents the motion of the interface $\Gamma$ that is simply transported by the flow.
This follows from the immiscibility condition, which translates into
$(\uu, 1) \in  \text{Tan} \lbrace (x,t)\in \Omega \times (0,T): x\in \Gamma(t) \rbrace$.
Although the system \eqref{FB2} seems to be more amenable than the system \eqref{FB}-\eqref{Y-L}, the presence of the Dirac mass still makes the analysis challenging \cite{LIN2012}.

In the literature, two different approaches have been introduced to overcome the singular nature of the right-hand side of \eqref{FB2}$_1$, which both rely on the idea of continuous transition at the interface. The first approach is the Level Set method developed in the seminal works \cite{OS1988,OF2002} (see \cite{CHMS1994,SSO1994,SS2003} for applications to two-phase flows). This approach consists in approximating the Heaviside function $K(\phi)$ by a smoothing regularization $K_\varepsilon(\phi)$.
More precisely, for a given $\varepsilon>0$, we introduce the function (cf. \cite{SS2003})
\begin{equation}
\label{Heav-e}
K_\varepsilon(\phi)=
\begin{cases}
1 \quad &\phi>\varepsilon,\\
\displaystyle{\frac12 \left[ \frac{\phi}{\varepsilon} + \frac{1}{\pi} \sin \big( \frac{\pi \phi}{\varepsilon}\big) \right]}\quad &|\phi| \leq \varepsilon,\\
-1 \quad &\phi<-\varepsilon.
\end{cases}
\end{equation}
Then the resulting approximating system reads as follows
\begin{equation}
\label{LS1}
\begin{cases}
\rho_\varepsilon(\phi) \big( \partial_t \uu + \uu \cdot \nabla \uu\big) - \div (\nu_\varepsilon(\phi)D\uu) + \nabla P=\sigma [H(\phi) \nabla \phi] \delta_\varepsilon (\phi),\\
\div \uu=0,\\
\partial_t \phi +\uu\cdot \nabla \phi = 0,
\end{cases}
\quad \text{in } \Omega \times (0,T),
\end{equation}
 where
$$
\rho_\varepsilon(\phi)= \rho_1 \frac{1+K_\varepsilon(\phi)}{2}+ \rho_2 \frac{1-K_\varepsilon(\phi)}{2}, \quad
\nu_\varepsilon(\phi)=\nu_1 \frac{1+K_\varepsilon(\phi)}{2}+ \nu_2 \frac{1-K_\varepsilon(\phi)}{2}, \quad \delta_\varepsilon(\phi)= \frac{\mathrm{d} K_\varepsilon(\phi)}{\mathrm{d} \phi}.
$$
As a consequence of the approximation \eqref{Heav-e}, the thickness of the interface is approximately $2\varepsilon/|\nabla \phi|$. This necessarily requires that $|\nabla \phi|=1$ when $|\phi|\leq \varepsilon$, namely, $\phi$ is a signed-distance function near the interface. However, even though the initial condition is suitably chosen, the evolution under the transport equation \eqref{LS1}$_3$ does not guarantee that this property remains true for all time. This fact had led to different numerical algorithms aiming to avoid the expansion of the interface (see \cite{SSO1994,SS2003} and the references therein).
In addition, as pointed out in \cite{LT1998}, another drawback of this approach is that the dynamics is sensitive to the particular choice of the approximation for the surface stress tensor.

The second approach is the so-called Diffuse Interface method (see \cite{AMW1998,DF20,FLSY2005,LS03,LT1998,YFLS2004} and the references cited therein). This is based on the postulate that the free interface is a thin layer with positive volume, whose thickness is determined by the interactions of particles occurring at small scales. The latter phenomenon is suitably incorporated in the free energy of the system that maintains the integrity of the interface. In this context, the auxiliary phase function $\phi$ represents the difference between the local concentrations of the two fluids (or in some cases, just the rescaled density/volume fractions). This phase function (sometimes called a phase-field) may exhibit a smooth transition at the interface, which is identified as an intermediate level set between the two values $1$ and $-1$ corresponding to the pure phases. The evolution equations for the state variables (density, velocity and concentration) are derived by combining the theory of binary mixtures and the energy-based formalism from  thermodynamics and statistical mechanics. In particular, within the framework of diffuse interface, the surface stress tensor is replaced by a diffuse stress tensor whose action is essentially localized in the regions of high gradients, namely, $- \div(\nabla \phi \otimes \nabla \phi)$. This yields the well-known (Korteweg) capillary force (cf., e.g., \cite{AMW1998,LS03}). As a consequence, a Diffuse Interface system takes the following form
\begin{equation}
 \label{Complex}
\begin{cases}
\rho(\phi) \big( \partial_t \uu + \uu \cdot \nabla \uu \big) - \div (\nu(\phi)D\uu) + \nabla P= -  \sigma \div(\nabla \phi \otimes \nabla \phi),\\
\div \uu=0,\\
\partial_t \phi +\uu\cdot \nabla \phi =\Delta_{\text{diss}},
\end{cases}
\quad \text{in } \Omega \times(0,T),
\end{equation}
where $\uu$ is understood as the (volume) averaged velocity and the term $\Delta_{\text{diss}}$ includes suitable dissipative effects at the free interface.
The averaged density and viscosity of the binary fluids are now given by
\begin{equation}
\label{rhonu}
\rho(\phi)= \rho_1 \frac{1+\phi}{2}+ \rho_2 \frac{1-\phi}{2}, \quad
\nu(\phi)=\nu_1 \frac{1+\phi}{2}+ \nu_2 \frac{1-\phi}{2}.
\end{equation}
The coupling system \eqref{Complex} is equipped with the no-slip boundary condition
$\uu=\mathbf{0}$ on $\partial \Omega \times (0,T)$, together with suitable boundary condition(s) on $\phi$ depending on the specific form of $\Delta_{\text{diss}}$.
Besides, the total energy associated to system \eqref{Complex} is defined as
\begin{align}
E(\uu,\phi)= \int_{\Omega}  \frac12 \rho(\phi) |\uu|^2+ \frac12 |\nabla \phi|^2 + \Psi(\phi)   \, \d x,
\label{EEE}
\end{align}
where $\Psi:[-1,1] \rightarrow \mathbb{R}$ is a suitable potential function.
The first integrand in \eqref{EEE} corresponds to the kinetic energy of the fluid mixture. Next, we note that particles of the two fluids interact at a miscoscopic scale, and their disposition is the result of a competition between the diffusion and the attraction of molecules of the same type (i.e., mixing vs demixing or, ``hydrophilic" vs ``hydrophobic" effects). This competition is described in the Helmholtz free energy of the system such that (see e.g., \cite{CH1958})
$$
\E(\phi)=  \int_{\Omega}  \frac12 |\nabla \phi|^2 + \Psi(\phi) \, \d x.
$$
The gradient term describes the spatial heterogeneity in composition of the mixture. It corresponds to the tendency of the mixture to prefer to be uniform in space \cite{DF20} and provides an approximation of nonlocal interactions between particles (see e.g., \cite{GL98,YFLS2004}). The second part of $\E(\phi)$ is called the homogeneous free energy and the potential $\Psi$ is related to the mixing entropy (cf. e.g., \cite{LL2013}) that has the following form
\begin{equation}  \label{Log}
\Psi(s)=\frac{\theta}{2}\left[ (1+s)\ln(1+s)+(1-s)\ln(1-s)\right]-\frac{%
\theta_0}{2} s^2, \quad s \in (-1,1).
\end{equation}
The potential function given by \eqref{Log} is also closely related to the Flory-Huggins free energy for polymer solutions.
We consider hereafter the physically relevant case $0<\theta<\theta_0$, which implies,
in particular, that $\Psi$ is a non-convex potential with double well structure. In the case $\theta\geq \theta_0>0$, mixing prevails over demixing, and no phase separation will take place.

In the case $\Delta_{\text{diss}}\equiv 0$, the phase function $\phi$ solves a transport equation (cf. \eqref{FB2}, \eqref{LS1}), and the resulting system obeys the following energy identity (at least for smooth solutions)
 \begin{equation}
\label{Complex-EE}
\ddt E(\uu, \phi) +\int_{\Omega} \nu(\phi) |D \uu|^2 \, \d x=0.
\end{equation}
This system, also referred to as the Complex Fluid model in the literature (see e.g., \cite{LIN2012,LZ2014,LZ2015}), dissipates its total energy only due to the fluid viscosity, such that there is no dissipation for $\phi$. We refer the interested readers to \cite{AT2010} for local well-posedness of the initial boundary value problem in some anisotropic $L^q$-Sobolev spaces, and to \cite{LZ2014,LZ2015} for global
well-posedness of the Cauchy problem in $\mathbb{R}^3$ provided that $\rho=\nu=\sigma=1$ and the initial data are close to some equilibrium states. It is worth noting that when $\Delta_{\text{diss}}\equiv 0$, system \eqref{Complex} is closely related to the models for incompressible viscoelastic fluids (see e.g., \cite{LLZ2005, LZ2008}) and to the magneto-hydrodynamical (MHD) system for incompressible flows without magnetic diffusion (see e.g., \cite{LIN2012,RWXZ2014} and the references therein). Before proceeding with the introduction of diffusive relaxations in the transport equation and their physical motivations, it is important to point out two main properties of the partially dissipative system with $\Delta_{\text{diss}}\equiv 0$:
\begin{itemize}
\item[(1)] Conservation of mass:
\begin{equation}
\label{CM}
\int_{\Omega} \phi(t)\, \d x=\int_{\Omega} \phi_0\, \d x, \quad \forall \, t \in [0,T].
\end{equation}
\item[(2)] Conservation of $L^\infty(\Omega)$-norm:
\begin{equation}
\label{CLinf}
\| \phi(t)\|_{L^\infty(\Omega)}=\| \phi_0\|_{L^\infty(\Omega)}, \quad \forall \, t\in [0,T],
\end{equation}
which implies that
\begin{equation}
\label{Crange}
-1\leq \phi_0(x)\leq 1 \quad \text{a.e. in} \ \Omega \quad \Longrightarrow
\quad -1\leq \phi(x,t)\leq 1 \quad \text{a.e. in} \ \Omega \times (0,T).
\end{equation}
\end{itemize}

 The Diffuse Interface theory of binary fluid mixtures takes into account dissipative mechanisms occurring at the interface, namely, $\Delta_{\text{diss}}\neq 0$. In order to include dissipative effects in the dynamics of fluid concentrations, we define the first variation of the Helmholtz free energy $\E(\phi)$, that is usually called the chemical potential
\begin{align}
\mu= \frac{\delta \E(\phi)}{\delta \phi}= -\Delta \phi+ \Psi'(\phi).
\label{00mu}
\end{align}
Two fundamental models proposed in the literature for the conserved dynamics of binary mixtures are:
\begin{itemize}
\item[(1)] \textbf{Mass-conserving Allen-Cahn dynamics} (see e.g. \cite{RS1992,CHL2010})
\begin{equation}
\label{MAC}
\begin{cases}
\partial_t \phi +\uu\cdot \nabla \phi + m \big(\mu-  \overline{\mu}\big)=0, \quad \text{in} \ \Omega \times (0,T),\\
\partial_\n \phi=0, \qquad \qquad\qquad\quad\quad\quad\quad \  \text{on} \ \partial \Omega \times (0,T);
\end{cases}
\end{equation}

\item[(2)] \textbf{Cahn-Hilliard dynamics} (see e.g. \cite{CH1958,CH1961})
\begin{equation}
\label{CH}
\begin{cases}
\partial_t \phi +\uu\cdot \nabla \phi - \mathrm{div}(m \nabla \mu) =0, \quad \text{in} \ \Omega \times (0,T), \\
\partial_\n \phi= \partial_\n \mu=0, \qquad \qquad\qquad\quad\ \ \text{on} \ \partial \Omega \times (0,T).
\end{cases}
\end{equation}
\end{itemize}
Here in \eqref{MAC} $\overline{\mu}$ is the spatial average of the chemical potential defined by $
\overline{\mu}= \frac{1}{|\Omega|}\int_{\Omega} \mu \, \d x$, and $m$ represents the microscopic elastic relaxation time. We note that the equation \eqref{MAC} differs from the classical Allen-Cahn equation (see \cite{AC1979}) due to the presence of the nonlocal term $\overline{\mu}$, see \cite{BS1997,CHL2010,GO1997} and the references cited therein.

From the thermodynamic point of view, the relaxation for $\phi$ describes some generalized diffusion at free interfaces (cf. \cite{HMR2012,HMR2012-2,LT1998}) and $m$ is also referred to as the (diffusion) mobility. In addition, the mixing-demixing mechanism (which also translates into $\mu$) allows a balance that may avoid uncontrolled expansion or shrinkage of the interface layer (cf. \cite{FLSY2005,DF20}). As for the transport equation, both the mass-conserving Allen-Cahn equation \eqref{MAC} and the Cahn-Hilliard equation \eqref{CH} satisfy the property of mass conservation \eqref{CM} and the uniform $L^\infty$ bound \eqref{Crange}. The former is a consequence of the boundary conditions and the incompressibility condition for $\uu$, while the latter is essentially guaranteed by the singular nature of the potential function \eqref{Log}. It is worth mentioning that the Landau theory that leads to the well-known Ginzburg-Landau free energy with a smooth double-well potential like $\Psi_0(s)=\frac14(s^2-1)^2$ gives an approximation of the energy $\E(\phi)$. This regularization can be obtained, up to suitable constants, through a Taylor's expansion of the logarithmic potential $\Psi$. On one hand, it has been widely used in the related literature (see, for instance, \cite{E1989,LS03,Miranville} and the references cited therein). On the other hand, this polynomial approximation has the main drawback that in general the solution $\phi$ may not belong to the physical interval $[-1,1]$ as time evolves (see \cite{Miranville} for detailed discussions on the Cahn-Hilliard equation).

The fundamental feature of the Diffuse Interface method lies in the connection between the systems \eqref{FB}-\eqref{Y-L} and \eqref{Complex}. After rescaling the capillary tensor and the free energy by a parameter $\varepsilon>0$ (that is proportional
to the thickness of the diffuse interface), it is possible to show the convergence of \eqref{Complex} to \eqref{FB}-\eqref{Y-L} in a suitable sense as $\varepsilon \to 0^+$, i.e., the sharp interface limit. In particular, we have the convergence of the (Helmholtz) free energy $\int_{\Omega}  \left[(\varepsilon/2) |\nabla \phi|^2 + \varepsilon^{-1}\Psi(\phi)\right] \d x$ to the area functional $\mathcal{H}^{d-1}(\Gamma)$ (see e.g., \cite{Modica}), and the convergence of the stress tensor
\begin{equation}
\label{DCF}
-\int_{\Omega}  \varepsilon \div(\nabla \phi \otimes \nabla \phi) \cdot \vv \, \d x
\xrightarrow{\varepsilon\rightarrow 0} -\int_{\Gamma} \sigma H \n_\Gamma \cdot \vv \, \d \mathcal{H}^{d-1},
\end{equation}
for some suitable test function $\vv$, where the limit integral corresponds to the weak formulation of \eqref{Y-L} with $\sigma$ depending on the so-called optimal diffuse interface profile. We refer to, for instance, \cite{AL2018} for a rigorous proof on the convergence for a Stokes-Allen-Cahn system with a positive, $\varepsilon$-independent mobility $m$ for sufficiently small times and for well-prepared initial data. However, it is worth mentioning that the convergence in \eqref{DCF} to the mean curvature functional does not hold when $\Delta_{\text{diss}}\equiv 0$ (i.e., a pure transport equation for $\phi$) or when the $\varepsilon$-dependent mobility $m(\varepsilon)$ tends to zero ``too fast" as $\varepsilon\to 0^+$, e.g., $m(\varepsilon)=m_0\varepsilon^\theta$ $(\theta>2)$, see \cite{A2022, AL2014, AS2016} for detailed discussions. Roughly speaking, the pure transport equation (i.e., the lack of the diffusive part compared to the Allen-Cahn or Cahn-Hilliard models) destroys the right shape of the diffuse interface in the normal direction. This also motivates the study of the diffusive model \eqref{Complex} with \eqref{MAC} or \eqref{CH}.

In this paper, we first analyze a Diffuse Interface model for incompressible two-phase viscous fluids that extends the hydrodynamic system recently derived in \cite{JLL2017} (see also \cite[Part I, Chapter 2, Section 4.2.1]{GKL2018}). It accounts for the general situation with unmatched fluid densities and viscosities. The resulting Navier-Stokes-Allen-Cahn (NSAC for short) system reads as follows
\begin{equation}
 \label{Complex2}
\begin{cases}
\rho(\phi) \big( \partial_t \uu + \uu \cdot \nabla \uu \big) - \div (\nu(\phi)D\uu) + \nabla P= - \sigma \div(\nabla \phi \otimes \nabla \phi),\\
\div \uu=0,\\
\partial_t \phi +\uu\cdot \nabla \phi +  m \Big(\mu + \rho'(\phi)\displaystyle{\frac{|\uu|^2}{2}} - \xi\Big) =0,
\end{cases}
\quad \text{in } \Omega \times (0,T),
\end{equation}
 where $\uu$ denotes the volume averaged velocity of the binary fluid mixture and the chemical potential $\mu$ is defined by \eqref{00mu}.
 Without loss of generality, hereafter we simply take the two positive parameters as
$$\sigma=m=1,$$
since their values do not play an essential role in the subsequent analysis.
Neglecting possible nontrivial dynamics on the boundary, e.g., the moving contact lines, we assume that the above system is subject to a no-slip boundary condition for $\uu$ and a homogeneous Neumann boundary condition for $\phi$, namely,
\begin{equation}
\label{bc-C2}
\uu=\mathbf{0},\quad  \partial_{\n} \phi =0 \quad \text{ on } \partial\Omega
\times (0,T).
\end{equation}
We see that in the system \eqref{Complex2}, the conserved dynamics of $\phi$ is now described through a suitable modification of the Allen-Cahn equation \eqref{MAC} such that
$$
\partial_t \phi +\uu\cdot \nabla \phi +  \left( \mu + \rho'(\phi)\frac{|\uu|^2}{2} - \xi \right)=0 \quad \text{ in} \ \Omega \times (0,T),
$$
with $\xi$ being
 the Lagrange multiplier corresponding to the mass conservation, that is,
$$
\xi(t)= \frac{1}{|\Omega|}\int_\Omega  \mu+ \rho'(\phi)\frac{|\uu|^2}{2}  \, \d x \quad \text{ in} \ (0,T).
$$
The system \eqref{Complex2}-\eqref{bc-C2} (formally) satisfies the following basic energy law
\begin{align*}
\ddt E(\uu, \phi)+\int_{\Omega} \nu(\phi) |D \uu|^2 \, \d x + \|\partial_t \phi + \uu \cdot \nabla \phi \|_{\L2}^2=0,
\end{align*}
where the total energy $E(\uu, \phi)$ is the same as in \eqref{EEE}. The energy dissipation of the above system involves both the contributions due to the fluid viscosity and the interfacial mixing. In particular, the dissipative mechanism of $\phi$ is similar to that of the mass-conserving Allen-Cahn equation \eqref{MAC}, but it also includes extra effects due to the fluid advection and the difference of fluid densities in view of the  equation satisfied by $\phi$. Besides, other boundary conditions can be considered, see for instance, the simplified case with periodic boundary conditions \cite[Part I, Chapter 2, Section 4.2.3]{GKL2018}, and the more complicated case involving a generalized Navier boundary conditions for $\uu$ coupled with a dynamic boundary condition for $\phi$, which describes the moving contact line dynamics \cite{MCYZ2017}.

In the second part of this paper, we shall consider the following mass-conserving Euler-Allen-Cahn system for homogeneous incompressible two-phase flows
\begin{equation}
 \label{Euler}
\begin{cases}
\partial_t \uu + \uu \cdot \nabla \uu + \nabla P= - \div(\nabla \phi \otimes \nabla \phi),\\
\div \uu=0,\\
\partial_t \phi +\uu\cdot \nabla \phi + \big(\mu- \overline{\mu}\big)=0,
\end{cases}
\quad \text{ in } \Omega \times (0,T),
\end{equation}
endowed with the boundary conditions
\begin{equation}
\label{bc-EAC}
\uu\cdot \n =0,\quad  \partial_{\n} \phi =0 \quad \text{ on } \partial\Omega
\times (0,T).
\end{equation}
The coupled system \eqref{Euler} is obtained from \eqref{Complex2} in the case of inviscid flow with $\nu=0$ and matched densities (i.e., $\rho_1=\rho_2=1$, so that $\rho\equiv 1$). In particular, we note that it has a partially dissipative structure such that
\begin{equation*}
\ddt E(\uu, \phi) + \|\partial_t \phi + \uu \cdot \nabla \phi\|_{L^2(\Omega)}^2=0.
\end{equation*}

The aim of this paper is to address the existence, uniqueness and (possibly) regularity of the solutions to the aforementioned Diffuse Interface systems:  the Navier-Stokes-Allen-Cahn system  \eqref{Complex2}-\eqref{bc-C2} and the Euler-Allen-Cahn system \eqref{Euler}-\eqref{bc-EAC}.
On one hand, the purpose of our analysis is to stay as close as possible to a thermodynamically grounded framework by keeping densities and viscosities of the fluid mixture to be dependent on $\phi$, and by considering the physically relevant Flory-Huggins type potential \eqref{Log}. Although this choice requires some extra technical efforts (which indeed have independent interests from the mathematical point of view), it provides results that are physically more reasonable. On the other hand, by working in this general setting, we demonstrate that the dynamics originating from general initial data are global (in time) when the mass-conserving Allen-Cahn relaxation is taken into account. This is achieved in two and three dimensions for finite energy (weak) solutions of the general system \eqref{Complex2}, and further more, in two dimensions for more regular solutions in the cases of (i) non-constant density and viscosity, and of (ii) constant density and zero viscosity.

The mathematical literature concerning systems similar to the mass-conserving Navier-Stokes-Allen-Cahn system \eqref{Complex2} has been widely developed in last years, in terms of physical modeling, theoretic analysis and numerical simulation. There are different ways to model the unmatched densities for incompressible binary fluid mixtures, see for instance, \cite{AGG2012,B2001,GKL2018,HMR2012,HMR2012-2,FK2017,LSY2015,LT1998} and the references cited therein. The model \eqref{Complex2} can be derived via an energetic variational approach based on the least action principle and the maximum dissipation principle due to Onsager, see \cite{HKL10,JLL2017,GKL2018} (we also refer to \cite{LSY2015,NSW14} for the Navier-Stokes-Cahn-Hilliard system and its variants). One feature of our system is that it fulfills three important physical constraints such as conservation of mass, dissipation of energy and, in addition, the force balance (cf. \cite{GKL2018}). The additional nonlinear coupling term $(1/2)\rho'(\phi)|\uu|^2$ represents certain density force that well captures the macroscopic fluid effect on the microscopic descriptions due to density differences \cite{JLL2017}. On the other hand, as observed in \cite[Remark 2.2]{AGG2012} for the Navier-Stokes-Cahn-Hilliard system with unmatched densities, the term $(1/2)\rho'(\phi)|\uu|^2$ in the chemical potential $\mu$ is not an objective scalar, which thus makes the system \eqref{Complex2} not frame invariant. Nevertheless, an alternative frame invariant formulation could be introduced for the Navier-Stokes-Allen-Cahn system using an idea similar to that in \cite{AGG2012}. This leads to a differential system with a different coupling structure (see \cite{LOCY} for an attempt to a formal derivation), which is beyond the scope of this paper and will be analyzed in the next future.

As far as the mathematical analysis is concerned, the system \eqref{Complex2} with boundary conditions \eqref{bc-C2} and suitable initial conditions has been investigated in \cite{JLL2017} in the specific case of constant viscosity, a standard advective Allen-Cahn equation with a regular Landau potential $\Psi_0(s)=\frac14(s^2-1)^2$ and without the constraint of mass conservation. The authors proved the existence of a global weak solution in three dimensions and the existence as well as uniqueness of a global strong solution in two dimensions. In the two dimensional case, they also showed the convergence of a global weak solution to a single equilibrium as time goes to infinity and established the existence of a global attractor. Thanks to the choice of potential therein and the absence of mass constraint, the authors of \cite{JLL2017} could easily ensure that the phase function $\phi$ always takes its value in the physical interval $[-1,1]$. This property followed from a comparison principle for the advective Allen-Cahn equation and was crucial for the analysis therein. However, the mass constraint that is important for the conserved dynamics of binary fluids, does not allow to establish a similar comparison principle for our system \eqref{Complex2} if the double-well potential $\Psi$ is smooth. We also mention the previous contributions \cite{GG2010,GG2010-2,HM2019,W17,WX13,XZL2010} for the case with constant density, and \cite{AL2018} for the sharp interface limit in the quasi-stationary Stokes case. Besides, there are works devoted to Navier-Stokes-Allen-Cahn models in which the fluid density is regarded as an independent variable (see, for instance, \cite{FL2019,LDH2016,LH2018} and the references cited therein). In all these works, the potential was taken as the classical Landau double-well and the mass conservation property was not considered. The (non-conserved) compressible case (see \cite{Bl1999,FK2017} for modeling issues) has been analyzed, for instance, in \cite{DLL2013,FRPS2010,K2012,YZ2019} (see also \cite{W2011} for the sharp interface limit).
On the other hand, in comparison with the viscous case mentioned above, only few works have been addressed with the Euler-Allen-Cahn system \eqref{Euler}-\eqref{bc-EAC}. In this respect, we mention \cite{ZGH2011}, where the authors proved local existence of smooth solutions for the Euler-Allen-Cahn system in the case of a regular Landau potential and without the constraint of mass conservation (see also \cite{Gal2016} for the analysis of a nonlocal model).

Before concluding the introduction, let us make some more precise comments on the mathematical analysis of systems \eqref{Complex2} and \eqref{Euler}, in particular, on the main novelties of our techniques. Concerning the Navier-Stokes-Allen-Cahn system \eqref{Complex2} subject to the boundary conditions in \eqref{bc-C2}, we prove (1) the existence of a global weak solution with finite initial energy in both two and three dimensions, that is, with the initial datum $(\uu_0,\phi_0)\in \H_\sigma \times H^1(\Omega)$ (see Theorems \ref{weak-D} and \ref{W-S}), and (2) the existence of a global strong solution in two dimensions with more regular initial data $(\uu_0,\phi_0)\in \V_\sigma \times H^2(\Omega)$ such that $\Psi'(\phi_0)\in L^2(\Omega)$  (see Theorem \ref{strong-D}). For the latter, we combine the classical energy method with a new end-point estimate of the product of two functions in $L^2(\Omega)$ (see Lemma \ref{result1}) and a new estimate for the Stokes system with non-constant viscosity (see Appendix \ref{App-0}). The proof is concluded with a nontrivial logarithmic Gronwall argument that leads to a double-exponential control. However, in light of the singularity of the Flory-Huggins potential \eqref{Log}, uniqueness of these strong solutions turns out to be a hard task. To overcome this difficulty, we establish some global estimates on the derivatives of the mixing entropy
$$F(s)=\frac{\theta}{2}\left[ (1+s)\ln(1+s)+(1-s)\ln(1-s)\right], \quad s \in (-1,1)$$ which corresponds to the convex part of \eqref{Log}, provided that $\Vert\rho^\prime\Vert_{L^\infty(-1,1)}$ is small enough and $F''(\phi_0)\in L^1(\Omega)$. These entropy estimates allow us to prove that $F''(\phi)^2 \ln (1+F''(\phi)) \in L^1(\Omega \times (0,T))$, and in turn, the uniqueness of global strong solutions in two dimensions. As a further consequence of these entropy estimates,
we achieve the so-called uniform separation property, which means that the phase function $\phi$ stays uniformly away from the pure states $\pm 1$ in finite time (see Theorem \ref{Proreg-D}).\footnote{It is worth pointing out that in Theorem \ref{Proreg-D} the initial concentration $\phi_0$ for strong solutions is not strictly separated from the pure phases. Indeed, the imposed conditions $F'(\phi_0)\in L^2(\Omega)$ or $F''(\phi_0)\in L^1(\Omega)$ allow $\phi_0$ being arbitrarily close to $+1$ and $-1$. This is different from the case of the Navier-Stokes-Cahn-Hilliard system in two dimensions, see e.g., \cite{GMT2019,HW21}.}
This crucial fact, besides being physically relevant, can yield further regularity properties of the strong solutions, because the singular potential \eqref{Log} can thus be regarded as a smooth, globally Lipschitz function on a compact subset of $(-1,1)$ (cf. \cite{A2009,GGW2018,GMT2019} for hydrodynamic systems involving the Cahn-Hilliard equation). Also, it yields a rigorous justification of the approximation based on the regular double-well potential. Next, concerning the inviscid case, i.e., the Euler-Allen-Cahn system \eqref{Euler}-\eqref{bc-EAC}, although it turns out to be similar to the MHD system with magnetic diffusion and without viscosity, the classical argument in the literature (see, e.g., \cite{CW2011}) does not apply, mainly because of the singular potential \eqref{Log}. In our proof, it is crucial to use the structure of the incompressible Euler equations \eqref{Euler}$_1$-\eqref{Euler}$_2$ and the end-point estimate of the product in $L^2(\Omega)$ (see Lemma \ref{result1}). This gives the existence of global solutions with initial datum $(\uu_0,\phi_0) \in (\H_\sigma\cap \mathbf{H}^1(\Omega)) \times H^2(\Omega)$ in two dimensions. Moreover, in light of the entropy estimates, we also prove the existence of smoother global solutions originating from the initial data $(\uu_0,\phi_0) \in (\H_\sigma\cap \mathbf{W}^{1,p}(\Omega)) \times H^2(\Omega)$ provided that $p>2$ and $\nabla \mu_0:= \nabla (-\Delta \phi_0+\Psi'(\phi_0))\in \mathbf{L}^2(\Omega)$. These results are summarized in Theorem \ref{Th-EAC}.
\smallskip

\textbf{Plan of the paper.} In Section \ref{2} we introduce the notations, some functional inequalities and then prove an end-point estimate for the product of two functions. Section \ref{S-WEAK} is devoted to the existence and (possible) uniqueness of global weak solutions for the Navier-Stokes-Allen-Cahn system \eqref{Complex2}-\eqref{bc-C2}.
In Section \ref{S-STRONG} we study the existence and uniqueness of global strong solutions to the Navier-Stokes-Allen-Cahn system \eqref{Complex2}-\eqref{bc-C2} in two dimensions. Section \ref{EAC-sec} is devoted to the existence of  global solutions to the Euler-Allen-Cahn system \eqref{Euler}-\eqref{bc-EAC}.
Some open problems are presented in Section \ref{con}.
In Appendix \ref{App-0} we prove a result on the Stokes problem with variable viscosity, and in Appendix \ref{App} we recall the Osgood lemma as well as two logarithmic versions of the Gronwall lemma.

\section{Preliminaries}
\label{2}
\setcounter{equation}{0}
\subsection{Function spaces}
For a real Banach space $X$, its norm is denoted by $\|\cdot\|_{X}$.
The symbol $\langle\cdot, \cdot\rangle_{X',X}$ stands for the duality pairing between $X$ and its dual space $X'$. The boldface letter $\bm{X}$ denotes the vectorial space endowed with the product structure.
We assume that $\Omega\subset \mathbb{R}^d$, $d=2,3$, is a bounded domain with smooth boundary $\partial \Omega$, $\n$ is the unit outward normal vector on $\partial \Omega$, and
 $\partial_\n$ denotes the outer normal derivative on $\partial \Omega$.
We denote the Lebesgue spaces by $L^p(\Omega)$ $(p\geq 1)$  with norms $\|\cdot\|_{L^p(\Omega)}$. When $p=2$, the inner product in the Hilbert space $L^2(\Omega)$  is denoted by
$(\cdot, \cdot)$.
For $s\in \mathbb{R}$, $p\geq 1$, $W^{s,p}(\Omega)$
denotes the Sobolev space with its corresponding norm $\|\cdot\|_{W^{s,p}(\Omega)}$. If $p=2$, we use the notation $W^{s,p}(\Omega)=H^s(\Omega)$. For every $f\in (H^1(\Omega))'$, we denote by $\overline{f}$ the generalized mean value over $\Omega$ defined by
$\overline{f}=|\Omega|^{-1}\langle f,1\rangle_{(H^1(\Omega))',H^1(\Omega)}$. If $f\in L^1(\Omega)$, then $\overline{f}=|\Omega|^{-1}\int_\Omega f \, \d x$.
Thanks to the Poincar\'{e}-Wirtinger inequality, there exists a positive constant $C=C(\Omega)$ such that
\begin{equation}
\label{normH1-2}
\| f\|_{H^1(\Omega)}\leq C \left(\| \nabla f\|_{L^2(\Omega)}^2+ \left|\int_{\Omega} f \, \d x\right|^2\right)^\frac12, \quad \forall \, f \in H^1(\Omega).
\end{equation}
We introduce the following Hilbert spaces of solenoidal vector-valued functions (see \cite{Galdi,Temam})
\begin{align*}
&\H_\sigma= \left\lbrace \uu\in \mathbf{L}^2(\Omega): \mathrm{div}\, \uu=0,\ \uu\cdot \n =0\ \  \text{on}\ \partial \Omega \right\rbrace = \overline{C_{0,\sigma}^\infty(\Omega)}^{\mathbf{L}^2(\Omega)},\\
& \V_\sigma =\left\lbrace \uu\in \mathbf{H}^1(\Omega): \mathrm{div}\, \uu=0,\ \uu=\mathbf{0}\ \  \text{on}\ \partial \Omega \right\rbrace=\overline{C_{0,\sigma}^\infty(\Omega)}^{\mathbf{H}^1(\Omega)},
\end{align*}
where $C_{0,\sigma}^\infty(\Omega)$ is  the space of divergence-free vector fields in $C_{0}^\infty(\Omega)$.
We also use the notations $(\cdot, \cdot)$ and
$\| \cdot \|_{L^2(\Omega)}$ for
the inner product and the norm in $\H_\sigma$. The space $\V_\sigma$ is endowed with the inner product and norm given by
$( \uu,\vv )_{\V_\sigma}=
( \nabla \uu,\nabla \vv )$ and  $\|\uu\|_{\V_\sigma}=\| \nabla \uu\|_{L^2(\Omega)}$, respectively.
Its dual space is denoted by $\V_\sigma'$. We recall the Korn inequality
\begin{equation}
\label{KORN}
\|\nabla\uu\|_{L^2(\Omega)} \leq \sqrt2\|D\uu\|_{L^2(\Omega)} \leq \sqrt2 \| \nabla \uu\|_{L^2(\Omega)},
\quad \forall \, \uu \in \V_\sigma,
\end{equation}
where $D\uu =  \frac12\big(\nabla \uu+ (\nabla \uu)^t\big)$.
Next, we define the Hilbert space
$\W_\sigma= \mathbf{H}^2(\Omega)\cap \V_\sigma$
with the inner product and norm
$ ( \uu,\vv)_{\W_\sigma}=( \A\uu, \A \vv )$ and $\| \uu\|_{\W_\sigma}=\|\A \uu \|_{L^2(\Omega)}$, where $\A$ is the Stokes operator.
We recall that there exists a positive constant $C$ depending only on $\Omega$ such that
\begin{equation}
\label{H2equiv}
 \| \uu\|_{H^2(\Omega)}\leq C\| \uu\|_{\W_\sigma}, \quad \forall \, \uu\in \W_\sigma.
\end{equation}

\subsection{Analytic tools}
We report the Ladyzhenskaya, Agmon, Gagliardo-Nirenberg, Brezis-Gallouet-Wainger and trace interpolation inequalities:
\begin{align}
\label{LADY}
&\| f\|_{L^4(\Omega)}\leq C \|f\|_{L^2(\Omega)}^{\frac12}\|f\|_{H^1(\Omega)}^{\frac12},  &&\forall \, f \in H^1(\Omega), \ d=2,\\
\label{GN2}
&\| f\|_{L^p(\Omega)}\leq C p^\frac12 \| f\|_{L^2(\Omega)}^{\frac{2}{p}} \| f\|_{H^1(\Omega)}^{1-\frac{2}{p}},  &&\forall \, f \in H^1(\Omega),\ \  2\leq p<\infty, \ d=2,\\
\label{GN3}
&\| f\|_{L^p(\Omega)}\leq C(p) \| f\|_{L^2(\Omega)}^{\frac{6-p}{2p}} \| f\|_{H^1(\Omega)}^{\frac{3(p-2)}{2p}},  &&\forall \, f \in H^1(\Omega),\ \ 2\leq p\leq 6, \ d=3,\\
\label{Agmon2d}
&\| f\|_{L^\infty(\Omega)}\leq C \|f\|_{L^2(\Omega)}^{\frac12}\|f\|_{H^2(\Omega)}^{\frac12},  && \forall \, f \in H^2(\Omega),\ d=2,\\
\label{GN-L4}
&\| \nabla f\|_{L^4(\Omega)}\leq C\| f \|_{H^2(\Omega)}^\frac12
\| f \|_{L^\infty(\Omega)}^\frac12, && \forall \, f \in H^2(\Omega),\ d=2,3,\\
\label{BGI}
&\| f\|_{L^\infty(\Omega)}\leq C \| f\|_{H^1(\Omega)} \ln^\frac12 \left( {e}\frac{\| f\|_{H^2(\Omega)}}{\| f\|_{H^1(\Omega)}} \right), &&\forall \, f \in H^2(\Omega), \ d=2,\\
\label{BGW}
&\| f\|_{L^\infty(\Omega)}\leq C(p) \| f\|_{H^1(\Omega)} \ln^\frac12 \left( C(p) \frac{\| f\|_{W^{1,p}(\Omega)}}{\| f\|_{H^1(\Omega)}} \right) , &&\forall \, f \in W^{1,p}(\Omega), \ p>2, \ d=2,\\
\label{trace}
&\| f\|_{L^2(\partial \Omega)} \leq C \| f\|_{\L2}^\frac12 \| f\|_{H^1(\Omega)}^\frac12, &&\forall \, f \in H^1(\Omega), \ d=2,3.
\end{align}
Here, the positive constant $C$ depends only on $\Omega$, whereas the positive constant $C(p)$ depends on $\Omega$ and $p$.
\smallskip

We now prove the following end-point estimate for the product of two functions in two dimensions, which will play an important role in the subsequent analysis.

\begin{lemma}
\label{result1}
Let $\Omega$ be a bounded domain in $\mathbb{R}^2$ with smooth boundary $\partial\Omega$. Assume that $f\in H^1(\Omega)$, $g\in L^p(\Omega)$ for some $p>2$. Then, we have
\begin{equation}
\label{logprod}
\| f g\|_{L^2(\Omega)}\leq C \left(\frac{p}{p-2} \right)^\frac12 \| f\|_{H^1(\Omega)}
\| g\|_{L^2(\Omega)}  \ln^{\frac12} \left(  {e} |\Omega|^{\frac{p-2}{2p}} \frac{\| g\|_{L^p(\Omega)}}{\| g\|_{L^2(\Omega)}} \right),
\end{equation}
for some positive constant $C$ depending only on $\Omega$.
\end{lemma}
\begin{proof}
Let us consider the Neumann operator $A=-\Delta + I$ on $L^2(\Omega)$ with domain $D(A)=\lbrace u\in H^2(\Omega): \partial_{\n}u=0$ on $\partial \Omega\rbrace$. By the classical spectral theory, there exists a sequence of positive eigenvalues $\{\lambda_k\}$ ($k\in \mathbb{N}$) associated with $A$ such that $\lambda_1=1$, $\lambda_{k}\leq \lambda_{k+1}$ and $\lambda_{k}\nearrow +\infty$ as $k$ goes to $+\infty$. The sequence of eigenfunctions $w_k\in D(A)$ satisfying $A w_k=\lambda_k w_k$ forms an orthonormal basis in $L^2(\Omega)$ and an orthogonal basis in $H^1(\Omega)$.

Let us fix $N \in \mathbb{Z}^+$ whose value will be chosen later. We write the function $f$ as follows
\begin{equation}
\label{decompsition}
f=\sum_{n=0}^N f_n + f_N^{\bot},
\end{equation}
where
$$
f_n=\sum_{k:\, {e}^n\leq \sqrt{\lambda_k}<{e}^{n+1}}
(f,w_k) w_k, \quad f_{N}^{\bot}= \sum_{k:\,\sqrt{\lambda_k} \geq {e}^{N+1}} (f,w_k) w_k.
$$
By using the above decomposition and H\"{o}lder's inequality, we have
\begin{align*}
\| f g\|_{L^2(\Omega)} &\leq
\sum_{n=0}^N \| f_n g \|_{L^2(\Omega)}+
\| f_{N}^{\bot} g \|_{L^2(\Omega)} \\
& \leq \sum_{n=0}^N \| f_n\|_{L^\infty(\Omega)} \| g\|_{L^2(\Omega)}
+ \| f_N^{\bot}\|_{L^{p'}(\Omega)}
\| g\|_{L^p(\Omega)},
\end{align*}
where $p>2$ and $\frac{1}{p}+\frac{1}{p'}=\frac12$. By using \eqref{GN2} and \eqref{Agmon2d}, we obtain
\begin{align*}
\| fg\|_{L^2(\Omega)}
&\leq C \sum_{n=0}^N \| f_n\|_{L^2(\Omega)}^\frac12
\| f_n\|_{H^2(\Omega)}^\frac12 \| g\|_{L^2(\Omega)}+ C \left(\frac{2p}{p-2}\right)^\frac12 \| f_N^{\bot}\|_{L^2(\Omega)}^\frac{2}{p'}
\| f_N^{\bot}\|_{H^1(\Omega)}^{1-\frac{2}{p'}}
\| g\|_{L^p(\Omega)},
\end{align*}
for some $C$ independent of $p$. We recall that
$$
\| f_n\|_{L^2(\Omega)}^2= \sum_{k:\, {e}^n\leq \sqrt{\lambda_k}<{e}^{n+1}} |(f,w_k)|^2 \leq \frac{1}{{e}^{2n}} \sum_{k:\, {e}^n\leq \sqrt{\lambda_k}<{e}^{n+1}} \lambda_k |(f,w_k)|^2 = \frac{1}{{e}^{2n}} \| f_n\|_{H^1(\Omega)}^2,
$$
where we have used the fact $D(A^\frac12)=H^1(\Omega)$.
Observing that $\partial_\n f_n=0$ on $\partial \Omega$ (indeed $f_n$ is a finite sum of $w_k$'s), by the regularity theory of the Neumann problem, we have
\begin{align*}
\| f_n\|_{H^2(\Omega)}^2 &\leq C \| A f_n\|_{L^2(\Omega)}^2= C \sum_{k:\, {e}^n\leq \sqrt{\lambda_k}<{e}^{n+1}} \lambda_k^2 |(f,w_k)|^2\\
&\leq C  \sum_{k:\, {e}^n\leq \sqrt{\lambda_k}<{e}^{n+1}} {e}^{2(n+1)} \lambda_k |(f,w_k)|^2\\
&\leq C {e}^{2(n+1)} \| f_n\|_{H^1(\Omega)}^2.
\end{align*}
Thus, we infer that
$$
\| f_n\|_{L^2(\Omega)}^\frac12
\| f_n\|_{H^2(\Omega)}^\frac12 \leq C {e}^\frac12 \| f_n\|_{H^1(\Omega)},
$$
where the constant $C>0$ is independent of $n$. On the other hand, reasoning as above, we deduce that
$$
\| f_{N}^{\bot}\|_{L^2(\Omega)}^2\leq \frac{1}{{e}^{2(N+1)}} \| f_{N}^{\bot}\|_{H^1(\Omega)}^2.
$$

Combining the above inequalities and applying the Cauchy-Schwarz inequality, we get
\begin{equation}
\label{est2}
\begin{split}
\| fg\|_{L^2(\Omega)} &\leq
C \sum_{n=0}^N {e}^\frac12 \| f_n\|_{H^1(\Omega)} \| g\|_{L^2(\Omega)} + C \frac{\Big(\frac{2p}{p-2}\Big)^\frac12}{{e}^{\frac{2(N+1)}{p'}}} \| f_N^{\bot}\|_{H^1(\Omega)}  \| g\|_{L^p(\Omega)} \\
&\leq C  \| g\|_{L^2(\Omega)}
\left(  \sum_{n=0}^N {e}^\frac12 \| f_n\|_{H^1(\Omega)}
+ \frac{\Big(\frac{2p}{p-2}\Big)^\frac12}{{e}^{\frac{(p-2)(N+1)}{p}}}
\frac{\| g\|_{L^p(\Omega)}}{\| g\|_{L^2(\Omega)}} \| f_N^{\bot}\|_{H^1(\Omega)}\right) \\
&\leq C \| g\|_{L^2(\Omega)}
\left(   {e} (N+1) + \frac{\Big(\frac{2p}{p-2}\Big)}{{e}^{\frac{2(p-2)(N+1)}{p}}}
\frac{\| g\|^2_{L^p(\Omega)}}{\| g\|^2_{L^2(\Omega)}} \right)^\frac12
 \left( \sum_{n=0}^N
\| f_n\|_{H^1(\Omega)}^2  +\| f_N^{\bot}\|^2_{H^1(\Omega)}
\right)^\frac12 \\
&\leq  C \| g\|_{L^2(\Omega)}
\left(   {e} (N+1) + \frac{ \Big(\frac{2p}{p-2}\Big)}{{e}^{\frac{2(p-2)(N+1)}{p}}}
\frac{\| g\|_{L^p(\Omega)}^2}{\| g\|_{L^2(\Omega)}^2} \right)^\frac12 \| f\|_{H^1(\Omega)},
\end{split}
\end{equation}
where we have used the fact $p'=\frac{2p}{p-2}$. Here, the constant $C>0$ is independent of $N$. Now, we choose the positive integer $N$ such that
$$
\frac{p}{2(p-2)} \ln \left( {e} |\Omega|^{\frac{p-2}{p}}\frac{\| g\|^2_{L^{p}(\Omega)}}{\| g\|^2_{L^2(\Omega)}}\right) \leq N+1 < 1+ \frac{p}{2(p-2)}\ln \left({e} |\Omega|^{\frac{p-2}{p}}\frac{\| g\|^2_{L^p(\Omega)}}{\| g\|^2_{L^2(\Omega)}}\right).
$$
We observe that the logarithm term in the above relations is greater than $1$ for any function $g\in L^p(\Omega)$ with $p>2$ and $g\neq 0$.
Then by using the choice of $N$ in \eqref{est2}, we infer that
\begin{align*}
\| f g\|_{L^2(\Omega)}
&\leq
C  \| f\|_{H^1(\Omega)} \| g\|_{L^2(\Omega)}
\Bigg(  {e} \left[ 1+\frac{p}{2(p-2)} \ln \left( {e} |\Omega|^{\frac{p-2}{p}} \frac{\| g\|^2_{L^p(\Omega)}}{\| g\|^2_{L^2(\Omega)}}\right) \right] +\frac{2p}{{e} (p-2) |\Omega|^{\frac{p-2}{p}}} \Bigg)^\frac12\\
&\leq C  \| f\|_{H^1(\Omega)} \| g\|_{L^2(\Omega)}
\Bigg[  \frac{3{e}}{2} \frac{p}{(p-2)} \ln \left( {e}^2 |\Omega|^{\frac{p-2}{p}} \frac{\| g\|^2_{L^p(\Omega)}}{\| g\|^2_{L^2(\Omega)}}\right) +\frac{2p}{{e} (p-2) |\Omega|^{\frac{p-2}{p}}} \Bigg]^\frac12,
\end{align*}
which implies the desired conclusion.
\end{proof}

\begin{remark}\label{gmtPC}
Lemma \ref{result1} is a generalization of \cite[Proposition C.1]{GMT2019}, which states that in a smooth bounded domain $\Omega \subset \mathbb{R}^2$, for any $f,g\in H^1(\Omega)$, it holds
\begin{equation*}
\| f g\|_{L^2(\Omega)}\leq C \| f\|_{H^1(\Omega)}
\| g\|_{L^2(\Omega)}  \ln^{\frac12} \left(  {e} \frac{\| g\|_{H^1(\Omega)}}{\| g\|_{L^2(\Omega)}} \right).
\end{equation*}
Moreover, the conclusion of Lemma \ref{result1} also holds in $\mathbb{T}^2$.
\end{remark}

\begin{remark}
It is well-known that $H^1(\Omega)$ is not a Banach algebra in two dimensions. An interesting application of Lemma \ref{result1} together with the Brezis-Gallouet-Wainger inequality \eqref{BGW} is that
$$
\| f g\|_{H^1(\Omega)}\leq C_1 \| f\|_{H^1(\Omega)} \| g\|_{H^1(\Omega)} \ln^{\frac12} \left( C_2 \frac{\| g\|_{W^{1,p}(\Omega)}}{\| g\|_{H^1(\Omega)}}\right),
$$
for any $f\in H^1(\Omega)$, $g \in W^{1,p}(\Omega)$ with $p>2$, where $C_1$ and $C_2$ are two positive constants depending only on $\Omega$ and $p$.
\end{remark}

\begin{remark}
Lemma \ref{result1} can also be regarded as a generalization of the H\"{o}lder and Young inequalities. It provides a remedy for the lack of the critical Sobolev embedding $H^1(\Omega) \hookrightarrow L^\infty(\Omega)$ when $d= 2$. The inequality \eqref{logprod} is sharp since the product between $f$ and $g$ is not defined in $\L2$ if $f\in H^1(\Omega)$ and $g\in L^2(\Omega)$. Indeed, we have the following counterexample in $\mathbb{R}^2$:
$$
g(x)= \frac{1}{|x|\ln^{\frac34} \left(\frac{1}{|x|}\right)}, \quad f(x)=\ln^{\frac12-\frac{1}{100}} \left( \frac{1}{|x|}\right),
\quad 0< x\leq 1.
$$
We notice that $g \in L^2(B_{\mathbb{R}^2}(0,1))$ since
$$
\int_{B_{\mathbb{R}^2}(0,1)} |g(x)|^2 \, \d x= 2 \pi \int_0^1 \frac{1}{r \ln^{\frac32}\left( \frac{1}{r} \right) } \, \d r= 2\pi \int_1^{+\infty} \frac{1}{s\ln^{\frac32}(s)} \, \d s <+\infty.
$$
However, $g \notin L^p(B_{\mathbb{R}^2}(0,1))$ for any $p>2$ because
$$
\int_{B_{\mathbb{R}^2}(0,1)} |g(x)|^p \, \d x= 2 \pi \int_0^1 \frac{1}{r^{p-1} \ln^{\frac{3p}{4}} \left( \frac{1}{r} \right) } \, \d r = 2 \pi \int_1^{+\infty} \frac{1}{s^{3-p}\ln^{ \frac{3p}{4}}(s)} \, \d s =+\infty.
$$
We easily observe that $f \in L^2(B_{\mathbb{R}^2}(0,1))$, but $f \notin L^\infty(B_{\mathbb{R}^2}(0,1))$ since
$\lim_{|x|\rightarrow 0} f(x)=+\infty$.
This, in turn, implies that $f \notin W^{1,p}(B_{\mathbb{R}^2}(0,1))$ for any $p>2$, due to the Sobolev embedding theorem. Nonetheless, we have
$$
\partial_{x_i}f(x)= \left(\frac12 -\frac{1}{100}\right)  \frac{x_i}{|x|^2} \frac{1}{\ln^{\frac{1}{2}+\frac{1}{100}} \left(\frac{1}{|x|} \right)}, \quad i=1,2,
$$
such that
\begin{align*}
\int_{B_{\mathbb{R}^2}(0,1)} |\partial_{x_i}f(x)|^2 \, \d x &\leq 2\pi
\left(\frac12 -\frac{1}{100}\right)^2 \int_0^1
\frac{1}{r \ln^{2(\frac12+\frac{1}{100})} \left( \frac{1}{r}\right)}  \, \d r\\
&\leq C \int_0^1 \frac{1}{r \ln^{1+\frac{1}{50}} \left( \frac{1}{r} \right)} < + \infty.
\end{align*}
Thus, we have $f \in W^{1,2}(B_{\mathbb{R}^2}(0,1))$.
Finally, we observe that
\begin{align*}
\int_{B_{\mathbb{R}^2}(0,1)} |g(x)f(x)|^2 \, \d x &
= \int_{B_{\mathbb{R}^2}(0,1)} \frac{\ln^{1-\frac{1}{50}}
\left( \frac{1}{|x|} \right)}{|x|^2\ln^{\frac32} \left(\frac{1}{|x|}\right) } \, \d x\\
&= 2 \pi \int_0^1 \frac{1}{r \ln^{\frac12 + \frac{1}{50}} \left( \frac{1}{r} \right) } \, \d r =+\infty,
\end{align*}
namely, the product $fg \notin L^2(B_{\mathbb{R}^2}(0,1))$.
The above counterexample can be generalized to any pair of functions
$$
g(x)= \frac{1}{|x|\ln^\alpha \left(\frac{1}{|x|}\right)}, \quad f(x)=\ln^\beta \left( \frac{1}{|x|}\right),
\quad x\in B_{\mathbb{R}^2}(0,1),
$$
where $\frac12 <\alpha<1$ and $\beta< \frac12 $ such that  $\alpha-\beta<\frac{1}{2}$.
\end{remark}

\begin{remark}
The argument in the proof of Lemma \ref{result1} also implies that, for any $p\in (1,\infty)$, there exist two positive constants $\widetilde{C}_1$ and $\widetilde{C}_2$ depending only on $\Omega$ and $p$ such that
\begin{equation}
\left| \int_{\Omega} f  g \, \d x\right|\leq \widetilde{C}_1 \| f\|_{H^1(\Omega)} \| g\|_{L^1(\Omega)} \ln^\frac12 \left( \widetilde{C}_2 \frac{\| g\|_{L^p(\Omega)}}{\| g\|_{L^1(\Omega)}} \right), \quad \forall\, f\in H^1(\Omega), \ g \in L^p(\Omega).
\end{equation}
The above inequality can be compared with \cite[Theorem 1']{BB}.
\end{remark}

\section{Mass-Conserving NSAC System: Weak Solutions}
\setcounter{equation}{0}
\label{S-WEAK}

In this section, we consider the Navier-Stokes-Allen-Cahn system \eqref{Complex2} written in the following form:
\begin{equation}
 \label{NSAC-D}
\begin{cases}
\rho(\phi)\big( \partial_t \uu + \uu \cdot \nabla \uu \big)- \div \big( \nu(\phi)D\uu\big)+ \nabla P
= - \div(\nabla \phi \otimes \nabla \phi),\\
\div \uu=0,\\
\partial_t \phi +\uu\cdot \nabla \phi + \mu + \displaystyle{\rho'(\phi)\frac{|\uu|^2}{2}} =  \xi, \smallskip\\
\mu= -\Delta \phi + \Psi' (\phi), \qquad \xi= \displaystyle{\overline{\mu+ \rho'(\phi)\frac{|\uu|^2}{2}}},
\end{cases}
\quad \text{ in } \Omega \times (0,T),
\end{equation}
subject to the boundary conditions
\begin{equation}  \label{boundary-D}
\uu=\mathbf{0},\quad  \partial_{\n} \phi =0 \quad \text{ on } \partial\Omega
\times (0,T),
\end{equation}
and the initial conditions
\begin{equation}
\label{IC-D}
\uu(\cdot, 0)= \uu_0, \quad \phi(\cdot, 0)=\phi_0 \quad \text{ in } \Omega.
\end{equation}
The coefficients $\rho(\cdot)$ and $\nu(\cdot)$ represent the density and the viscosity of the fluid mixture depending on $\phi$. Let $\rho_1$, $\rho_2$ and $\nu_1$, $\nu_2$ are, respectively, the (positive) densities and viscosities of two homogeneous (different) fluids. We introduce the notations
\begin{align*}
&\rho^\ast=\max \lbrace \rho_1,\rho_2\rbrace,\qquad  \rho_\ast=\min \lbrace \rho_1,\rho_2\rbrace,\\
&\nu^\ast =\max \lbrace \nu_1,\nu_2 \rbrace,\qquad  \nu_\ast =\min \lbrace \nu_1,\nu_2 \rbrace,
\end{align*}
 Throughout this work, motivated by the linear interpolation density and viscosity functions in \eqref{rhonu}, we assume that
\begin{equation}
\label{Hp-rn}
\rho, \nu \in C^2([-1,1]): \quad \rho(s)\in [\rho_\ast, \rho^\ast], \quad  \nu(s)\in [\nu_\ast,\nu^\ast] \ \  \text{for}\ \ s\in[-1,1].
\end{equation}

The main result of this section concerns the existence of global weak solutions.
\begin{theorem}[Global weak solution]
\label{weak-D}
Let $\Omega$ be a bounded domain in $\mathbb{R}^d$, $d=2,3$, with smooth boundary $\partial\Omega$. Assume that the initial datum $(\uu_0,\phi_0)$ satisfies
$\uu_0 \in \H_\sigma, \phi_0\in H^1(\Omega)\cap L^\infty(\Omega)$ with $
\| \phi_0\|_{L^\infty(\Omega)}\leq 1$ and $ |\overline{\phi}_0|<1$ (i.e., the initial energy $E(\uu_0, \phi_0)$ is finite and the initial phase $\phi_0$ is not a pure state). Then there exists a global weak solution $(\uu,\phi)$ to problem  \eqref{NSAC-D}-\eqref{IC-D} in the following sense:
\begin{itemize}
\item[(i)] For all $T>0$, the pair $(\uu,\phi)$ satisfies
\begin{align*}
&\uu \in L^\infty(0,T;\H_\sigma)\cap L^2(0,T;\V_\sigma),\\
&\phi \in L^\infty(0,T; H^1(\Omega))\cap  L^q(0,T;H^2(\Omega)),\quad \partial_t \phi \in L^q(0,T;L^2(\Omega)), \\
&\phi \in L^\infty(\Omega\times (0,T)) : |\phi(x,t)|<1 \ \ \text{a.e. in} \ \Omega\times(0,T),\\
&\mu \in L^q(0,T;L^2(\Omega)),\quad F'(\phi)\in L^q(0,T;L^2(\Omega)),
\end{align*}
with $q=2$ if $d=2$, and $q=\frac{4}{3}$ if $d=3$.
\item[(ii)] For all $T>0$, the system \eqref{NSAC-D}  is solved as follows
\begin{align*}
&-\int_0^T\!\int_\Omega \big(\rho'(\phi) \partial_t \phi \, \eta(t)+ \rho(\phi)\, \eta'(t)\big)\, \uu\cdot \vv \, \d x \d t+ \int_0^T\!\int_\Omega \big(\rho(\phi)\uu\cdot\nabla\uu \big) \cdot \vv \, \eta(t) \, \d x\d t \\
&\qquad + \int_0^T\!\int_\Omega \nu(\phi) \big(D\uu: D \vv\big) \, \eta(t) \, \d x\d t\\
&\quad = \int_\Omega \rho(\phi_0)\uu_0 \cdot \vv\, \eta(0) \, \d x+
\int_0^T\!\int_\Omega \big((\nabla \phi \otimes \nabla \phi): \nabla \vv \big)\eta(t) \, \d x\d t,
\end{align*}
for any $\vv \in \V_\sigma$, $\eta\in C^1([0,T])$ with $\eta(T)=0$, and
\begin{align*}
&\partial_t \phi+ \uu\cdot \nabla \phi -\Delta \phi+ \Psi'(\phi)+ \displaystyle{\rho'(\phi)\frac{|\uu|^2}{2}}=\overline{\Psi'(\phi)+ \rho'(\phi)\frac{|\uu|^2}{2}}, \quad  \text{a.e. in} \ \Omega \times (0,T).
\end{align*}

\item[(iii)] The pair $(\uu,\phi)$ fulfills the regularity $\uu \in C_w([0,T];\H_\sigma)$ and $\phi \in C_w([0,T];H^1(\Omega))$, for all $T>0$. The initial conditions are satisfied such that $\uu|_{t=0}=\uu_0$, $\phi|_{t=0}=\phi_0$ in $\Omega$. In addition, the boundary condition $ \partial_{\n}\phi=0$ holds almost everywhere on $\partial\Omega\times(0,T)$ for all $T>0$.

\item[(iv)] The energy inequality
\begin{align*}
E(\uu(t), \phi(t))+\int_0^t \int_{\Omega} \nu(\phi(\tau)) |D \uu(\tau)|^2 \, \d x \d\tau + \int_0^t \|(\partial_t \phi + \uu \cdot \nabla \phi)(\tau) \|_{\L2}^2 \, \d \tau \leq E(\uu_0, \phi_0)
\end{align*}
holds for all $t \geq 0$, where $E(\uu, \phi)$ is defined as in \eqref{EEE}.
\end{itemize}
\end{theorem}\smallskip

Next, we investigate the special case with matched densities (i.e., $\rho_1=\rho_2=1$, so that $\rho\equiv 1$). The resulting model is the mass-conserving Navier-Stokes-Allen-Cahn system for homogeneous fluids
\begin{equation}
 \label{NSAC}
\begin{cases}
\partial_t \uu + \uu \cdot \nabla \uu - \div (\nu(\phi)D\uu) + \nabla p= - \div(\nabla \phi \otimes \nabla \phi),\\
\div \uu=0,\\
\partial_t \phi +\uu\cdot \nabla \phi + \mu=  \overline{\mu}, \\
\mu= -\Delta \phi + \Psi' (\phi),
\end{cases}
\quad \text{ in } \Omega \times (0,T).
\end{equation}
The system \eqref{NSAC} is associated with the following boundary and initial conditions
\begin{equation}  \label{bic}
\begin{cases}
\uu=\mathbf{0},\quad  \partial_{\n} \phi =0 \qquad \qquad \quad \text{ on } \partial\Omega
\times (0,T), \\
\uu(\cdot, 0)= \uu_0, \quad \phi(\cdot, 0)=\phi_0 \quad\, \text{ in } \Omega.
\end{cases}
\end{equation}

We first state the existence of global weak solutions, whose proof follows from similar (indeed simpler)  {\it a priori} estimates as the ones obtained for the nonhomogeneous case in the proof of Theorem  \ref{weak-D}.

\begin{theorem}[Global weak solution]
\label{W-S}
Let $\Omega$ be a bounded domain in $\mathbb{R}^d$, $d=2,3$, with smooth boundary $\partial \Omega$. Assume that the initial datum $(\uu_0,\phi_0)$ satisfies
$\uu_0 \in \H_\sigma, \phi_0\in H^1(\Omega)\cap L^\infty(\Omega)$ with $
\| \phi_0\|_{L^\infty(\Omega)}\leq 1$ and $ |\overline{\phi}_0|<1$. Then  there exists a global weak solution $(\uu,\phi)$ to problem \eqref{NSAC}-\eqref{bic}. That is, the solution $(\uu,\phi)$ satisfies, for all $T>0$,
\begin{align*}
&\uu \in L^\infty(0,T;\H_\sigma)\cap L^2(0,T;\V_\sigma),\\
&\partial_t \uu \in L^2(0,T;\V'_\sigma) \ \ \text{if} \ d=2, \quad
 \partial_t \uu \in L^\frac43(0,T;\V'_\sigma) \ \ \text{if} \ d=3,\\
&\phi \in L^\infty(0,T; H^1(\Omega))\cap  L^2(0,T;H^2(\Omega)), \\
&\phi \in L^\infty(\Omega\times (0,T)) : |\phi(x,t)|<1 \ \ \text{a.e. in } \  \Omega\times(0,T),\\
&\partial_t \phi \in L^2(0,T;L^2(\Omega)) \ \ \text{if} \ d=2, \quad
 \partial_t \phi \in L^\frac43(0,T;L^2(\Omega)) \ \ \text{if} \ d=3,
\end{align*}
and
\begin{align*}
&\l \partial_t \uu, \vv\r_{\V'_\sigma,\V_\sigma} + (\uu \cdot \nabla \uu, \vv)+ (\nu(\phi)D\uu,\nabla \vv)
= (\nabla \phi \otimes \nabla \phi, \nabla \vv),  &&\forall \, \vv \in \V_\sigma, \ \text{and a.a.} \ t \in (0,T),\\
&\partial_t \phi+ \uu\cdot \nabla \phi -\Delta \phi+ \Psi'(\phi)=\overline{\Psi'(\phi)}, && \text{for a.e.} \ (x,t) \in \Omega \times (0,T).
\end{align*}
Moreover, the initial and boundary conditions and the energy inequality hold as in Theorem \ref{weak-D}.
\end{theorem}

Furthermore, due to the particular form of the density function (i.e., a positive constant), we are able to prove a conditional uniqueness result in two dimensions.

\begin{theorem}[Uniqueness of weak solutions in 2D]
\label{uni2d}
Let $d=2$ and  $(\uu_1,\phi_1)$ and $(\uu_2,\phi_2)$  be two weak solutions to problem \eqref{NSAC}-\eqref{bic} on $[0,T]$ subject to the same initial datum $(\uu_0, \phi_0)$ that satisfies the assumptions of Theorem \ref{W-S}. Moreover, assume  either one of the following conditions:
 \begin{itemize}
 \item[(1)] $\phi_1$ satisfies the additional regularity such that
 $$\phi_1\in L^\gamma(0,T;H^2(\Omega))\quad \text{with}\ \gamma>\frac{12}{5};$$
\item[(2)]  the viscosity function $\nu$ satisfies
\begin{equation}
\label{vis-add}
 \inf_{c\in [\nu_\ast,\, \nu^\ast]} \max_{s \in [-1,1]} \left|  \nu(s) - c \right| < \frac{\nu_\ast}{10 \, C_\Omega^{(2)}},
 \end{equation}
where $C_\Omega^{(2)}>0 $ is given in \eqref{p}.
\end{itemize}
Then,  $(\uu_1,\phi_1)=(\uu_2,\phi_2)$ on $[0,T]$.
\end{theorem}

\begin{remark}
The existence of strong solutions obtained in Section \ref{S-STRONG} (cf. Remark \ref{strong-hom}), which yields the regularity $\phi\in L^\gamma(0,T;H^2(\Omega))$ with $\gamma>\frac{12}{5}$, entails that
Theorem \ref{uni2d} can be regarded as a weak-strong uniqueness result for problem \eqref{NSAC}-\eqref{bic} in two dimensions. That is, the weak solution originating from an initial condition $(\uu_0,\phi_0)$ such that $\uu_0\in \V_\sigma$ and $\phi_0\in H^2(\Omega)$ with $\Psi'(\phi_0)\in \L2$ coincides with the (unique) strong solution departing from the same initial datum.
\end{remark}

\begin{remark}
When the viscosity function is the linear one as in \eqref{rhonu}. Then, choosing $c= \frac{\nu_1+\nu_2}{2}$, we infer that
$$
\inf_{c\in [\nu_\ast,\, \nu^\ast]} \max_{s \in [-1,1]} \left| \nu(s) - c \right| \leq \left| \frac{\nu_1-\nu_2}{2}\right|.
$$
Without loss of generality, assume that $\nu_1<\nu_2$. Then, the second condition in Theorem \ref{uni2d} is satisfied provided that
$$
\frac{\nu_2-\nu_1}{2} < \frac{\nu_1}{10 \, C_\Omega^{(2)}} \quad \Longleftrightarrow \quad \frac{\nu_2-\nu_1}{\nu_1} < \frac{1}{5 C_{\Omega}^{(2)}}.
$$
This implies that the relative difference of the viscosities is required to be small.
\end{remark}

\subsection{Proof of Theorem \ref{weak-D}}
First, we derive \textit{a priori} estimates of problem \eqref{NSAC-D}-\eqref{IC-D} that will be crucial to prove the existence of global weak solutions. \medskip

\textbf{Mass conservation and energy dissipation.}
Integrating the equation \eqref{NSAC-D}$_3$ over $\Omega$ and using the definition of $\xi$, we observe that
$$
\int_{\Omega} \phi (t) \, \d x= \int_{\Omega} \phi_0 \, \d x, \quad \forall \, t \geq 0.
$$
Below we derive the energy equation associated with \eqref{NSAC-D}.
First, multiplying  \eqref{NSAC-D}$_1$ by $\uu$ and integrating over $\Omega$, we obtain
\begin{equation}
\label{NSAC-D1}
\int_{\Omega} \frac12  \rho(\phi) \partial_t |\uu|^2 \, \d x+ \int_{\Omega} \rho(\phi) (\uu \cdot \nabla) \uu \cdot \uu \, \d x+ \int_{\Omega} \nu(\phi) |D \uu|^2 \, \d x= -\int_{\Omega} \Delta \phi \nabla \phi \cdot \uu \, \d x.
\end{equation}
Here, we have used the relation $-\Delta \phi \nabla \phi= \frac12 \nabla |\nabla \phi|^2 -{\rm div}(\nabla \phi \otimes \nabla \phi)$ and the incompressibility condition \eqref{NSAC-D}$_2$. Thanks to the no-slip boundary condition for $\uu$, we observe that
\begin{align*}
\int_{\Omega} \rho(\phi) (\uu \cdot \nabla) \uu \cdot \uu \, \d x
&=
\int_{\Omega} \rho(\phi) \uu \cdot \nabla \left( \frac12 |\uu|^2 \right) \, \d x = - \frac12 \int_{\Omega} \div ( \rho(\phi) \uu ) |\uu|^2 \, \d x\\
&= -  \int_{\Omega} \rho'(\phi)\nabla \phi \cdot \uu  \ \frac{|\uu|^2}{2} \, \d x.
\end{align*}
Next, we multiply \eqref{NSAC-D}$_3$ by $\partial_t \phi+ \uu \cdot \nabla \phi$ and integrate over $\Omega$. Noticing that $\overline{\partial_t \phi + \uu\cdot \nabla \phi}=0$, we get
\begin{equation}
\label{NSAC-D2}
\| \partial_t \phi+ \uu \cdot \nabla \phi \|_{\L2}^2+ \int_{\Omega} \mu \,  \big( \partial_t \phi+ \uu \cdot \nabla \phi\big) \, \d x + \int_{\Omega}  \rho'(\phi) \frac{|\uu|^2}{2}\big(\partial_t \phi+ \uu \cdot \nabla \phi\big)\, \d x=0.
\end{equation}
On the other hand, the following equalities hold
\begin{align*}
&\int_{\Omega} \mu \,   \partial_t \phi  \, \d x= \ddt  \int_{\Omega}  \frac12 |\nabla \phi|^2 + \Psi(\phi) \, \d x,
\\
&\int_{\Omega} \mu \,   \uu \cdot \nabla \phi  \, \d x=
\int_{\Omega} -\Delta \phi \nabla \phi \cdot \uu \, \d x
+ \int_{\Omega} \uu \cdot \nabla \Psi(\phi) \, \d x= \int_{\Omega} -\Delta \phi \nabla \phi \cdot \uu \, \d x,\\
&
\int_{\Omega} \rho'(\phi) \frac{|\uu|^2}{2} \partial_t \phi \, \d x=
 \int_{\Omega} \partial_t (\rho(\phi))  \frac{|\uu|^2}{2} \, \d x.
\end{align*}
Thus, by adding \eqref{NSAC-D1} and \eqref{NSAC-D2}, and using the above identities, we obtain the energy equation
\begin{align}
\ddt E(\uu, \phi)+\int_{\Omega} \nu(\phi) |D \uu|^2 \, \d x + \|\partial_t \phi + \uu \cdot \nabla \phi \|_{\L2}^2=0, \quad \forall\, t>0.\label{BEL-D}
\end{align}

\medskip
\textbf{Lower-order estimates.}  Integrating \eqref{BEL-D} with respect to time, it follows that
\begin{equation}
\label{E-bound}
E(\uu(t),\phi(t))+ \int_0^t \int_{\Omega} \nu(\phi) |D \uu|^2 \, \d x \d \tau + \int_0^t \|\partial_t \phi + \uu \cdot \nabla \phi \|_{\L2}^2 \, \d \tau\leq E(\uu_0, \phi_0), \quad \forall \, t \geq 0.
 \end{equation}
We shall use the \textit{a priori} $L^\infty$-estimate
\begin{align}
\|\phi(t)\|_{L^\infty(\Omega)}\leq 1, \quad \text{for a.a}\ \ t\in [0,T],
\label{L-inf}
\end{align}
which can be recovered from the singular nature of the potential $\Psi$ (recall \eqref{Log}). Since now $\rho$ and $\nu$ are strictly positive thanks to \eqref{L-inf}, then we deduce from \eqref{E-bound} that
\begin{equation}
\label{B1-D}
\uu \in L^\infty(0,T; \H_\sigma)\cap L^2(0,T;\V_\sigma), \quad \phi\in L^\infty(0,T;H^1(\Omega))
\end{equation}
and
\begin{equation}
\label{B2-D}
\partial_t \phi + \uu \cdot \nabla \phi \in L^2(0,T; L^2(\Omega)),
\end{equation}
for all $T>0$.
In light of \eqref{B1-D} and \eqref{B2-D}, when $d=2$, we have
$$
\left\| - \rho'(\phi) |\uu|^2 \right\|_{\L2}\leq C \| \uu\|_{L^4(\Omega)}^2 \leq C \| \nabla \uu\|_{\L2},
$$
which entails that $\rho'(\phi) |\uu|^2\in L^2(0,T;\L2)$.
Instead, when $d=3$, we have
$$
\left\| - \rho'(\phi) |\uu|^2 \right\|_{\L2}\leq C \| \uu\|_{L^4(\Omega)}^2 \leq C \| \nabla \uu\|_{\L2}^\frac32,
$$
thus
$\rho'(\phi) |\uu|^2\in L^\frac43(0,T;\L2)$.
Since $\overline{\rho'(\phi) |\uu|^2} \in L^\infty(0,T)$,
we also find that
\begin{equation}
\label{B3-D}
\mu - \overline{\mu}\in L^q(0,T;L^2(\Omega)),
\end{equation}
for $q=2$ if $d=2$, and $q=\frac{4}{3}$ if $d=3$.
Thanks to the boundary condition for $\phi$, we see that $\overline{\Delta \phi}=0$. Then, multiplying \eqref{NSAC-D}$_4$ by $-\Delta \phi$ and integrating by parts, we have
$$
\int_{\Omega} |\Delta \phi|^2 + F''(\phi) |\nabla \phi|^2 \, \d x= \theta_0
\|\nabla \phi\|_{\L2}^2 - \int_{\Omega} (\mu-\overline{\mu})\Delta \phi \, \d x,
$$
where $F$ is the convex part of the potential $\Psi$, i.e.,
$$F(s)=\frac{\theta}{2}\left[ (1+s)\ln(1+s)+(1-s)\ln(1-s)\right].$$
By \eqref{B1-D} and \eqref{B3-D}, we obtain
\begin{equation}
\label{H2-D}
\| \Delta \phi \|_{\L2} \leq C \left( 1+ \| \mu-\overline{\mu}\|_{\L2} \right).
\end{equation}
Then, from the regularity theory of the Neumann problem, we infer that
\begin{equation}
\label{B4-D}
\phi \in L^q(0,T;H^2(\Omega)), \quad \forall\, T>0.
\end{equation}
From \eqref{LADY}, \eqref{GN3} and the above bounds, we get
\begin{align*}
\| \uu \cdot \nabla \phi\|_{\L2}
&\leq C \| \uu \|_{L^4(\Omega)}
  \|\nabla \phi\|_{L^4(\Omega)}\notag\\
  &\leq C \| \uu \|_{\L2}^\frac12\| \nabla \uu\|_{\L2}^\frac12  \|\nabla \phi\|_{\L2}^\frac12 \| \phi\|_{H^2(\Omega)}^\frac12\notag\\
 & \leq C \| \nabla \uu\|_{\L2}^\frac12  \|  \phi\|_{H^2(\Omega)}^\frac12,\quad \text{if}\ d=2,
\end{align*}
and
\begin{align*}
\| \uu \cdot \nabla \phi\|_{\L2}
&\leq C \| \uu \|_{L^4(\Omega)}
  \|\nabla \phi\|_{L^4(\Omega)}\notag\\
  &\leq C \| \uu \|_{\L2}^\frac14\| \nabla \uu\|_{\L2}^\frac34  \|\phi\|_{L^\infty(\Omega)}^\frac12 \| \phi\|_{H^2(\Omega)}^\frac12\notag\\
 & \leq C \| \nabla \uu\|_{\L2}^\frac34  \|  \phi\|_{H^2(\Omega)}^\frac12,\quad \text{if}\ d=3,
\end{align*}
which implies that $\uu \cdot \nabla \phi\in L^q(0,T;L^2(\Omega))$ with the corresponding choice of the index $q$. Thus, it holds
\begin{equation}
\label{B5-D}
\partial_t \phi \in L^q(0,T;L^2(\Omega)).
\end{equation}
Moreover, we observe that
\begin{align}
\| \mu-\overline{\mu}\|_{\L2}&\leq \| \partial_t \phi\|_{\L2}+ \| \uu \cdot \nabla \phi\|_{\L2}
+ \left\|-\rho'(\phi)\frac{|\uu|^2}{2}\right\|_{\L2} + |\Omega|^{-\frac12} \left\|-\rho'(\phi)\frac{|\uu|^2}{2}\right\|_{L^1(\Omega)} \notag \\
&\leq \| \partial_t \phi\|_{\L2}+ C \| \uu \|_{L^4(\Omega)}
  \|\nabla \phi\|_{L^4(\Omega)}+C \| \uu\|_{L^4(\Omega)}^2+  C \| \uu\|_{\L2}^2 \notag \\
&\leq \| \partial_t \phi\|_{\L2}+ C \| \nabla \uu\|_{\L2}^\frac12  \| \phi\|_{H^2(\Omega)}^\frac12 +C \| \uu\|_{\L2}\| \nabla \uu\|_{\L2} +  C \| \uu\|_{\L2}^2 \notag \\
&\leq  \| \partial_t \phi\|_{\L2}+ C \| \nabla \uu\|_{\L2}^\frac12  \|  \phi\|_{H^2(\Omega)}^\frac12
+C \| \nabla \uu\|_{\L2} +  C,\quad \text{if}\ d=2,
\label{mu-L2-2}
\end{align}
and in a similar manner
\begin{align}
\| \mu-\overline{\mu}\|_{\L2}&\leq
\| \partial_t \phi\|_{\L2}+ C \| \nabla \uu\|_{\L2}^\frac34  \|  \phi\|_{H^2(\Omega)}^\frac12
+C \| \nabla \uu\|_{\L2}^\frac32 +  C,\quad \text{if}\ d=3.
\label{mu-L2-3}
\end{align}
Recalling \eqref{B1-D} and \eqref{H2-D}, and using Young's inequality, we find that
\begin{equation}
\label{estH2-D}
\|\phi \|_{H^2(\Omega)}\leq C \left( 1+ \| \partial_t \phi\|_{\L2}+  \| \nabla \uu\|_{\L2} \right), \quad \text{if}\ d=2,
\end{equation}
and
\begin{equation}
\label{estH3-D}
\|\phi \|_{H^2(\Omega)}\leq C\left(1+ \| \partial_t \phi\|_{\L2}+  \| \nabla \uu\|_{\L2}^\frac32\right), \quad \text{if}\ d=3.
\end{equation}
In order to recover the full $L^2$-norm of $\mu$, we observe that
$$
\overline{\mu}=\overline{F'(\phi)}- \theta_0 \overline{\phi}.
$$
Since $|\overline{\phi}(t)|=|\overline{\phi}_0|<1$, it is well-known that (see e.g. \cite{Miranville})
$$
\int_{\Omega} |F'(\phi)| \, \d x \leq C \int_{\Omega} F'(\phi) (\phi-\overline{\phi}) \, \d x+ C,
$$
for some positive constant $C$ depending on $F$ and $\overline{\phi}$.
Multiplying \eqref{NSAC-D}$_4$ by $\phi - \overline{\phi}$ and using the homogeneous Neumann boundary condition on $\phi$, we obtain
$$
\|\nabla \phi\|_{\L2}^2+ \int_{\Omega} F'(\phi) (\phi-\overline{\phi})\, \d x = \int_{\Omega} \mu  (\phi-\overline{\phi}) \, \d x+ \int_{\Omega} \theta_0 \phi (\phi-\overline{\phi}) \, \d x.
$$
Combining the above two relations and exploiting the energy bounds in \eqref{B1-D}, we arrive at
\begin{equation}
\label{mubar}
\| F'(\phi) \|_{L^1(\Omega)} \leq C (1+ \| \mu-\overline{\mu}\|_{\L2}).
\end{equation}
Therefore, we can infer from \eqref{B4-D}, \eqref{B5-D}, \eqref{mu-L2-3} and \eqref{mubar} that
\begin{equation}
\mu \in L^q(0,T;L^2(\Omega)),\quad \forall\, T>0,
\label{muq2}
\end{equation}
where $q=2$ if $d=2$, and $q=\frac{4}{3}$ if $d=3$.
Furthermore, the above fact and \eqref{B4-D} yield
\begin{equation}
F'(\phi)\in L^q(0,T;L^2(\Omega)).\label{fpq2}
\end{equation}

\smallskip

Next, we prove the following estimate for the time translations of $\uu$:

\begin{lemma}\label{est-tran}
For any $\delta\in(0,T)$, the following bound holds
\begin{align}
\int_0^{T-\delta}\|\uu(t+\delta)-\uu(t)\|_{\L2}^2 \, \d t\leq C\delta^\frac14.\label{est-tr}
\end{align}
\end{lemma}
\begin{proof}
We only present the proof for the case $d=3$.
The case $d=2$ follows along the same lines.
It follows from \eqref{B1-D} and the interpolation inequality \eqref{GN3} with $p=3$ that $\uu\in L^4(0,T;L^3(\Omega))$. Similar to \cite{Lions} (see also \cite[Lemma 3.5]{JLL2017}), we have
\begin{align*}
&\left\|\sqrt{\rho(\phi(t+\delta))}(\uu(t+\delta)-\uu(t))\right\|_{\L2}^2 \\
&\qquad  \leq -\int_\Omega(\rho(\phi(t+\delta))-\rho(\phi(t)))\uu(t)\cdot (\uu(t+\delta)-\uu(t)) \, \d x\\
&\qquad\quad  -\int_{t}^{t+\delta} \!\int_\Omega \rho(\phi(\tau))(\uu(\tau)\cdot\nabla)\uu(\tau)\cdot (\uu(t+\delta)-\uu(t))\, \d x\d \tau\\
&\qquad\quad  -\int_t^{t+\delta}\!\int_\Omega \nu(\phi(\tau))D\uu(\tau):D(\uu(t+\delta)-\uu(t))\, \d x\d \tau\\
&\qquad\quad  +\int_t^{t+\delta}\!\int_\Omega (\nabla \phi(\tau)\otimes\nabla \phi(\tau)): \nabla (\uu(t+\delta)-\uu(t)) \, \d x\d \tau\\
&\qquad\quad  +\int_t^{t+\delta}\!\int_\Omega \rho'(\phi)\partial_\tau \phi(\tau)\uu(\tau)\cdot (\uu(t+\delta)-\uu(t))\, \d x\d \tau\\
&\qquad  =: \sum_{i=1}^{5}J_i.
\end{align*}
First, we observe that
\begin{align}
\int_0^{T-\delta} J_1(t)\, \d t
&\leq \int_0^{T-\delta}\!\int_{t}^{t+\delta}\!\int_\Omega
|\rho'(\phi)||\partial_\tau\phi(\tau)||\uu(t)| \left( |\uu(t+\delta)|+|\uu(t)| \right)\, \d x\d\tau\d t\nonumber\\
&\leq \int_0^{T-\delta}(\|\uu(t+\delta)\|_{L^3(\Omega)}+\|\uu(t)\|_{L^3(\Omega)})
\|\uu(t)\|_{L^6(\Omega)}\int_t^{t+\delta}\| \partial_\tau\phi(\tau)\|_{L^2(\Omega)} \, \d \tau \d t\nonumber\\
&\leq C\delta^\frac14\left(\int_0^T \|\nabla \uu(t)\|_{L^2(\Omega)} \, \d t\right) \left(\int_0^{T}\| \partial_t\phi(t)\|_{L^2(\Omega)}^\frac43\d t\right)^\frac34 \leq C\delta^\frac14,\nonumber
\end{align}
and, in a similar manner,
\begin{align}
\int_0^{T-\delta} J_5(t) \, \d t
&\leq \int_0^{T-\delta}(\|\uu(t+\delta)\|_{L^3(\Omega)}+\|\uu(t)\|_{L^3(\Omega)})
\int_t^{t+\delta}\|\uu(\tau)\|_{L^6(\Omega)}\| \partial_\tau\phi(\tau)\|_{L^2(\Omega)} \, \d \tau \d t\nonumber\\
&\leq C\delta^\frac14\left(\int_0^T \|\nabla \uu(t)\|_{L^2(\Omega)}\, \d t\right) \left(\int_0^{T}\| \partial_t\phi(t)\|_{L^2(\Omega)}^\frac43\d t\right)^\frac34
\leq C\delta^\frac14.\nonumber
\end{align}
Next, we have
\begin{align}
&\int_0^{T-\delta} J_2(t) \, \d t\nonumber\\
&\quad \leq \int_0^{T-\delta}\!\!\int_{t}^{t+\delta} \| \rho(\phi(\tau))\|_{L^\infty(\Omega)}\|\uu(\tau)\|_{L^6(\Omega)}
\|\nabla\uu(\tau)\|_{L^2(\Omega)} \, \d \tau  \big(\|\uu(t+\delta)\|_{L^3(\Omega)}+\|\uu(t)\|_{L^3(\Omega)}\big)\, \d t\nonumber\\
&\quad \leq C\delta^\frac12\int_0^{T-\delta}\left(\int_t^{t+\delta}
\|\nabla\uu(\tau)\|_{L^2(\Omega)}^2 \, \d \tau\right)^\frac12\big(\|\uu(t+\delta)\|_{L^3(\Omega)}+\|\uu(t)\|_{L^3(\Omega)}\big) \, \d t\nonumber\\
&\quad \leq C\delta^\frac12\left(\int_0^T
\|\nabla\uu(t)\|_{L^2(\Omega)}^2 \, \d t\right)^\frac12\int_0^T\|\uu(t)\|_{L^3(\Omega)} \, \d t \leq C\delta^\frac12,\nonumber
\end{align}
and
\begin{align}
\int_0^{T-\delta} J_3(t) \, \d t\nonumber
& \leq
\int_0^{T-\delta} \int_t^{t+\delta} \|\nu(\phi(\tau))\|_{L^\infty(\Omega)}\|D\uu(\tau)\|_{L^2(\Omega)} \, \d \tau
\big(\|D\uu(t+\delta)\|_{\L2}+\|D\uu(t)\|_{\L2}\big)\, \d t\nonumber\\
& \leq C\delta^\frac12 \int_0^{T-\delta}
\left(\int_t^{t+\delta}
\|\nabla \uu(\tau)\|_{L^2(\Omega)}^2 \, \d \tau\right)^\frac12 \big(\|D\uu(t+\delta)\|_{\L2}+\|D\uu(t)\|_{\L2}\big) \, \d t\nonumber\\
& \leq C\delta^\frac12
\left(\int_0^{T}
\|\nabla \uu(t)\|_{L^2(\Omega)}^2 \, \d t\right)^\frac12 \int_0^{T} \|\nabla \uu(t)\|_{\L2}\, \d t  \leq C\delta^\frac12.\nonumber
\end{align}
Finally, by using \eqref{GN-L4} we get
\begin{align*}
\int_0^{T-\delta} J_4(t) \, \d t &\leq \int_0^{T-\delta}\! \int_t^{t+\delta} \|\nabla \phi(\tau)\|_{L^4(\Omega)}^2 \, \d \tau (\|\nabla \uu(t+\delta)\|_{\L2} +\|\nabla \uu(t)\|_{\L2}) \, \d t\nonumber\\
&\leq C\delta^\frac14 \int_0^{T-\delta} \left(\int_t^{t+\delta} \| \phi(\tau)\|_{H^2(\Omega)}^\frac43 \, \d \tau\right)^\frac34 (\|\nabla \uu(t+\delta)\|_{\L2} +\|\nabla \uu(t)\|_{\L2})\, \d t\nonumber\\
& \leq C\delta^\frac14\left(\int_0^{T} \| \phi(t)\|_{H^2(\Omega)}^\frac43 \, \d t\right)^\frac34 \int_0^T \|\nabla \uu(t)\|_{\L2} \, \d t \leq C\delta^\frac14.
\end{align*}
From the above estimates and the fact that $\rho$ is strictly bounded from below, we obtain the conclusion \eqref{est-tr}. The proof is complete.
\end{proof}
\medskip

\textbf{Existence of weak solutions.} With the above \textit{a priori} estimates, we are able to prove the existence of a global weak solution by using a semi-Galerkin scheme similar to that in \cite{JLL2017} (see also \cite{W17}).

More precisely, for any $n\in \mathbb{N}$, we find a local-in-time approximating solution $(\uu_n, \phi_n)$ where $\uu_n$ solves \eqref{NSAC-D}$_1$ as in the classical Faedo-Galerkin approximation (projected on the $n$-dimensional space spanned by the eigenfunctions of the Stokes operator) and $\phi_n$ is the (non-discrete) solution to the equations \eqref{NSAC-D}$_3$-\eqref{NSAC-D}$_4$ with the velocity $\uu_n$, the singular potential and the nonlocal term. This goal can be achieved by a Schauder fixed point argument. To complete the proof, it is necessary to solve separately a convective nonlocal Allen-Cahn equation with a given velocity field $\vv_n$ (in the same finite dimensional space as $\uu_n$). To this aim, we introduce a family of regular potentials
$\lbrace \Psi_\varepsilon \rbrace$ that approximates the original singular potential
$\Psi$ by setting (see, e.g., \cite{GGW2018})
\begin{equation}
\Psi_\varepsilon(s)=F_\varepsilon(s)-\frac{\theta_0}{2}s^2,\quad \forall\, s\in \mathbb{R},\nonumber
\end{equation}
where
\begin{equation}
F_\varepsilon(s)=
\begin{cases}
 \displaystyle{\sum_{j=0}^2 \frac{1}{j!}}
 F^{(j)}(1-\varepsilon) \left[s-(1-\varepsilon)\right]^j,
 \qquad\qquad \!\!  \forall\,s\geq 1-\varepsilon,\\
 F(s), \qquad \qquad \qquad \qquad \qquad \qquad \qquad \quad
 \forall\, s\in[-1+\varepsilon, 1-\varepsilon],\\
 \displaystyle{\sum_{j=0}^2
 \frac{1}{j!}} F^{(j)}(-1+\varepsilon)\left[ s-(-1+\varepsilon)\right]^j,
 \qquad\ \forall\, s\leq -1+\varepsilon.
 \end{cases}
 \nonumber
\end{equation}
Substituting the regular potential $\Psi_\varepsilon$ into the original Allen-Cahn equation \eqref{NSAC-D}$_3$-\eqref{NSAC-D}$_4$, we are able to prove the existence of an approximating solution $\phi_\varepsilon$ to the resulting regularized equation using e.g., the semigroup approach like in \cite[Lemma 3.2]{JLL2017} or simply by the Galerkin method. For the approximating solution $\phi_{\varepsilon}$, we can derive estimates that are uniform in $\varepsilon$ and then pass to the limit as $\varepsilon\to 0$ to recover the case with singular potential. It is worth mentioning that in this procedure, thanks to the singular potential $\Psi$, we can show the limit function satisfies
$$
\widetilde{\phi} \in L^\infty(\Omega\times (0,T))\quad \text{and}\quad |\widetilde{\phi}(x,t)|<1\ \ \text{a.e. in}\ \Omega\times(0,T),
$$
using a similar argument like in \cite{GGW2018,GMT2019}, without the additional assumption $s\rho'(s)\geq 0$ for $|s|>1$ that was required in \cite{JLL2017}. This together with the spatial regularity of $\widetilde{\phi}$ (e.g., from the energy estimates) implies that  $\|\widetilde{\phi}(t)\|_{L^\infty(\Omega)}\leq 1$ for almost any  $t\in[0,T]$ (cf. \eqref{L-inf}). Inserting the function $\widetilde{\phi}=\widetilde{\phi}[\vv_n]$ back into the Galerkin approximation equation for $\uu_n$, we can thus define a mapping $\uu_n=\mathcal{T}[\vv_n]$ in a suitable finite dimensional space. Then, by means of the classical Schauder's argument, it is possible to prove that $\mathcal{T}$ admits a fixed point defined on a certain time interval. This  gives the local approximating solution $(\uu_n, \phi_n)$.

Next, thanks to the {\it a priori} estimates showed above, it follows that the existence time interval of any approximating solution $(\uu_n, \phi_n)$ is independent of the parameter $n$. From the same argument, we deduce uniform estimates that allows sufficient compactness for the phase function $\phi_n$. Then, the key issue is to obtain uniform estimates of translations $\int_0^{T-\delta} \|\uu_n(t+\delta)-\uu_n(t)\|_{L^2(\Omega)}^2\, \d t$ (see Lemma \ref{est-tran}) that yields compactness of the velocity field in the case of unmatched densities (cf. \cite{Lions}). Since the above two-level approximating procedure is standard, we omit the details here.
\medskip

\textbf{Time continuity and the initial condition.} We first observe that the regularity properties \eqref{B1-D} and \eqref{B5-D}, together with the global bound $\| \phi\|_{L^\infty(0,T;L^\infty(\Omega))}\leq 1$, entail that
$$
\phi \in C([0,T]; L^p(\Omega)), \quad \forall\, 2\leq p<\infty\ \ \text{if}\ d=2,3.
$$
In addition, since $\phi \in L^\infty(0,T;H^1(\Omega))$, we also infer from \cite[Theorem 2.1]{STRAUSS} that $\phi \in C_w([0,T];H^1(\Omega))$. If $d=2$, since $\phi \in L^2(0,T;H^2(\Omega))\cap W^{1,2}(0,T;L^2(\Omega))$, we can further deduce that $\phi \in C([0,T];H^1(\Omega))$.

Next, the weak formulation of \eqref{NSAC-D}$_1$-\eqref{NSAC-D}$_2$ can be written as
\begin{align*}
\ddt \l \mathbb{P}(\rho(\phi)\uu), \vv \r_{\V_\sigma',\V_\sigma}=\l\, \widetilde{\f}, \vv\r_{\V_\sigma',\V_\sigma},
\end{align*}
for all $\vv \in \V_\sigma$, in the sense of distribution on $(0,T)$, where $\mathbb{P}$ is the Leray projection onto $\H_\sigma$ and
$$
\l\, \widetilde{\f}, \vv\r_{\V_\sigma',\V_\sigma} = (\rho'(\phi)\partial_t \phi \, \uu, \vv)-(\rho(\phi)(\uu \cdot\nabla) \uu, \vv)-
(\nu(\phi) D \uu, \nabla \vv)+ (\nabla \phi \otimes \nabla \phi, \nabla \vv).
$$
Arguing similarly to the proof of Lemma \ref{est-tran}, we observe that
\begin{align*}
\|\, \widetilde{\f}\|_{\V_\sigma'}
&\leq C \| \partial_t \phi\|_{L^2(\Omega)} \| \uu\|_{L^3(\Omega)}+ C \| \uu\|_{L^3(\Omega)} \| D \uu\|_{\L2}+ C \| D \uu\|_{\L2}+ C \| \nabla \phi\|_{L^4(\Omega)}^2\\
&\leq C \| \partial_t \phi\|_{L^2(\Omega)} \| D \uu\|_{\L2}^\frac12 + C \| D\uu\|_{\L2}^{\frac32}+ \| D \uu\|_{\L2} + C \| \phi\|_{H^2(\Omega)}\\
&\leq C \left( 1+ \| \partial_t \phi\|_{L^2(\Omega)}^\frac43 + \| D \uu\|_{\L2}^2+C \| \phi\|_{H^2(\Omega)}^\frac43 \right).
\end{align*}
In light of the regularity of the weak solution, we find that $\widetilde{\f} \in L^1(0,T;\V_\sigma')$. By definition of the weak time derivative, this implies that $$
\partial_t \mathbb{P}(\rho(\phi)\uu)\in L^1(0,T;\V_\sigma').$$
Observing that $\mathbb{P} (\rho(\phi)\uu) \in L^\infty(0,T;\H_\sigma)$, we have  $\mathbb{P} (\rho(\phi)\uu) \in C([0,T];\V_\sigma')$. As a consequence, we deduce from \cite[Theorem 2.1]{STRAUSS} that $\mathbb{P} (\rho(\phi)\uu) \in C_w([0,T]; \H_\sigma)$. Besides, it easily follows from the properties of the Leray operator $\mathbb{P}$ that $\mathbb{P} (\rho(\phi)\uu) \in C_w([0,T]; \mathbf{L}^2(\Omega))$. Now, repeating the argument in \cite[Section 5.2]{ADG2013}, we deduce that $\rho(\phi)\uu \in C_w([0,T]; \mathbf{L}^2(\Omega))$. Therefore, since $\rho(\phi) \in C([0,T];L^2(\Omega))$ and $\rho(\phi)\geq \rho_\ast>0$, we can conclude that $\uu \in C_w([0,T];\mathbf{L}^2(\Omega))$. Thanks to the time continuity of $\uu$ and $\phi$, a standard argument ensures that the initial conditions $\uu|_{t=0}=\uu_0$, $\phi|_{t=0}=\phi_0$ are satisfied in $\Omega$.

Proof of Theorem \ref{weak-D} is complete.
\hfill$\square$
\smallskip

\subsection{Proof of Theorem \ref{uni2d}}
Let us consider two global weak solutions $(\uu_1,\phi_1)$ and $(\uu_2,\phi_2)$ to problem \eqref{NSAC}-\eqref{bic} given by Theorem \ref{W-S}. Denote the differences of solutions by $$\uu=\uu_1-\uu_2,\quad \phi=\phi_1-\phi_2.$$
Then we have
\begin{align}
& \l \partial_t \uu, \vv\r_{\V_\sigma',\V_\sigma} + (\uu_1 \cdot \nabla \uu, \vv) + (\uu \cdot \nabla \uu_2, \vv)+ (\nu(\phi_1)D\uu,\nabla \vv) + ((\nu(\phi_1)-\nu(\phi_2))D\uu_2,\nabla \vv) \notag  \\
&\quad = (\nabla \phi_1 \otimes \nabla \phi, \nabla \vv)+
(\nabla \phi \otimes \nabla \phi_2, \nabla \vv),
 \label{NS-diff}
\end{align}
for all $ \vv \in \V_\sigma$, almost every $t \in (0,T)$, and
\begin{equation}
\partial_t \phi+ \uu_1\cdot \nabla \phi + \uu \cdot \nabla \phi_2-\Delta \phi+ \Psi'(\phi_1)-\Psi'(\phi_2)=\overline{\Psi'(\phi_1)}-\overline{\Psi'(\phi_2)}  \label{AC-diff}
\end{equation}
for almost every $(x,t) \in \Omega \times (0,T)$.

Following the same strategy as in \cite{GMT2019}, we take $\vv= \A^{-1}\uu$, where $\A$ is the Stokes operator, and we find that
\begin{align*}
&\frac12 \ddt \| \uu\|_{\ast}^2 + (\nu(\phi_1) D\uu, \nabla \A^{-1}\uu)\\
&\quad  =
(\uu\otimes \uu_1, \nabla \A^{-1}\uu)+ (\uu_2\otimes \uu, \nabla \A^{-1}\uu)\\
&\qquad - ((\nu(\phi_1)-\nu(\phi_2))D\uu_2,\nabla \A^{-1}\uu)
 + (\nabla \phi_1 \otimes \nabla \phi, \nabla \A^{-1}\uu)\\
 &\qquad
+(\nabla \phi \otimes \nabla \phi_2, \nabla \A^{-1}\uu),
\end{align*}
where $\| \uu\|_{\ast}= \| \nabla \A^{-1} \uu\|_{\L2}$, which is a norm on $\V'_\sigma$ equivalent to the usual\ dual norm. Here, we have used the equality $\uu_i \cdot \nabla \uu= \div ( \uu\otimes \uu_i)$, $i=1,2$,
due to the incompressibility condition.
Beside, multiplying \eqref{AC-diff} by $\phi$, integrating over $\Omega$ and observing that
\begin{align*}
& \int_{\Omega} (\uu_1 \cdot \nabla \phi) \, \phi \,\d x= \int_{\Omega} \uu_1 \cdot \left(\frac12 \nabla \phi^2 \right) \,\d x=0, \\
& \int_{\Omega} (\overline{\Psi'(\phi_1)}-\overline{\Psi'(\phi_2)} ) \phi \, \d x= (\overline{\Psi'(\phi_1)}-\overline{\Psi'(\phi_2)} ) \overline{\phi}|\Omega| =0,
\end{align*}
we obtain
$$
\frac12 \ddt \|\phi\|_{\L2}^2+ \| \nabla \phi\|_{\L2}^2+ \int_{\Omega} (\uu \cdot \nabla \phi_2) \, \phi \, \d x+ \int_{\Omega} (F'(\phi_1)-F'(\phi_2)) \, \phi \, \d x=
\theta_0 \| \phi\|_{\L2}^2.
$$
By adding the above two equations and using the convexity of $F$, we deduce that
\begin{align}
& \ddt G(t) + (\nu(\phi_1) D\uu, \nabla \A^{-1}\uu) +  \| \nabla \phi\|_{\L2}^2  \notag \\
&\quad \leq (\uu\otimes \uu_1, \nabla \A^{-1}\uu)+ (\uu_2\otimes \uu, \nabla \A^{-1}\uu)
- ((\nu(\phi_1)-\nu(\phi_2))D\uu_2,\nabla \A^{-1}\uu) \notag \\
 &\qquad + (\nabla \phi_1 \otimes \nabla \phi, \nabla \A^{-1}\uu)
+(\nabla \phi \otimes \nabla \phi_2, \nabla \A^{-1}\uu)- (\uu \cdot \nabla \phi_2 , \phi) + \theta_0 \| \phi\|_{\L2}^2,
\label{u-est1}
\end{align}
where
$$
G(t)= \frac12 \| \uu(t)\|_{\ast}^2+ \frac12 \|\phi(t)\|_{\L2}^2.
$$

In order to recover an $L^2(\Omega)$-norm of $\uu$, which is a key term to control the nonlinear terms on the right-hand side, we obtain by integration by parts such that (see \cite[(3.9)]{GMT2019})
\begin{align*}
 (\nu(\phi_1) D\uu, \nabla \A^{-1}\uu)&= (\nabla \uu, \nu(\phi_1)D \A^{-1}\uu)\\
 &=- (\uu, \div  (\nu(\phi_1)D \A^{-1}\uu )\\
 &=- (\uu, \nu'(\phi_1) D \A^{-1}\uu \nabla \phi_1) - \frac12 (\uu, \nu(\phi_1) \Delta \A^{-1}\uu).
\end{align*}
Here we have used the fact that $\div \nabla^t \vv= \nabla \div \vv$.
Notice that, by the definition of the Stokes operator,
there exists a scalar function $p\in L^\infty(0,T;H^1(\Omega))\cap L^2(0,T;H^2(\Omega))$ (unique up to a constant) such that
$$
-\Delta \A^{-1} \uu+ \nabla p= \uu\qquad \text{for almost every}\ (x,t)\in \Omega \times (0,T).
$$
Moreover, we have the following estimates from \cite{Galdi} and \cite[Appendix B]{GMT2019}
\begin{equation}
\label{p}
\begin{cases}
\| p\|_{\L2}\leq  C_\Omega^{(1)}    \| \nabla \A^{-1} \uu \|_{\L2}^{\frac12} \| \uu\|_{\L2}^{\frac12}, \smallskip \\
\| p \|_{H^1(\Omega)}\leq C_\Omega^{(2)} \|\uu \|_{\L2}, \smallskip \\
\| p\|_{H^2(\Omega)}\leq C_\Omega^{(3)}  \| \nabla \uu\|_{\L2},
\end{cases}
\end{equation}
 where the positive  constants $C_\Omega^{(1)}, C_\Omega^{(2)}$ and $C_\Omega^{(3)}$ only depend on $\Omega$. Then we can write
\begin{align*}
- \frac12 (\uu, \nu(\phi_1) \Delta A^{-1}\uu)&= \frac12 ( \uu, \nu(\phi_1) \uu ) -\frac12 (\uu, \nu(\phi_1) \nabla p)\\
&=  \frac12 ( \uu, \nu(\phi_1) \uu ) + \frac12 (\div (\nu(\phi_1)\uu ), p)\\
&= \frac12 ( \uu, \nu(\phi_1) \uu ) + \frac12 (\nu'(\phi_1) \nabla \phi_1 \cdot \uu, p).
\end{align*}
Recalling that $\nu(\cdot)\geq \nu_\ast>0$, we then get
\begin{align*}
(\nu(\phi_1) D\uu, \nabla \A^{-1}\uu) \geq \frac{\nu_\ast}{2} \| \uu\|_{\L2}^2+
 \frac12 (\nu'(\phi_1) \nabla \phi_1 \cdot \uu, p)- (\uu, \nu'(\phi_1) D \A^{-1}\uu \nabla \phi_1).
\end{align*}
Owing to the above estimate, we rewrite the differential inequality \eqref{u-est1} as follows
\begin{align}
&\ddt G(t) + \frac{\nu_\ast}{2} \| \uu\|_{\L2}^2 +  \| \nabla \phi\|_{\L2}^2 \notag \\
&\quad   \leq \, (\uu\otimes \uu_1, \nabla \A^{-1}\uu)+ (\uu_2\otimes \uu, \nabla \A^{-1}\uu)
+ ((\nu(\phi_1)-\nu(\phi_2))D\uu_2,\nabla \A^{-1}\uu) \notag \\
&\qquad + (\nabla \phi_1 \otimes \nabla \phi, \nabla \A^{-1}\uu)
+(\nabla \phi \otimes \nabla \phi_2, \nabla \A^{-1}\uu)- (\uu \cdot \nabla \phi_2 , \phi)  \notag \\
&\qquad + \theta_0 \| \phi\|_{\L2}^2
+ (\uu, \nu'(\phi_1) D \A^{-1}\uu \nabla \phi_1)-\frac12 (\nu'(\phi_1) \nabla \phi_1 \cdot \uu, p).
\label{u-est2}
\end{align}
By the Ladyzhenskaya inequality \eqref{LADY}, together with \eqref{H2equiv} and the bounds for weak solutions, we have
\begin{align*}
& (\uu\otimes \uu_1, \nabla \A^{-1}\uu)+ (\uu_2\otimes \uu, \nabla \A^{-1}\uu)\\
&\quad   \leq \| \uu\|_{\L2} \left( \| \uu_1\|_{L^4(\Omega)}+  \| \uu_2\|_{L^4(\Omega)}\right)
\| \nabla \A^{-1}\uu\|_{L^4(\Omega)}\\
&\quad  \leq C \left(\| \uu_1\|_{H^1(\Omega)}^\frac12+ \| \uu_2\|_{H^1(\Omega)}^\frac12 \right) \| \uu\|_{\L2}^\frac32 \|\uu \|_{\ast}^\frac12\\
&\quad   \leq \frac{\nu_\ast}{20} \| \uu\|_{\L2}^2 + C \left(\| \uu_1\|_{H^1(\Omega)}^2+ \| \uu_2\|_{H^1(\Omega)}^2 \right) \|\uu \|_{\ast}^2.
\end{align*}
In a similar manner, we obtain
\begin{align*}
&(\nabla \phi_1 \otimes \nabla \phi, \nabla \A^{-1}\uu)
+(\nabla \phi \otimes \nabla \phi_2, \nabla \A^{-1}\uu)\\
&\quad \leq \left( \| \nabla \phi_1\|_{L^4(\Omega)}+  \| \nabla \phi_2\|_{L^4(\Omega)}\right)
\| \nabla \phi\|_{\L2}
\| \nabla \A^{-1}\uu\|_{L^4(\Omega)}\\
&\quad \leq C \left(\| \phi_1\|_{H^2(\Omega)}^\frac12+ \| \phi_2\|_{H^2(\Omega)}^\frac12 \right) \| \nabla \phi\|_{\L2} \|\uu \|_{\L2}^\frac12 \|\uu \|_{\ast}^\frac12\\
&\quad \leq \frac{\nu_\ast}{20} \| \uu\|_{\L2}^2 + \frac{1}{12} \| \nabla \phi\|_{\L2}^2+ C \big(\| \phi_1\|_{H^2(\Omega)}^2+ \| \phi_2\|_{H^2(\Omega)}^2) \|\uu \|_{\ast}^2,
\end{align*}
as well as
\begin{align*}
(\uu \cdot \nabla \phi_2, \phi)
&\leq \|\uu \|_{\L2} \| \nabla \phi_2 \|_{L^4(\Omega)} \|\phi \|_{L^4(\Omega)}\\
&\leq C \|\uu \|_{\L2} \| \nabla \phi_2 \|_{H^1(\Omega)}^\frac12 \|\phi \|_{\L2}^\frac12 \|\nabla \phi \|_{\L2}^\frac12\\
&\leq \frac{\nu_\ast}{20} \| \uu\|_{\L2}^2+ \frac{1}{12} \| \nabla \phi\|_{\L2}^2
+C \|\phi_2 \|_{H^2(\Omega)}^2 \| \phi\|_{\L2}^2,
\end{align*}
where we have also used the inequality \eqref{normH1-2} and the conservation of mass such that $\overline{\phi}=0$. Since $\nu'$ is bounded, by exploiting \eqref{LADY} we deduce that
\begin{align*}
 (\uu, \nu'(\phi_1) D \A^{-1}\uu \nabla \phi_1)
 &\leq C \| \uu\|_{\L2} \| D \A^{-1}\uu\|_{L^4(\Omega)} \| \nabla \phi_1\|_{L^4(\Omega)}\\
 &\leq C \| \uu\|_{\L2}^\frac32 \| \uu\|_{\ast}^\frac12 \| \nabla \phi_1\|_{H^1(\Omega)}^\frac12\\
 &\leq \frac{\nu_\ast}{20} \| \uu\|_{\L2}^2+ C \| \phi_1\|_{H^2(\Omega)}^2\| \uu\|_{\ast}^2.
\end{align*}
Next, by using Lemma \ref{result1} together with \eqref{normH1-2}, we infer that
\begin{align*}
-((\nu(\phi_1)-\nu(\phi_2))D\uu_2,\nabla \A^{-1}\uu)
&  =-\int_{\Omega} \left(\int_0^1 \nu'(\tau \phi_1+ (1-\tau)\phi_2) \, \d \tau\right) \phi D \uu_2 : \nabla \A^{-1}\uu \, \d x\\
&  \leq C \| D\uu_2\|_{\L2} \| \phi \nabla \A^{-1}\uu\|_{\L2}\\
&  \leq C \| \uu_2 \|_{H^1(\Omega)} \| \nabla \phi\|_{\L2} \|\A^{-1}\uu\|_{\V_\sigma}
\ln^\frac12 \left(C\frac{ \| \A^{-1}\uu\|_{W^{1,p}(\Omega)}}{\|\A^{-1}\uu\|_{\V_\sigma}}\right)\\
&  \leq C \| \uu_2 \|_{H^1(\Omega)} \| \nabla \phi\|_{\L2} \| \uu\|_{\ast}
\ln^\frac12 \left(C\frac{ \| \uu\|_{\L2}}{\| \uu\|_{\ast}} \right)\\
&  \leq \frac{1}{12} \|\nabla \phi \|_{\L2}^2
 +C \| \uu_2 \|_{H^1(\Omega)}^2 \| \uu\|_{\ast}^2
 \ln \left(\frac{ \widetilde{C}}{\| \uu\|_{\ast}} \right),
\end{align*}
where $\widetilde{C}>0$ is chosen sufficiently large such that
$\ln \Big( \frac{\widetilde{C} }{\|  \uu\|_{\ast}}\Big)>1$. \medskip

To finish the proof of Theorem \ref{uni2d}, we distinguish two cases.

\textbf{Case 1.} By using the Stokes operator (i.e., $\A=\mathbb{P}(-\Delta)$) and the integration by parts, we infer that
\begin{align*}
-\frac12 (\nu'(\phi_1) \nabla \phi_1 \cdot \uu, p)
&= \frac12 \big( \Delta \A^{-1} \uu , \mathbb{P} (  \nu'(\phi_1)\nabla \phi_1 p ) \big)\\
&=-\frac12 \int_{\Omega} (\nabla \A^{-1} \uu)^t : \nabla \mathbb{P} (  \nu'(\phi_1)\nabla \phi_1 p ) \, \d x \\
&\quad  +\frac12  \int_{\partial \Omega} \big( (\nabla \A^{-1} \uu)^t \mathbb{P}(  \nu'(\phi_1)\nabla \phi_1 p ) \big) \cdot \n \, \d S.
\end{align*}
Thanks to \eqref{H2equiv}, \eqref{trace}, and the properties of the Leray projection, we find
 \begin{align}
-\frac12 (\nu'(\phi_1) \nabla \phi_1 \cdot \uu, p)
& \leq C \| \nabla \A^{-1}\uu\|_{\L2} \| \nabla  \mathbb{P} (  \nu'(\phi_1)\nabla \phi_1 p )  \|_{\L2} \notag\\
&\quad
+ C \| \nabla \A^{-1}\uu\|_{L^2(\partial \Omega)} \| \mathbb{P}(  \nu'(\phi_1)\nabla \phi_1 p )\|_{L^2(\partial \Omega)}   \notag \\
& \leq C  \| \uu\|_{\ast} \| \nu'(\phi_1)\nabla \phi_1 p \|_{H^1(\Omega)} \notag\\
&\quad  +C \|\uu\|_{\ast}^\frac12 \| \uu\|_{\L2}^\frac12
\| \mathbb{P}(  \nu'(\phi_1)\nabla \phi_1 p )\|_{\L2}^\frac12
\| \mathbb{P}(  \nu'(\phi_1)\nabla \phi_1 p )\|_{H^1(\Omega)}^\frac12  \notag \\
& \leq C  \| \uu\|_{\ast} \| \nu'(\phi_1)\nabla \phi_1 p \|_{H^1(\Omega)} \notag\\
&\quad  +C \| \uu\|_{\ast}^\frac12 \| \uu\|_{\L2}^\frac12
\|  \nu'(\phi_1)\nabla \phi_1 p \|_{\L2}^\frac12
\|  \nu'(\phi_1)\nabla \phi_1 p \|_{H^1(\Omega)}^\frac12.
\label{pterm}
\end{align}
Owing to \eqref{LADY}, \eqref{BGI}, Lemma \ref{result1} and \eqref{p}, we observe that
\begin{align*}
\|  \nu'(\phi_1)\nabla \phi_1 p \|_{\L2}
&\leq C \|\nabla \phi_1 \|_{L^4(\Omega)} \| p\|_{L^4(\Omega)}\\
&\leq C \| \phi_1\|_{H^2(\Omega)}^\frac12 \| p\|_{\L2}^\frac12 \| p\|_{H^1(\Omega)}^\frac12 \\
&\leq C \| \phi_1\|_{H^2(\Omega)}^\frac12 \| \nabla \A^{-1} \uu\|_{\L2}^\frac14 \| \uu\|_{\L2}^\frac34,
\end{align*}
and
\begin{align*}
\|  \nu'(\phi_1)\nabla \phi_1 p \|_{H^1(\Omega)}
&\leq \|  \nu'(\phi_1)\nabla \phi_1 p \|_{L^2(\Omega)}
+\|  \nu''(\phi_1)\nabla \phi_1 \otimes \nabla \phi_1 p \|_{L^2(\Omega)}  \\
&\quad + \| \nu'(\phi_1) \nabla^2 \phi_1 p\|_{L^2(\Omega)}
+ \| \nu'(\phi_1) \nabla \phi_1 \otimes \nabla p\|_{L^2(\Omega)}\\
&\leq C \| \uu\|_{L^2(\Omega)} \ln^\frac12 \left( C \frac{\|\nabla  \uu\|_{L^2(\Omega)} }{\| \uu\|_{L^2(\Omega)}} \right) + C \| \nabla \phi_1\|_{L^4(\Omega)}^2 \| p\|_{L^\infty(\Omega)}\\
&\quad + C \| \phi_1\|_{H^2(\Omega)} \| p\|_{L^\infty(\Omega)} +
C \| \phi_1\|_{H^2(\Omega)} \| p\|_{H^1(\Omega)} \ln^\frac12 \left( C \frac{\| p\|_{H^2(\Omega)}}{\|p \|_{H^1(\Omega)}}\right)\\
&\leq C \big( 1+ \| \phi_1\|_{H^2(\Omega)}\big) \| \uu\|_{L^2(\Omega)} \ln^\frac12 \left( C \frac{\|\nabla  \uu\|_{L^2(\Omega)} }{\| \uu\|_{L^2(\Omega)}}\right).
\end{align*}
Combining the above estimates with \eqref{pterm}, we are led to
\begin{align*}
-\frac12 (\nu'(\phi_1) \nabla \phi_1 \cdot \uu, p)
& \leq C \big( 1+ \| \phi_1\|_{H^2(\Omega)}\big) \| \uu\|_{\ast}  \| \uu\|_{L^2(\Omega)} \ln^\frac12 \left( C \frac{\|\nabla  \uu\|_{L^2(\Omega)} }{\| \uu\|_{L^2(\Omega)}}\right)\\
&\quad + C \left( 1+ \| \phi_1\|_{H^2(\Omega)}^\frac34 \right) \| \uu\|_{\ast}^\frac58 \| \uu\|_{L^2(\Omega)}^\frac{11}{8} \ln^\frac14 \left( C \frac{\|\nabla  \uu\|_{L^2(\Omega)} }{\| \uu\|_{L^2(\Omega)}}\right)\\
&\leq \frac{\nu_\ast}{20}  \| \uu\|_{\L2}^2+
C \big( 1+ \| \phi_1\|_{H^2(\Omega)}^2 \big) \| \uu\|_{\ast}^2  \ln \left( C \frac{\|\nabla  \uu\|_{L^2(\Omega)} }{\| \uu\|_{L^2(\Omega)}}\right)\\
&\quad +C \left( 1+ \| \phi_1\|_{H^2(\Omega)}^\frac{12}{5} \right) \|  \uu\|_{\ast}^2 \ln^\frac45 \left( C \frac{\|\nabla  \uu\|_{L^2(\Omega)} }{\| \uu\|_{L^2(\Omega)}}\right).
\end{align*}
In order to handle the logarithmic terms, we recall that $\frac{ C \|\nabla  \uu\|_{L^2(\Omega)} }{\| \uu\|_{L^2(\Omega)}}>1$. Since $\frac{C' \| \uu\|_{\L2}}{\|\uu\|_{\ast}}>1$, for some $C'>0$ depending on $\Omega$, we have
\begin{align*}
 \ln^\frac45 \left( C \frac{\|\nabla  \uu\|_{L^2(\Omega)} }{\| \uu\|_{L^2(\Omega)}}\right)
& \leq 1+  \ln  \left( C \frac{\|\nabla  \uu\|_{L^2(\Omega)} }{\| \uu\|_{L^2(\Omega)}}\right)\\
& \leq 1+  \ln  \left( C \frac{ C' \|\nabla  \uu\|_{L^2(\Omega)} }{\|  \uu\|_{\ast}}\right)\\
 &\leq C+ \ln \left(1+\| \nabla \uu\|_{\L2} \right) + \ln \left( \frac{\widetilde{C} }{\|  \uu\|_{\ast}}\right),
\end{align*}
where $\widetilde{C}>0$ is a sufficiently large constant such that
$\ln \Big( \frac{\widetilde{C} }{\|  \uu\|_{\ast}}\Big)>1$, which holds true in light of \eqref{B1-D}. Thus, we obtain that
\begin{align*}
-\frac12 (\nu'(\phi_1) \nabla \phi_1 \cdot \uu, p)
& \leq \frac{\nu_\ast}{20}  \| \uu\|_{\L2}^2+
C \left( 1+ \| \phi_1\|_{H^2(\Omega)}^\frac{12}{5} \right)
 \ln \left(1+\| \nabla \uu\|_{\L2} \right)
\| \uu\|_{\ast}^2\\
&\quad +
C \left( 1+ \| \phi_1\|_{H^2(\Omega)}^\frac{12}{5} \right) \|\uu\|_{\ast}^2 \ln \left(  \frac{\widetilde{C}}{\| \uu\|_{\ast}} \right).
\end{align*}
%
Summing up, we arrive at the differential inequality
\begin{equation}
\label{u-est3}
\ddt G(t) + \frac{\nu_\ast}{4} \| \uu\|_{\L2}^2 +  \frac12 \| \nabla \phi\|_{\L2}^2
\leq C S(t)  G(t) \ln \left( \frac{\widetilde{C}}{G(t)} \right),
\end{equation}
where
\begin{align*}
S(t) &= 1+ \| \uu_1(t)\|_{H^1(\Omega)}^2+ \| \uu_2(t)\|_{H^1(\Omega)}^2
+ \|\phi_1 (t)\|_{H^2(\Omega)}^2 +\| \phi_2(t)\|_{H^2(\Omega)}^2 \\
&\quad +  \|\phi_1(t) \|_{H^2(\Omega)}^\frac{12}{5} \left( 1+ \ln \left(1+\| \nabla \uu\|_{\L2} \right) \right).
\end{align*}
In the derivation of \eqref{u-est3} we have used that the function $s \ln (\widetilde{C}/s )$ is increasing on $(0, \widetilde{C}/e)$.

We observe that $S\in L^1(0,T)$ provided that $\phi_1 \in L^{\gamma}(0,T;H^2(\Omega))$ with $\gamma>\frac{12}{5}$. Indeed, we recall the fact that
$\ln(1+s)\leq C(\kappa) (1+s)^\kappa$, for any $\kappa>0$ and $s>0$. Taking
$$\kappa= \frac{2(5\gamma -12)}{5\gamma},$$
we have
\begin{align*}
& \int_0^T \|\phi_1(\tau) \|_{H^2(\Omega)}^\frac{12}{5} \ln \left(1+\| \nabla \uu (\tau)\|_{\L2} \right) \, \d \tau\\
&\quad \leq C \int_0^T \| \phi_1 (\tau)\|_{H^2(\Omega)}^\frac{12}{5}
\left(1+\| \nabla \uu(\tau)\|_{\L2} \right)^\frac{2(5\gamma -12)}{5\gamma} \, \d \tau \\
&\quad \leq C \int_0^T \| \phi_1(\tau)\|_{H^2(\Omega)}^\gamma + \| \nabla \uu_1(\tau)\|_{\L2}^2+ \| \nabla \uu_2 (\tau)\|_{\L2}^2 \, \d \tau.
\end{align*}
Therefore, it holds $S\in L^1(0,T)$.

Integrating \eqref{u-est3} on the time interval $[0,t]$,  we find
$$
G(t) \leq G(0)+C \int_0^t S(s) G(s) \ln \left( \frac{\widetilde{C}}{G(s)} \right) \, \d s,
$$
for almost every $t \in [0,T]$. Observe that
$$\int_0^1 \frac{1}{s\ln(\frac{\widetilde{C}}{s})} \, \d s= +\infty.$$
Thus, if $G(0)=0$, applying the Osgood lemma \ref{Osgood}, we can deduce that $G(t)=0$ for all $t\in [0,T]$, namely
$$\uu_1(t)=\uu_2(t),\quad \phi_1(t)=\phi_2(t), \quad \forall\, t\in [0,T].$$
This demonstrates the uniqueness of solutions in the class of weak solutions satisfying the additional regularity $\phi_1 \in L^\gamma(0,T;H^2(\Omega))$ with $\gamma>\frac{12}{5}$.

In addition, we are able to deduce a continuous dependence estimate with respect to the initial datum. To this end, we define
$$\mathcal{M}(s)=\ln \left( \ln\left( \frac{\widetilde{C}}{s} \right) \right).$$
By the Osgood lemma, for $G(0)>0$, we are led to the inequality
\begin{equation}
\label{u-est4}
-\ln \left(\ln \left(\frac{\widetilde{C}}{G(t)}\right) \right)+\ln \left(\ln \left(\frac{\widetilde{C}}{G(0)}\right) \right)\leq C\int_0^t S(s)\, \d s
\end{equation}
for almost every $t \in [0,T]$.
Taking the double exponential of \eqref{u-est4}, we eventually infer the control
\begin{equation}
\label{CD}
G(t)\leq \widetilde{C} \left(\frac{G(0)}{\widetilde{C}}\right)^{  \exp(-C\int_{0}^t S(s)\, \mathrm{d}s)}, \qquad \forall \, t  \in [0,T_0],
\end{equation}
where $T_0>0$ is defined by
$$
\ln \left(\ln \left(\frac{\widetilde{C}}{G(0)}\right) \right) \geq C \int_0^{T_0} S(s)\, \d s.
$$

\textbf{Case 2.} Observing that $(\uu, \nabla p)=0$ and recalling \eqref{p}, we have the following alternative estimate:
\begin{equation*}
\begin{split}
\left| -\frac12 (\nu'(\phi_1) \nabla \phi_1 \cdot \uu, p) \right|
&= \left| \frac12 ( \nu(\phi_1) \uu, \nabla p ) \right|\\
&=  \frac12 \inf_{c\in [\nu_\ast,\nu^\ast]}\left|  \left(  \left(\nu(\phi_1)- c \right) \uu, \nabla p \right) \right|\\
&\leq \frac12 \inf_{c\in [\nu_\ast,\nu^\ast]}\left\| \nu(\phi_1)- c \right\|_{L^\infty(\Omega)} \| \uu\|_{L^2(\Omega)} \| \nabla p\|_{L^2(\Omega)}\\
&\leq \frac{C_\Omega^{(2)}}{2} \inf_{c\in [\nu_\ast,\nu^\ast]} \max_{s \in [-1,1]} \left| \nu(s)- c \right| \| \uu\|_{L^2(\Omega)}^2,
\end{split}
\end{equation*}
Exploiting the additional assumption \eqref{vis-add} on $\nu$, we obtain the differential inequality
\begin{equation}
\label{u-est5}
\ddt G(t) + \frac{\nu_\ast}{4} \| \uu\|_{\L2}^2 +  \frac12 \| \nabla \phi\|_{\L2}^2
\leq C \widetilde{S}(t)  G(t) \ln \Big( \frac{\widetilde{C}}{G(t)}\Big),
\end{equation}
where
$$
\widetilde{S}(t)= \Big(1+ \| \uu_1(t)\|_{H^1(\Omega)}^2+ \| \uu_2(t)\|_{H^1(\Omega)}^2
+ \|\phi_1(t) \|_{H^2(\Omega)}^2 +\| \phi_2(t)\|_{H^2(\Omega)}^2 \Big).
$$
Then the conclusion follows by arguing as in the first case.

The proof of Theorem \ref{uni2d} is complete.
\hfill $\square$
\medskip

\begin{remark}
\label{uniper}
We note that the same existence result as in Theorem \ref{W-S} holds in the periodic setting, i.e., $\Omega=\mathbb{T}^d$, $d=2,3$.
In the particular case $\Omega=\mathbb{T}^2$, uniqueness of global weak solutions can be achieved, without the conditions required in Theorem \ref{uni2d}.
Indeed, in this case the solutions of the Stokes operator $\A^{-1}\uu$ and $p$ are given by (see \cite[Chapter 2.2]{Temam})
$$
\A^{-1}\uu= \sum_{k\in \mathbb{Z}^2} g_k {e}^{\frac{2i\pi k \cdot x}{L}},
\quad p= \sum_{k\in \mathbb{Z}^2} p_k {e}^{\frac{2i\pi k \cdot x}{L}},
$$
where
$$
g_k=-\frac{L^2}{4\pi^2|k|^2} \Big( \uu_k-\frac{(k\cdot \uu_k)k}{|k|^2}\Big), \quad
p_k= \frac{L k \cdot \uu_k}{2i\pi |k|^2}, \quad
 k \in \mathbb{Z}^2, \ k\neq 0,
$$
with $L>0$ being the cell size. Here $\uu_k$ is the $k$-mode of $\uu$. We observe that we only need to consider the case $k \neq 0$ since $\overline{\uu}$ is conserved for \eqref{NSAC}$_1$ on $\mathbb{T}^2$, and thus we can choose $\overline{\uu}=0$. Moreover, since $\uu\in \H_\sigma$, we have $k \cdot \uu_k=0$ for any $k \in \mathbb{Z}^2$, which implies that $p_k=0$ for any $k \in \mathbb{Z}^2$. As a consequence, following the above arguments, we are led to the differential inequality \eqref{u-est2} without the last term on the right-hand side, i.e.,  $-\frac12 (\nu'(\phi_1)\nabla \phi_1\cdot \uu,p)$.
Hence, we eventually end up with
$$
\ddt G(t) + \frac{\nu_\ast}{4} \| \uu\|_{\L2}^2 +  \frac12 \| \nabla \phi\|_{\L2}^2
\leq C \widehat{S}(t)  G(t) \ln \Big( \frac{\widetilde{C}}{G(t)}\Big),
$$
where
$$
\widehat{S}(t)= \Big(1+ \| \uu_1(t)\|_{H^1(\Omega)}^2+ \| \uu_2(t)\|_{H^1(\Omega)}^2
+ \|\phi_1 (t)\|_{H^2(\Omega)}^2 +\| \phi_2(t)\|_{H^2(\Omega)}^2 \Big).
$$
Since $\widehat{S}\in L^1(0,T)$ for any couple of weak solutions on $[0,T]$, an application of the Osgood lemma as above entails the uniqueness of global weak solutions (now without additional regularity on $\phi$) and a continuous dependence estimate with respect to the initial data, i.e., \eqref{CD}.
\end{remark}

\section{Mass-conserving NSAC System: Strong Solutions}
\setcounter{equation}{0}
\label{S-STRONG}

This section is devoted to the analysis of global strong solutions to the Navier-Stokes-Allen-Cahn system \eqref{NSAC-D}-\eqref{IC-D} in two dimensions. The main results are as follows.

\begin{theorem}[Global strong solution in 2D]
\label{strong-D}
Let $\Omega$ be a bounded domain in $\mathbb{R}^2$ with smooth boundary $\partial\Omega$.
Assume that the initial data $\uu_0 \in \V_\sigma$, $\phi_0 \in H^2(\Omega)$ satisfy  $\partial_{\n} \phi_0=0$ on $\partial \Omega$, $F'(\phi_0)\in L^2(\Omega)$, $\|\phi_0 \|_{L^\infty(\Omega)}\leq 1$ and $|\overline{\phi}_0|<1$.
\smallskip

\begin{itemize}
\item[(1)]
There exists a global strong solution $(\uu,\phi)$ to problem \eqref{NSAC-D}-\eqref{IC-D} satisfying, for all $T>0$,
\begin{align*}
&\uu \in L^\infty(0,T;\V_\sigma)\cap L^2(0,T;\H^2(\Omega))\cap H^1(0,T;\H_\sigma),\\
&\phi \in L^\infty(0,T;H^2(\Omega))\cap L^2(0,T;W^{2,p}(\Omega)), \\
&\partial_t \phi \in L^\infty(0,T;\L2)\cap L^2(0,T;H^1(\Omega)),\\
&F'(\phi) \in L^\infty(0,T;\L2)\cap L^2(0,T;L^p(\Omega)),
\end{align*}
 for any $p \in (2,\infty)$. The strong solution satisfies the system \eqref{NSAC-D} almost everywhere in $\Omega \times (0,+\infty)$.
Besides, $|\phi(x,t)|<1$ holds for almost every $(x,t)\in \Omega\times(0,+\infty)$, the boundary condition  $\partial_{\n} \phi=0$ is satisfied almost everywhere on $\partial \Omega\times(0,+\infty)$ and the initial conditions are attained.
\smallskip

\item[(2)] There exists $\eta_1>0$ depending only on the norms of the initial data and on the parameters of the system:
\begin{align}
\eta_1=\eta_1(E(\uu_0,\phi_0), \| \uu_0\|_{\V_\sigma}, \| \phi_0\|_{H^2(\Omega)},\|F'(\phi_0)\|_{\L2},\theta,\theta_0).
\label{eeta1}
\end{align}
If, in addition,
$$\|\rho'\|_{L^\infty(-1,1)}\leq \eta_1\quad \text{and}\quad F''(\phi_0)\in L^1(\Omega),$$ then for any $T>0$, we have
\begin{align}
F''(\phi)\in L^\infty(0,T;L^1(\Omega)),\quad F''(\phi) \in L^q(0,T;L^p(\Omega)),
\end{align}
for any $p,q \in (1,\infty)$ satisfying  $\frac{1}{p}+\frac{1}{q}=1$, and
\begin{align}
\label{F''log}
(F''(\phi))^2 \ln (1+ F''(\phi)) \in L^1(\Omega\times(0,T)).
\end{align}
In particular, the strong solution satisfying \eqref{F''log} is unique.
\end{itemize}
\end{theorem}
\smallskip

\begin{theorem}[Propagation of regularity for strong solutions in 2D]
\label{Proreg-D}
Let the assumptions in Theorem \ref{strong-D}-(1) be satisfied.
Assume in addition that $\| \rho'\|_{L^\infty(-1,1)}\leq \eta_1$ (cf. \eqref{eeta1}). Given a strong  solution from Theorem \ref{strong-D}-(1), for any $\zeta>0$, there holds
$$
(F''(\phi))^2 \ln (1+ F''(\phi)) \in L^1(\Omega\times(\zeta,T)),
$$
and
$$
\partial_t\uu\in L^\infty(\zeta, T; \H_\sigma)\cap L^2(\zeta, T; \V_\sigma),\quad \partial_t \phi\in L^\infty(\zeta, T; H^1(\Omega))\cap L^2(\zeta, T; H^2(\Omega)).
$$
Moreover, for any $\zeta>0$, there exists $\delta=\delta(\zeta)>0$ such that
$$
-1+\delta \leq \phi(x,t) \leq 1-\delta, \qquad \forall \, x \in \overline{\Omega}, \ t \geq \zeta.
$$
\end{theorem}

\begin{remark}
The smallness assumption on $\rho'$ (see \eqref{Hyp} below for the explicit form) can be reformulated in terms of the difference of the (constant) densities of the two fluid components  when $\rho$ is a linear interpolation function. In this case, we have
$$
\rho(s)= \rho_1\frac{1+s}{2}+ \rho_2 \frac{1-s}{2},  \quad \rho'(s)= \frac{\rho_1-\rho_2}{2} \quad \forall \, s \in [-1,1].
$$
Roughly speaking, the results given by Theorem \ref{strong-D} and Theorem \ref{Proreg-D} imply that uniqueness and further regularity of strong solutions to the nonhomogeneous system can be achieved provided that the two fluids have a small difference in densities (i.e., $\rho_1 \approx \rho_2$).
\end{remark}

\begin{remark}[Matched densities]
\label{strong-hom}
It is worth noticing that Theorem  \ref{strong-D} and Theorem \ref{Proreg-D} hold true in the case of constant density $\rho\equiv 1$ (i.e., $\rho_1=\rho_2=1$) without any smallness assumption.
Moreover, thanks to Theorems \ref{uni2d} and \ref{strong-D},
any global weak solution $(\uu,\phi)$ to problem \eqref{NSAC}-\eqref{bic} with  $\rho\equiv 1$ defined in Theorem \ref{W-S} is a strong solution for any $t>0$.
\end{remark}

\subsection{Proof of Theorem \ref{strong-D}}
We perform higher-order \textit{a priori} estimates that are necessary for the existence of global strong solutions.\smallskip

\textbf{Higher-order estimates.}
Multiplying \eqref{NSAC-D}$_1$ by $\partial_t \uu$, integrating over $\Omega$, and observing that
$$
\int_{\Omega} \nu(	\phi)D\uu \cdot D \partial_t \uu \, \d x= \frac12 \ddt  \int_{\Omega} \nu(\phi) |D\uu|^2 \, \d x - \frac12 \int_{\Omega} \nu'(\phi) \partial_t \phi |D \uu|^2 \, \d  x,
$$
we obtain
\begin{align}
\frac12 &\ddt \int_{\Omega}  \nu(\phi) |D\uu|^2 \, \d x
+ \int_{\Omega} \rho(\phi) |\partial_t \uu|^2 \, \d x \notag \\
&= - ( \rho(\phi) \uu \cdot \nabla \uu, \partial_t \uu) - \int_{\Omega}\Delta \phi \, \nabla \phi \cdot \partial_t \uu \, \d x+ \frac12 \int_{\Omega} \nu'(\phi) \partial_t \phi |D \uu|^2 \, \d  x.
\label{NS1-D}
\end{align}
Next, differentiating \eqref{NSAC-D}$_3$ in time, multiplying the resultant by $\partial_t \phi$ and integrating over $\Omega$, we obtain
\begin{align}
\frac12 &\ddt \| \partial_t \phi\|_{\L2}^2+ \int_{\Omega} \partial_t \uu \cdot \nabla \phi \, \partial_t \phi \, \d x + \| \nabla \partial_t \phi\|_{\L2}^2
+\int_{\Omega} F''(\phi) |\partial_t \phi|^2 \, \d x  \notag \\
&= \theta_0 \| \partial_t \phi\|_{\L2}^2
- \int_{\Omega} \rho''(\phi) |\partial_t \phi|^2 \frac{|\uu|^2}{2}\, \d x
- \int_{\Omega} \rho'(\phi) \uu \cdot \partial_t \uu \, \partial_t \phi \, \d x + \partial_t \xi \int_{\Omega} \partial_t \phi \, \d x.
\label{AC1-D}
\end{align}
Adding the equations \eqref{NS1-D} and \eqref{AC1-D}, using $\overline{\partial_t \phi} =0$, we find that
\begin{align}
&\ddt H(t) +\rho_\ast \| \partial_t \uu\|_{\L2}^2 + \| \nabla \partial_t \phi\|_{\L2}^2+
\int_{\Omega} F''(\phi)|\partial_t \phi|^2 \, \d x \notag \\
&\quad \leq - (\rho(\phi) \uu \cdot \nabla \uu, \partial_t \uu) - \int_{\Omega}\Delta \phi \, \nabla \phi \cdot \partial_t \uu \, \d x +  \frac12 \int_{\Omega} \nu'(\phi) \partial_t \phi |D \uu|^2 \, \d  x+   \theta_0 \| \partial_t \phi\|_{\L2}^2 \notag \\
&\qquad - \int_{\Omega} \partial_t \uu \cdot \nabla \phi \, \partial_t \phi \, \d x - \int_{\Omega} \rho''(\phi) |\partial_t \phi|^2 \frac{|\uu|^2}{2}\, \d x
- \int_{\Omega} \rho'(\phi) \uu \cdot \partial_t \uu \, \partial_t \phi \, \d x,
\label{NSAC1-D}
\end{align}
where
\begin{equation}
\label{H-D}
H(t)= \frac12 \int_{\Omega}  \nu(\phi(t)) |D\uu(t)|^2 \, \d x + \frac12  \| \partial_t \phi(t)\|_{\L2}^2.
\end{equation}
In \eqref{NSAC1-D}, we have used the fact that $\rho$ is strictly positive ($\rho(s)\geq \rho_\ast>0$). In addition, we infer from \eqref{NSAC-D} that
$$
\|\partial_t \phi\|_{\L2}\leq C\left( 1+ \| \uu\|_{H^1(\Omega)}\right) \| \phi\|_{H^2(\Omega)}+C \| F'(\phi)\|_{\L2}+ C \| \uu\|_{H^1(\Omega)}^2.
$$
Therefore, it follows from the assumptions on the initial data that $H(0)<+\infty$.

We proceed to estimate the right-hand side of \eqref{NSAC1-D}.
By using \eqref{KORN} and \eqref{BGW}, we have
\begin{align*}
-( \rho(\phi) \uu \cdot \nabla \uu, \partial_t \uu)
&\leq\| \rho(\phi)\|_{L^\infty(\Omega)} \| \uu\|_{L^\infty(\Omega)} \| \nabla \uu\|_{\L2} \| \partial_t \uu\|_{\L2}\\
&\leq C \| D\uu\|_{L^2(\Omega)}^2 \ln^\frac12 \left( C \frac{\| \uu\|_{W^{1,p}(\Omega)}}{\| D \uu\|_{L^2(\Omega)}} \right) \| \partial_t \uu\|_{\L2}\\
&\leq \frac{\rho_\ast}{8}  \|\partial_t \uu \|_{\L2}^2+ C \| D \uu\|_{L^2(\Omega)}^4 \ln \left( C \frac{\| \uu\|_{W^{1,p}(\Omega)}}{\| D \uu\|_{L^2(\Omega)}}\right),
\end{align*}
for some $p>2$. Moreover, it holds
\begin{align*}
- \int_{\Omega}\Delta \phi \, \nabla \phi \cdot \partial_t \uu \, \d x
&\leq \| \Delta \phi \|_{\L2} \| \nabla \phi\|_{L^\infty(\Omega)}
\| \partial_t \uu\|_{\L2}\\
&\leq C \|\Delta \phi \|_{\L2} \| \nabla \phi\|_{H^1(\Omega)} \ln^\frac12 \left( C \frac{\| \nabla \phi\|_{W^{1,p}(\Omega)}}{\| \nabla \phi\|_{H^1(\Omega)}}\right) \| \partial_t \uu\|_{\L2}\\
&\leq \frac{\rho_\ast}{8}  \|\partial_t \uu \|_{\L2}^2+ C  \|\Delta \phi \|_{\L2}^2 \| \nabla \phi\|_{H^1(\Omega)}^2\ln \left( C \frac{\| \nabla \phi\|_{W^{1,p}(\Omega)}}{\| \nabla \phi\|_{H^1(\Omega)}}\right).
\end{align*}
Next, by exploiting Lemma \ref{result1} and $\overline{\partial_t \phi}=0$, we obtain
\begin{align*}
 \frac12 \int_{\Omega} \nu'(\phi) \partial_t \phi |D \uu|^2 \, \d  x
 &\leq \| \nu'(\phi)\|_{L^\infty(\Omega)} \| \partial_t \phi |D \uu|\|_{\L2} \| D \uu\|_{\L2}\\
 & \leq C \| \nabla \partial_t \phi\|_{\L2} \| D \uu\|_{\L2}^2 \ln^\frac12 \left( C \frac{\|D  \uu\|_{L^p(\Omega)}}{\| D \uu\|_{\L2}}\right)\\
 &\leq \frac18  \| \nabla \partial_t \phi\|_{\L2}^2+ C \| D \uu\|_{\L2}^4 \ln \left( C \frac{\|D  \uu\|_{L^p(\Omega)}}{\| D \uu\|_{\L2}}\right).
\end{align*}
It remains to control the last three terms on the right-hand side of \eqref{NSAC1-D}. By using \eqref{LADY} and \eqref{B1-D}, we obtain
\begin{align*}
- \int_{\Omega} \partial_t \uu \cdot \nabla \phi \, \partial_t \phi \, \d x
&\leq   \|  \partial_t \uu\|_{\L2} \| \nabla \phi\|_{L^4(\Omega)} \| \partial_t \phi\|_{L^4(\Omega)}\\
&\leq \|  \partial_t \uu\|_{\L2} \| \nabla \phi\|_{\L2}^\frac12 \| \phi\|_{H^2(\Omega)}^\frac12 \| \partial_t \phi\|_{\L2}^\frac12 \| \nabla \partial_t \phi\|_{\L2}^\frac12 \\
&\leq \frac{\rho_\ast}{8}  \|  \partial_t \uu\|_{\L2}^2
+\frac18 \| \nabla \partial_t \phi\|_{\L2}^2+ C \| \phi\|_{H^2(\Omega)}^2
\| \partial_t \phi\|_{\L2}^2,
\end{align*}
\begin{align*}
- \int_{\Omega} \rho''(\phi) |\partial_t \phi|^2 \frac{|\uu|^2}{2}\, \d x
&\leq C \| \rho''(\phi)\|_{L^\infty(\Omega)} \| \partial_t \phi\|_{L^4(\Omega)}^2
\|\uu\|_{L^4(\Omega)}^2\\
&\leq C \|\partial_t \phi \|_{\L2} \| \nabla \partial_t  \phi\|_{\L2} \| \uu\|_{\L2}\| \nabla \uu\|_{\L2}\\
&\leq \frac18 \|  \nabla \partial_t  \phi\|_{\L2}^2
+C  \|\partial_t \phi \|_{\L2}^2\| D\uu\|_{\L2}^2,
\end{align*}
and
\begin{align*}
- \int_{\Omega} \rho'(\phi) \uu \cdot \partial_t \uu \, \partial_t \phi \, \d x
&\leq C \|  \rho'(\phi)\|_{L^\infty(\Omega)}
\| \uu\|_{L^4(\Omega)}\| \partial_t \uu\|_{\L2} \| \partial_t \phi\|_{L^4(\Omega)}\\
&\leq C \| \uu\|_{\L2}^\frac12 \|\nabla \uu\|_{\L2}^\frac12
\| \partial_t \uu\|_{\L2} \| \partial_t \phi\|_{\L2}^\frac12  \| \nabla \partial_t \phi\|_{\L2}^\frac12\\
&\leq \frac{\rho_\ast}{8} \| \partial_t \uu\|_{\L2}^2+ \frac18 \|  \nabla \partial_t  \phi\|_{\L2}^2
+C \| \partial_t \phi\|_{\L2}^2\| D\uu\|_{\L2}^2.
\end{align*}
Combining \eqref{NSAC1-D} and the above inequalities, we deduce that
\begin{align*}
& \ddt H(t) +\frac{\rho_\ast}{2}   \| \partial_t \uu\|_{\L2}^2 + \frac12 \| \nabla \partial_t \phi\|_{\L2}^2 \\
&\quad \leq C \| \partial_t \phi\|_{\L2}^2
+C  \big(\| D\uu\|_{\L2}^2+\|\phi \|_{H^2(\Omega)}^2\big)
\|\partial_t \phi \|_{\L2}^2 \\
&\qquad + C\| D \uu\|_{L^2(\Omega)}^4 \ln \left( C \frac{\| \uu\|_{W^{1,p}(\Omega)}}{\|D \uu\|_{L^2(\Omega)}}\right) + C \| \Delta \phi\|_{\L2}^2 \| \nabla \phi\|_{H^1(\Omega)}^2 \ln \left( C \frac{\| \nabla \phi\|_{W^{1,p}(\Omega)}}{\| \nabla \phi\|_{H^1(\Omega)}}\right).
\end{align*}
From the inequalities
\begin{align}
\label{ineq0}
&x^2\ln \left(\frac{C y}{x}\right)\leq x^2 \ln (Cy) +1,\quad \forall\,x,\,y>0,\\
& \frac{\nu_\ast}{2} \| D \uu\|_{\L2}^2 +\frac12 \| \partial_t \phi\|_{\L2}^2
\leq  H(t)\leq C \Big( \| D \uu\|_{\L2}^2 +\| \partial_t \phi\|_{\L2}^2\Big),
\label{Hbb}
\end{align}
and the estimate \eqref{estH2-D}, we can rewrite the above differential inequality as follows
\begin{align}
\label{NSAC2-D}
& \ddt H(t) +\frac{\rho_\ast}{2}  \| \partial_t \uu\|_{\L2}^2 + \frac12 \| \nabla \partial_t \phi\|_{\L2}^2  \notag \\
&\quad \leq C\left(1+ H(t) + H^2(t)\right)
+C H^2(t) \ln \left( C \| \uu\|_{W^{1,p}(\Omega)}\right)
 + C \left( 1+H^2(t) \right)
\ln \left( C \|  \phi\|_{W^{2,p}(\Omega)}\right).
\end{align}

Let us now estimate the argument of the logarithmic terms on the right-hand side of
\eqref{NSAC2-D}. First, we rewrite \eqref{NSAC-D}$_1$ as a Stokes problem with non-constant viscosity
$$
\begin{cases}
-\div (\nu(\phi)D \uu)+\nabla P= \f, & \text{a.e. in } \Omega\times (0,T),\\
\div \uu=0, & \text{a.e. in } \Omega\times (0,T),\\
\uu=\mathbf{0}, & \text{a.e. on } \partial \Omega\times (0,T),
\end{cases}
$$
where $$\f=  -\rho(\phi) \big( \partial_t \uu + \uu \cdot \nabla \uu \big) - \Delta \phi \nabla \phi.$$
Applying Theorem \ref{Stokes-e} with the following choice of parameters
$$
p=1+\varepsilon\ \ \text{with}\ \ \varepsilon \in (0,1),\quad\text{ and}\ \ r\in (2,\infty)\ \ \text{such that}\ \ \frac{1}{r}=\frac{1}{1+\varepsilon}-\frac12,
$$
we infer that
\begin{align*}
\| \uu\|_{W^{2,1+\varepsilon}(\Omega)}
&\leq C \left( \| \partial_t \uu\|_{L^{1+\varepsilon}(\Omega)}+ \| \uu\cdot \nabla \uu\|_{L^{1+\varepsilon}(\Omega)}
+\|\Delta \phi \nabla \phi\|_{L^{1+\varepsilon}(\Omega)} \right)\\
&\quad +
C \| D \uu\|_{\L2} \| \nabla \phi\|_{L^r(\Omega)}\\
&\leq C \left( \| \partial_t \uu\|_{\L2}+ \| \uu\|_{L^\frac{2(1+\varepsilon)}{1-\varepsilon}(\Omega)}\|\nabla \uu \|_{\L2}+
\| \nabla \phi\|_{L^\frac{2(1+\varepsilon)}{1-\varepsilon}(\Omega)} \| \Delta \phi\|_{\L2}  \right)\\
&\quad  + C \| D \uu\|_{\L2} \| \phi\|_{H^2(\Omega)}\\
&\leq   C \| \partial_t \uu\|_{\L2}+ C  \|D \uu \|_{\L2}^2 +
 \|  \phi\|_{H^2(\Omega)}^2\\
 &\leq C \| \partial_t \uu\|_{\L2}+ C(1+H(t)),
\end{align*}
where the constant $C$ depends on $\varepsilon$.
We recall the Sobolev embedding $W^{2,1+\varepsilon}(\Omega)\hookrightarrow W^{1,p}(\Omega)$ where $\frac{1}{p}=\frac{1}{1+\varepsilon}- \frac12$. Then for any $p\in (2,\infty)$, there exists a constant $C>0$ depending on $p$ such that
\begin{align}
\|\uu\|_{W^{1,p}(\Omega)}\leq C \| \partial_t \uu\|_{\L2} + C\left( 1+H(t) \right).
\label{est-uw1p}
\end{align}

Next, we reformulate the equation \eqref{NSAC-D}$_4$ as the elliptic problem
\begin{equation}
\begin{cases}
-\Delta \phi+ F'(\phi)=\mu+\theta_0\phi,  &\quad \text{a.e. in}\ \Omega\times(0,T),\\
\partial_\n \phi=0, &\quad \text{a.e. on}\ \partial \Omega\times (0,T).
\end{cases}
\end{equation}
From the elliptic regularity theory (see, e.g., \cite[Lemma 2]{A2009} and \cite{GMT2019}) we find that
\begin{equation}
\label{pw2p}
\begin{split}
\| \phi\|_{W^{2,p}(\Omega)} +\|F'(\phi)\|_{L^p(\Omega)}
&\leq C (1+\|\phi\|_{L^2(\Omega)}+\|\mu+\theta_0\phi\|_{L^p(\Omega)})\\
&\leq C (1+\|\phi\|_{L^p(\Omega)}+\|\mu\|_{L^p(\Omega)}),
\end{split}
\end{equation}
for any $p\in (2,\infty)$. On the other hand, from the equation \eqref{NSAC-D}$_3$, we have
$$
\mu= -\partial_t \phi-\uu\cdot \nabla \phi- \rho'(\phi) \frac{|\uu|^2}{2}+\displaystyle{\overline{\mu+ \rho'(\phi)\frac{|\uu|^2}{2}}}.
$$
Observe that
$$
\left\| -\rho'(\phi) \frac{|\uu|^2}{2} \right\|_{L^p(\Omega)}\leq C \| \uu\|_{L^{2p}(\Omega)}^2
\leq C \|\nabla \uu \|_{L^2(\Omega)}^2.
$$
Then, owing to the Sobolev embedding theorem and \eqref{normH1-2}, we have
\begin{align*}
 \|\mu-\overline{\mu} \|_{L^p(\Omega)}
 &\leq \| \partial_t \phi\|_{L^p(\Omega)} + \| \uu \cdot \nabla \phi\|_{L^p(\Omega)}+\left\| \rho'(\phi) \frac{|\uu|^2}{2} -
 \displaystyle{\overline{\rho'(\phi)\frac{|\uu|^2}{2}}}\right\|_{L^p(\Omega)}\\
 &\leq C \| \nabla \partial_t \phi\|_{\L2}
 +C  \| \uu\|_{H^1(\Omega)} \| \phi\|_{H^2(\Omega)}+C \|\nabla \uu \|_{L^2(\Omega)}^2.
\end{align*}
In light of \eqref{mu-L2-2} and \eqref{mubar}, the above inequality yields
\begin{align}
\| \mu\|_{L^p(\Omega)}
&\leq C \|\mu-\overline{\mu} \|_{L^p(\Omega)} +C |\overline{\mu}| \notag \\
&\leq C \|\mu-\overline{\mu} \|_{L^p(\Omega)}+
C \left( 1+\| \mu-\overline{\mu} \|_{\L2} \right) \notag \\
 &\leq C \left(1+\| \nabla \partial_t \phi\|_{\L2}
+H(t) \right).
\label{mu-Lp}
\end{align}
Thus, for any $p>2$, we deduce from the above estimate and \eqref{pw2p} that
\begin{equation}
\label{estW2p-D}
\| \phi\|_{W^{2,p}(\Omega)} \leq C \left( 1+\| \nabla \partial_t \phi\|_{\L2} +H(t) \right),
\end{equation}
for some positive constant $C$ depending on $p$.

Recalling the generalized Young's inequality
\begin{align}
xy \leq \Phi(x)+ \Upsilon(y), \quad \forall \, x,\,y >0,
\label{Young0}
\end{align}
where
$$
\Phi(s)= s \ln s -s+1, \quad \Upsilon(s)= {e}^{s}-1,
$$
we have
\begin{align*}
H(t) \ln \left(1+\| \partial_t \uu\|_{\L2} \right)
&\leq H(t)\ln H(t)+ 1 + \| \partial_t \uu\|_{\L2}.
\end{align*}
Then using the above estimate and the elementary inequality
$$ \ln(x+y)<\ln(1+x)+\ln(1+y),\quad \forall\, x,y>0,$$
we can estimate the second term on the right-hand side of \eqref{NSAC2-D} as follows
\begin{align}
&CH^2(t )\ln \left( C\|\uu\|_{W^{1,p}(\Omega)}\right) \notag \\
&\quad \leq CH^2(t)\ln \left( C \| \partial_t \uu\|_{\L2} + C(1+H(t)) \right) \notag \\
&\quad \leq CH^2(t) \left(1+\ln(1+\| \partial_t \uu\|_{\L2})+\ln(1+H(t)) \right) \notag \\
&\quad \leq CH^2(t)+ CH(t)\left( H(t)\ln H(t)+ 1 \right) \notag \\
&\qquad + CH(t)\| \partial_t \uu\|_{\L2}+ CH^2(t)\ln(1+H(t)) \notag \\
&\quad \leq \frac{\rho_\ast}{4}  \| \partial_t \uu\|_{\L2}^2+ C\left(1+H^2(t)\right)+ CH^2(t)\ln(e+H(t)).
\label{RHD3-D}
\end{align}
In a similar manner, we have
\begin{align*}
H(t)\ln(1+\| \nabla \partial_t \phi\|_{\L2})
&\leq H(t)\ln H(t)+ 1 + \| \nabla \partial_t \phi\|_{\L2}.
\end{align*}
Using \eqref{estW2p-D}, the third term on the right-hand side of \eqref{NSAC2-D} can be estimated as follows
\begin{align}
&C \left( 1+H^2(t) \right)
\ln \left( C \|  \phi\|_{W^{2,p}(\Omega)}\right) \notag  \\
&\quad \leq C \left( 1+H^2(t) \right)  \ln \left(C(1+\| \nabla \partial_t \phi\|_{\L2} +H(t)) \right) \notag \\
&\quad \leq C\left(1+H^2(t) \right) +C\ln \left(1+\|\nabla \partial_t \phi\|_{L^2(\Omega)}+H(t)\right) \notag \\
&\qquad +H^2(t)\ln \left(1+\| \nabla \partial_t \phi\|_{\L2}+ H(t) \right) \notag \\
&\quad \leq C \left( 1+H^2(t) \right) +C \left( 1+\| \nabla \partial_t \phi\|_{\L2}+H(t) \right) + C\| \nabla \partial_t \phi\|_{\L2} H(t) \notag \\
&\qquad
+ H^2(t) \ln (1+H(t)) \notag \\
&\quad \leq \frac{1}{8}\| \nabla \partial_t \phi\|_{\L2}^2+ C\left(1+H^2(t)\right)+ C H(t)\left( e+H(t) \right) \ln (e+H(t)).
\label{RHD4-D}
\end{align}

Hence, by \eqref{RHD3-D} and \eqref{RHD4-D}, we easily deduce from \eqref{NSAC2-D} that
\begin{align}
\ddt (e+H(t)) &+\frac{\rho_\ast}{4}  \| \partial_t \uu\|_{\L2}^2 + \frac14 \| \nabla \partial_t \phi\|_{\L2}^2\leq C+CH(t) (e+H(t)) \ln (e+H(t)).
\label{NSAC3-D}
\end{align}
On the other hand, thanks to \eqref{E-bound}, \eqref{estH2-D} and \eqref{Hbb}, we infer that
\begin{equation}
\label{intH}
\int_t^{t+1} H(\tau) \, \d \tau \leq Q(E_0), \quad \forall \, t \geq 0,
\end{equation}
where $Q$ is independent of $t$, and $E_0=E(\uu_0,\phi_0)$.
Applying the generalized Gronwall lemma \ref{GL2} to \eqref{NSAC3-D}, we find the estimate
$$
\sup_{t \in [0,1]} H(t)\leq C \big(e+H(0)\big)^{{e}^{Q(E_0)}}.
$$
Besides, by using the generalized uniform Gronwall lemma \ref{UGL2} together with \eqref{intH}, we infer that
$$
\sup_{t\geq 1} H(t)\leq C e^{(e+Q(E_0) ) e^{(1+Q(E_0))}}.
$$
Combining the above inequalities, we can obtain
\begin{align}
\sup_{t \geq 0} H(t)\leq Q \left(E_0, \| \uu_0\|_{\V_\sigma}, \| \phi_0\|_{H^2(\Omega)}, \| F'(\phi_0)\|_{L^2(\Omega)} \right).
\label{NSAC3h-D}
\end{align}
In addition, integrating \eqref{NSAC3-D} on the time interval $[t,t+1]$, we have, for all $t\geq 0$,
\begin{align}
 \int_t^{t+1} \| \partial_t \uu(\tau)\|_{\L2}^2 +  \| \nabla \partial_t \phi(\tau)\|_{\L2}^2\, \d \tau \leq Q \left( E_0, \| \uu_0\|_{\V_\sigma}, \| \phi_0\|_{H^2(\Omega)}, \| F'(\phi_0)\|_{L^2(\Omega)} \right).
\label{NSAC4}
\end{align}

From the above estimates, we can deduce that
\begin{equation}
\label{str-1}
\uu \in L^\infty(0,T; \V_\sigma)\cap H^1(0,T; \H_\sigma), \quad
\partial_t \phi \in L^\infty(0,T; L^2(\Omega))\cap L^2(0,T;H^1(\Omega)), \quad \forall \, T >0.
\end{equation}
Thanks to \eqref{estH2-D} and \eqref{estW2p-D}, we also get
\begin{equation}
\label{str-2}
\sup_{t\geq 0}  \| \phi(t)\|_{H^2(\Omega))}
 \leq  Q \left( E_0, \| \uu_0\|_{\V_\sigma}, \| \phi_0\|_{H^2(\Omega)}, \| F'(\phi_0)\|_{L^2(\Omega)} \right),
\end{equation}
and, for all $t\geq 0$,
\begin{equation}
\label{str-2'}
\int_t^{t+1} \| \phi(\tau)\|_{W^{2,p}(\Omega))}^2 \, \d \tau  \leq  Q \left( E_0, \| \uu_0\|_{\V_\sigma}, \| \phi_0\|_{H^2(\Omega)}, \| F'(\phi_0)\|_{L^2(\Omega)} \right),
\end{equation}
for any $p \in (2,\infty)$.
This entails that
$$\phi \in L^\infty(0,T;H^2(\Omega))\cap L^2(0,T;W^{2,p}(\Omega)),\quad   \forall\,T>0.$$
According to \eqref{mu-L2-2}, \eqref{mubar} and \eqref{mu-Lp}, it further follows that
$$
\mu\in  L^\infty(0,T;L^2(\Omega))\cap L^2(0,T;L^{p}(\Omega)),\quad   \forall\, T>0.
$$
As a consequence, it holds that
$$
F'(\phi) \in L^\infty(0,T;\L2)\cap L^2(0,T;L^p(\Omega)), \quad \forall \, T>0.
$$
Finally, by exploiting Theorem \ref{Stokes-e} with $p=2$ and $r=\infty$, together with the regularity of $\phi$ obtained above, we have, for all $t\geq 0$,
\begin{equation}
\label{str-3}
\int_t^{t+1} \|\uu(\tau)\|_{H^2(\Omega)}^2 \, \d \tau \leq
Q \left( E_0, \| \uu_0\|_{\V_\sigma}, \| \phi_0\|_{H^2(\Omega)}, \| F'(\phi_0)\|_{L^2(\Omega)} \right),
\end{equation}
which yields that
$$\uu \in L^2(0,T;\mathbf{H}^2(\Omega)),  \quad \forall\, T>0.$$

\textbf{Existence of global strong solutions.} With the above \textit{a priori} estimates, we are able to prove the existence of global strong solutions using a standard approximation procedure and pass to the limit (cf. the proof of Theorem \ref{weak-D}). The details are omitted here.
\medskip

\textbf{Entropy bound in $L^\infty(0,T;L^1(\Omega))$.}
First of all, we observe that, for all $s\in (-1,1)$,
\begin{equation}
\label{Fder1}
F'(s)= \frac{\theta}{2} \ln \left( \frac{1+s}{1-s} \right), \quad F''(s)= \frac{\theta}{1-s^2}, \quad F'''(s)= \frac{2\theta s}{(1-s)^2(1+s)^2},
\end{equation}
and
\begin{equation}
\label{Fder2}
F^{(4)}(s)=  \frac{2 \theta(1+3s^2)}{(1-s)^3(1+s)^3}>0.
\end{equation}
Next, we compute
\begin{align*}
\ddt \int_{\Omega} F''(\phi) \, \d x&= \int_{\Omega} F'''(\phi) \partial_t \phi \, \d x\\
&=\int_{\Omega} F'''(\phi) \left( \Delta \phi - \uu \cdot \nabla \phi -F'(\phi)+ \theta_0 \phi - \rho'(\phi) \frac{|\uu|^2}{2}+ \xi \right) \, \d x.
\end{align*}
Since
$$
\int_{\Omega} F'''(\phi) \uu \cdot \nabla \phi \, \d x=
 \int_{\Omega} \uu \cdot \nabla ( F''(\phi)) \, \d x=0,
$$
and exploiting the integration by parts, we rewrite the above equality as follows
\begin{align}
&\ddt \int_{\Omega} F''(\phi) \, \d x + \int_{\Omega} F^{(4)}(\phi) |\nabla \phi|^2 \, \d x + \int_{\Omega} F'''(\phi) F'(\phi) \, \d x \notag \\
&\quad  =  \int_{\Omega} F'''(\phi) \left( \theta_0 \phi - \rho'(\phi) \frac{|\uu|^2}{2}+\xi \right) \, \d x.
\label{EntE-}
\end{align}
In particular, from \eqref{Fder2}, we obtain
\begin{equation}
\label{EntE}
\ddt \int_{\Omega} F''(\phi) \, \d x + \int_{\Omega} F'''(\phi) F'(\phi) \, \d x
\leq  \int_{\Omega} F'''(\phi) \left( \theta_0 \phi - \rho'(\phi) \frac{|\uu|^2}{2} + \xi \right) \, \d x.
\end{equation}
It follows from \eqref{Young0} that
\begin{equation}
\label{Young}
xy \leq \varepsilon x \ln x + {e}^{\frac{y}{\varepsilon}},\quad \forall \, x>0,\ y>0,\   \varepsilon \in (0,1),
\end{equation}
which implies
\begin{equation}
\label{EE1}
\begin{split}
\int_{\Omega} -F'''(\phi) \rho'(\phi) \frac{|\uu|^2}{2} \, \d x&\leq
\int_{\Omega} |F'''(\phi)| |\rho'(\phi)| \frac{|\uu|^2}{2} \, \d x  \\
& \leq \varepsilon \int_{\Omega} |F'''(\phi)| \ln ( |F'''(\phi)| ) \, \d x+
\int_{\Omega} {e}^{\frac{|\rho'(\phi)|}{\varepsilon} \frac{|\uu|^2}{2}}\, \d x.
\end{split}
\end{equation}
We observe that, for all $s\in [0,1)$, it holds
\begin{align*}
\ln \left( |F'''(s)| \right) &= \ln \left( F'''(s) \right)= \ln \left( \frac{2\theta s}{(1-s)^2(1+s)^2}\right)\notag \\
&= 2 \ln \left( \frac{1+s}{1-s} \frac{\sqrt{2\theta s}}{(1+s)^2}\right) \leq 2 \ln \left( \sqrt{2\theta} \frac{1+s}{1-s}\right)\notag\\
& = \ln(2\theta) + \frac{4}{\theta} F'(s).
\end{align*}
Since both $F'(s)$ and $F'''(s)$ are odd functions, we easily deduce that
$$
\ln \left( |F'''(s)| \right) \leq C_0+ \frac{4}{\theta} |F'(s)|, \quad \forall \, s \in (-1,1),
$$
where $C_0=\ln(2\theta)$ (without loss of generality, we assume in the sequel that $C_0>0$).
Then, using the fact that $F'''(s)F'(s)\geq 0$ for all $s\in (-1,1)$, we obtain
\begin{align*}
|F'''(s)|\ln \left( |F'''(s)| \right) \leq C_0|F'''(s)|+ \frac{4}{\theta} F'''(s) F'(s), \quad \forall \, s \in (-1,1).
\end{align*}
Fix the constant $\alpha \in (0,1)$ such that $F'(\alpha)=1$. We infer that
\begin{equation}
\label{estF'''}
|F'''(s)|\ln \left( |F'''(s)| \right)\leq C_1+ C_2F'''(s) F'(s), \quad \forall \, s \in (-1,1),
\end{equation}
where
$$
C_1= C_0F'''(\alpha), \quad C_2=\frac{4}{\theta} +C_0.
$$
Taking $\varepsilon=\frac{1}{2C_2}$ in \eqref{EE1}, we arrive at
\begin{align}
\int_{\Omega} -F'''(\phi) \rho'(\phi) \frac{|\uu|^2}{2} \, \d x
& \leq \frac{C_1 |\Omega|}{2C_2} + \frac12 \int_{\Omega} F'''(\phi) F'(\phi) \, \d x+
\int_{\Omega} {e}^{C_2 |\rho'(\phi)| |\uu|^2}\, \d x.
\label{EntE2}
\end{align}
Arguing in a similar way (with $\varepsilon= \frac{1}{4C_2}$), we obtain
$$
\int_{\Omega} F'''(\phi) \, (\theta_0 \phi+ \xi) \, \d x \leq \frac{C_1 |\Omega|}{4C_2} + \frac14 \int_{\Omega} F'''(\phi) F'(\phi) \, \d x+
\int_{\Omega} {e}^{4C_2 |\theta_0\phi +\xi |}\, \d x.
$$
Since $\phi$ is globally bounded (i.e., $\|\phi \|_{L^\infty(0,T;L^\infty(\Omega))}\leq 1$) and $\|\xi\|_{L^\infty(0,T)}\leq C^\ast_2$, we get
\begin{equation}
\label{EntE3}
\int_{\Omega} F'''(\phi) \, (\theta_0+\xi) \phi \, \d x \leq \frac14 \int_{\Omega} F'''(\phi) F'(\phi) \, \d x+  \frac{C_1 |\Omega|}{4C_2} +
{e}^{4 C_2 (\theta_0+C_2^\ast)} |\Omega|.
\end{equation}

Combining \eqref{EntE} with \eqref{EntE2} and \eqref{EntE3}, we deduce that
\begin{align}
&\ddt \int_{\Omega} F''(\phi) \, \d x + \frac14 \int_{\Omega} F'''(\phi) F'(\phi) \, \d x \notag  \\
&\quad \leq\frac{3C_1 |\Omega|}{4C_2} +
{e}^{4 C_2 (\theta_0+C_2^\ast)} |\Omega|+  \int_{\Omega} {e}^{ C_2 |\rho'(\phi)| \| \nabla \uu\|_{\L2}^2 \left(\frac{|\uu|^2}{\|\nabla \uu\|_{\L2}^2}\right)}\, \d x.
\label{EntE4}
\end{align}
In order to control the last term on the right-hand side of \eqref{EntE4}, we shall use the Trudinger-Moser inequality in two dimensions (see, e.g., \cite{Moser}). Namely, let $f\in H_0^1(\Omega)$ ($d=2$) such that
$\int_{\Omega} |\nabla f|^2 \,\d x\leq 1$. Then, there exists a constant $C_{TM}=C_{TM}(\Omega)$ (which depends only on the domain $\Omega$) such that
\begin{align}
\int_{\Omega} {e}^{4\pi |f|^2} \, \d x \leq C_{TM}(\Omega).\label{TrM}
\end{align}
As a consequence of \eqref{NSAC3h-D}, we have the following uniform estimate
\begin{equation}
\sup_{t\geq 0 }\| \nabla \uu(t)\|_{\L2}\leq Q \left( E(\uu_0,\phi_0) ,H(0) \right)=:R_0,
\end{equation}
where $R_0$ is independent of time. The exact value of $R_0$ can be estimated in terms of the norm of the initial conditions.
Now we make the following assumptions:
\begin{equation}
\label{Hyp} |\rho'(s)|_{L^\infty(-1,1)}\leq \frac{4 \pi}{C_2 R_0^2}.
\end{equation}
Thanks to \eqref{Hyp}, we conclude from \eqref{EntE4} that
\begin{equation}
\label{EntE5}
\ddt \int_{\Omega} F''(\phi) \, \d x + \frac14 \int_{\Omega} F'''(\phi) F'(\phi) \, \d x
\leq\frac{3C_1 |\Omega|}{4C_2} +
{e}^{4 C_2 (\theta_0+C_2^\ast)} |\Omega|
+ C_{TM}(\Omega).
\end{equation}
Observe that, for $s\in \big[\frac12,1)$,
\begin{align*}
F''(s)=\frac{\theta}{1-s^2}= \frac{(1-s)(1+s)}{2s} F'''(s)\leq 	\frac{3}{4F'\left( \frac12 \right) } F'''(s)F'(s).
\end{align*}
This gives
\begin{equation}
\label{estF''}
F''(s)\leq C_3 + C_4 F'''(s)F'(s), \quad \forall \, s \in (-1,1),
\end{equation}
where
$$
C_3= F''\left( \frac12 \right),  \quad C_4=\frac{3}{4F' \left( \frac12 \right)}.
$$
Hence, we are led to
$$
\ddt \int_{\Omega} F''(\phi) \, \d x + \frac{1}{4 C_4} \int_{\Omega} F''(\phi)  \, \d x
\leq C_5,
$$
where
$$
C_5=\frac{3C_1 |\Omega|}{4C_2} +
{e}^{4 C_2 (\theta_0+C_2^\ast)} |\Omega|+ C_{TM}(\Omega)+ \frac{C_3|\Omega|}{4C_4}.
$$
We recall the assumption $F''(\phi_0)\in L^1(\Omega)$. Then, an application of Gronwall's lemma entails that
\begin{equation}
\label{EB1}
\int_{\Omega} F''(\phi(t)) \, \d x \leq \| F''(\phi_0)\|_{L^1(\Omega)} {e}^{-\frac{t}{4C_4}} + 4 C_4 C_5, \quad \forall \, t \geq 0.
\end{equation}
In addition, integrating \eqref{EntE5} on the time interval $[t,t+1]$, we find
\begin{equation}
\label{EB2}
\int_t^{t+1}\! \int_{\Omega} F'''(\phi) F'(\phi) \,\d x \d \tau \leq 4 \| F''(\phi_0)\|_{L^1(\Omega)} + C_6,  \quad \forall \, t \geq 0,
\end{equation}
where $$C_6= 4 C_5-  \frac{C_3|\Omega|}{C_4}.$$

The above estimate allows us to improve the integrability of $F''(\phi)$. Indeed, arguing similarly to \eqref{estF''}, we have for $s\in \big[ \frac12 , 1)$
\begin{align*}
(F''(s))^2 \ln \left( 1+F''(s)\right)&= \frac{\theta^2}{(1-s)^2(1+s)^2} \ln \left( 1+\frac{\theta}{1-s^2} \right) \\
&\leq  \theta F'''(s) \ln \left( \frac{1+s}{1-s}  \frac{1-s^2 +\theta}{(1+s)^2}\right)\\
&\leq 2F'''(s)F'(s) + \theta F'''(s) \ln \left( \frac12 + \frac{2\theta}{3} \right) \\
&\leq C_7 F'''(s)F'(s).
\end{align*}
Hence, we infer that
$$
(F''(s))^2 \ln \left(1+F''(s) \right)\leq C_7 F'''(s)F'(s)+ C_8, \quad \forall \, s\in (-1,1).
$$
In light of \eqref{EB2}, we deduce \eqref{F''log}. Indeed, we have
\begin{equation}
\label{EB3}
\int_t^{t+1}\! \int_{\Omega} (F''(\phi))^2 \ln \left( 1+F''(\phi) \right) \,\d x \d \tau \leq
4 C_7 \| F''(\phi_0)\|_{L^1(\Omega)} + C_6 C_7 +C_8, \quad \forall \, t \geq 0.
\end{equation}

We notice that, by keeping the (non-negative) term $F^{(4)}(\phi)|\nabla \phi|^2$ (cf. \eqref{EntE-}) on the left-hand side of \eqref{EntE5} in the above argument, we can  also deduce that
$$
\int_{t}^{t+1}\! \int_{\Omega} F^{(4)}(\phi)|\nabla \phi|^2 \, \d x \d \tau
\leq C_9, \quad \forall \, t \geq 0,
$$
where $C_9$ depends on $\|F''(\phi_0)\|_{L^1(\Omega)}$, $R_0$, $\theta$, $\theta_0$ and $\Omega$.
Since $$ \left(\frac{s}{\sqrt{1-s^2}}\right)'=(1-s^2)^{-\frac32},\quad \forall\, s\in (-1,1),
$$ we infer  that
$$
\int_{t}^{t+1}\! \int_{\Omega} \left| \nabla \left( \frac{\phi}{\sqrt{1-\phi^2}}\right)\right|^2 \, \d x \d \tau  \leq \frac{C_9}{2\theta}, \quad \forall \, t \geq 0.
$$
Setting $\psi= \frac{\phi}{\sqrt{1-\phi^2}}$, and observing that
$$F''(s)= \theta \left[ \left( \frac{s}{\sqrt{1-s^2}}\right)^2 +1\right],$$
we have (cf. \eqref{EntE5})
$$
\| \psi(t)\|_{\L2}^2+ \int_t^{t+1}\, \| \nabla \psi(\tau)\|_{\L2}^2 \, \d \tau \leq C_{10}, \quad \forall \, t \geq 0.
$$
This implies that
$$
\psi \in L^\infty(0,T;\L2)\cap L^2(0,T;H^1(\Omega)),\quad \forall\, T>0.
$$ By the Sobolev embedding theorem in two dimensions, we also have
$$\psi \in L^q(0,T;L^p(\Omega)),\quad  \text{with}\ \ \frac12=\frac{1}{p}+\frac{1}{q}, \ \ \forall\, p \in (2,\infty).
$$
As a consequence, we conclude that
\begin{equation}
\label{F''Lp}
\int_t^{t+1}\! \| F''(\phi(\tau))\|_{L^p(\Omega)}^q \, \d \tau \leq C_{11}, \quad \forall \, t \geq 0,
\end{equation}
 for any $p,q\in (1,\infty)$ satisfying $\frac{1}{p}+\frac{1}{q}=1$.
\medskip

\textbf{Uniqueness of strong solutions.}
Let us consider two strong solutions $(\uu_1,\phi_1,P_1)$ and $(\uu_2,\phi_2.P_2)$ to system \eqref{NSAC-D}-\eqref{IC-D} satisfying the entropy bound \eqref{F''log}  and originating from the same initial datum.
The solutions difference $(\uu,\phi, P):=(\uu_1-\uu_2, \phi_1-\phi_2, P_1-P_2)$ solves
\begin{align}
&\rho(\phi_1)\big( \partial_t \uu + \uu_1 \cdot \nabla \uu + \uu \cdot \nabla \uu_2 \big)- \div \big( \nu(\phi_1)D\uu\big)+ \nabla P \notag \\
&\qquad  = - \Delta \phi_1 \nabla \phi -\Delta \phi \nabla \phi_2 - (\rho(\phi_1)-\rho(\phi_2)) (\partial_t \uu_2 + \uu_2 \cdot \nabla \uu_2) \notag \\
&\qquad\quad  + \div \big( (\nu(\phi_1)-\nu(\phi_2))D\uu_2\big)
\label{D-Diff1}
\end{align}
and
\begin{align}
&\partial_t \phi +\uu_1\cdot \nabla \phi +\uu \cdot \nabla \phi_2
-\Delta \phi + \Psi' (\phi_1)-\Psi'(\phi_2) \notag \\
&\qquad = - \rho'(\phi_1)\frac{|\uu_1|^2}{2}+ \rho'(\phi_2)\frac{|\uu_2|^2}{2}
+\xi_1-\xi_2,
\label{D-Diff2}
\end{align}
for almost every $(x,t) \in \Omega \times (0,T)$, together with the incompressibility constraint $\div \uu=0$.

It easily follows that $\overline{\phi}(t)= 0$ for all $t\geq 0$. Multiplying \eqref{D-Diff1} by $\uu$ and integrating over $\Omega$, we obtain
\begin{align}
&\ddt \int_{\Omega} \frac{\rho(\phi_1)}{2} |\uu|^2 \, \d x
+\int_{\Omega} \nu(\phi_1)|D \uu|^2 \, \d x \notag \\
&\quad =
-\int_{\Omega}  \rho(\phi_1) (\uu_1 \cdot \nabla) \uu \cdot \uu \, \d x
-\int_{\Omega} \rho(\phi_1) (\uu\cdot \nabla )\uu_2 \cdot \uu \, \d x
-\int_{\Omega} \Delta \phi_1 \nabla \phi \cdot \uu \, \d x \notag \\
&\qquad
- \int_{\Omega} \Delta \phi\nabla \phi_2 \cdot \uu \, \d x
- \int_{\Omega} (\rho(\phi_1)-\rho(\phi_2)) (\partial_t \uu_2 + \uu_2 \cdot \nabla \uu_2) \cdot \uu \, \d x  \notag \\
&\qquad - \int_{\Omega} (\nu(\phi_1)-\nu(\phi_2))D\uu_2 : D \uu \, \d x + \int_{\Omega} \frac12 |\uu|^2 \rho'(\phi_1) \partial_t \phi_1 \,\d x.
\label{D1}
\end{align}
Besides, multiplying \eqref{D-Diff2} by $-\Delta \phi$ and integrating over $\Omega$, we find
\begin{align}
&\ddt \int_{\Omega} \frac12 |\nabla \phi|^2 \,\d x + \|\Delta \phi \|_{\L2}^2\notag \\
&\quad  = \int_{\Omega} (\uu_1 \cdot \nabla \phi) \, \Delta \phi \, \d x
+ \int_{\Omega} (\uu\cdot \nabla \phi_2) \, \Delta \phi \,\d x + \int_{\Omega} (F'(\phi_1)-F'(\phi_2)) \Delta \phi \, \d x\notag  \\
&\qquad
+ \theta_0 \|\nabla \phi\|_{\L2}^2 + \int_{\Omega}
\left(  \rho'(\phi_1)\frac{|\uu_1|^2}{2}- \rho'(\phi_2)\frac{|\uu_2|^2}{2} \right) \Delta \phi \, \d x.
\label{D2}
\end{align}
Here we have used the fact that $\overline{\Delta \phi}=0$ thanks to the homogeneous Neumann boundary condition for $\phi_1$, $\phi_2$, which implies that $\int_{\Omega} (\xi_1-\xi_2) \Delta \phi \, \d x=0$.
Adding \eqref{D1} and \eqref{D2}, together with the bound from below of the viscosity function, we deduce that
\begin{align*}
&\ddt \left( \int_{\Omega} \frac{\rho(\phi_1)}{2} |\uu|^2 \, \d x +
\int_{\Omega} \frac12 |\nabla \phi|^2 \,\d x \right)
+ \nu_\ast \|D \uu\|_{\L2}^2
+ \|\Delta \phi \|_{\L2}^2 \\
&\quad \leq  -\int_{\Omega}  \rho(\phi_1) (\uu_1 \cdot \nabla) \uu \cdot \uu \, \d x-\int_{\Omega} \rho(\phi_1) (\uu\cdot \nabla )\uu_2 \cdot \uu \, \d x
-\int_{\Omega} \Delta \phi_1 \nabla \phi \cdot \uu \, \d x  \\
&\qquad - \int_{\Omega} (\rho(\phi_1)-\rho(\phi_2)) (\partial_t \uu_2 + \uu_2 \cdot \nabla \uu_2) \cdot \uu \, \d x - \int_{\Omega} (\nu(\phi_1)-\nu(\phi_2))D\uu_2 : D \uu \, \d x  \\
&\qquad + \int_{\Omega} \frac12 |\uu|^2 \rho'(\phi_1) \partial_t \phi_1 \,\d x +
\int_{\Omega} (\uu_1 \cdot \nabla \phi) \, \Delta \phi \, \d x  + \int_{\Omega} (F'(\phi_1)-F'(\phi_2)) \Delta \phi \, \d x \\
&\qquad
+ \theta_0 \|\nabla \phi\|_{\L2}^2+ \int_{\Omega}
\left(  \rho'(\phi_1)\frac{|\uu_1|^2}{2}- \rho'(\phi_2)\frac{|\uu_2|^2}{2} \right) \Delta \phi \, \d x.
\end{align*}

We now proceed by estimating the terms on the right hand side of the above differential inequality. We would like to mention that most of the bounds obtained below are simple applications of the Sobolev embedding theorem and interpolation inequalities in view of the estimates for global strong solutions that have been obtained before. Nevertheless, less standard is the estimate of $F'(\phi_1)-F'(\phi_2)$ for which we have to make use of the entropy estimate \eqref{EB3}.

To this end, by using the regularity of strong solutions, \eqref{KORN} and \eqref{LADY}, we have
\begin{align*}
-\int_{\Omega}  \rho(\phi_1) (\uu_1 \cdot \nabla) \uu \cdot \uu \, \d x
&\leq C \| \uu_1\|_{L^\infty(\Omega)}\| \nabla \uu\|_{\L2} \| \uu\|_{\L2}\\
&\leq \frac{\nu_\ast}{12} \| D \uu\|_{\L2}^2 +C \| \uu_1\|_{L^\infty(\Omega)}^2\| \uu\|_{\L2}^2,
\end{align*}
\begin{align*}
-\int_{\Omega} \rho(\phi_1) (\uu\cdot \nabla )\uu_2 \cdot \uu \, \d x
& \leq C \| \nabla \uu_2\|_{\L2} \| \uu\|_{L^4(\Omega)}^2\\
& \leq \frac{\nu_\ast}{12}  \|D \uu \|_{\L2}^2+ C\| \uu\|_{\L2}^2,
\end{align*}
\begin{align*}
-\int_{\Omega} \Delta \phi_1 \nabla \phi \cdot \uu \, \d x
& \leq C \|\Delta \phi_1 \|_{L^4(\Omega)} \| \nabla \phi\|_{\L2} \| \uu\|_{L^4(\Omega)}\\
& \leq \frac{\nu_\ast}{12}  \| D \uu\|_{\L2}^2+ C \|\Delta \phi_1 \|_{L^4(\Omega)}^2 \| \nabla \phi\|_{\L2}^2,
\end{align*}
\begin{align*}
&- \int_{\Omega} (\rho(\phi_1)-\rho(\phi_2)) (\partial_t \uu_2 + \uu_2 \cdot \nabla \uu_2) \cdot \uu \, \d x \notag\\
&\quad \leq C \| \phi\|_{L^4(\Omega)} \| \partial_t \uu_2 + \uu_2 \cdot \nabla \uu_2\|_{\L2} \| \uu\|_{L^4(\Omega)}\\
&\quad \leq C \| \nabla \phi\|_{\L2} \| \partial_t \uu_2 + \uu_2 \cdot \nabla \uu_2\|_{\L2} \| \nabla \uu\|_{\L2}\\
&\quad \leq \frac{\nu_\ast}{12}  \| D \uu\|_{\L2}^2 + C \| \partial_t \uu_2 + \uu_2 \cdot \nabla \uu_2\|_{\L2}^2 \| \nabla \phi\|_{\L2}^2,
\end{align*}
\begin{align*}
- \int_{\Omega} (\nu(\phi_1)-\nu(\phi_2))D\uu_2 : D \uu \, \d x
&\leq C \|\phi \|_{L^4(\Omega)} \| D \uu_2\|_{L^4(\Omega)} \| D \uu\|_{\L2}\\
&\leq \frac{\nu_\ast}{12}  \| D \uu\|_{\L2}^2 +
C \| D \uu_2\|_{L^4(\Omega)}^2 \| \nabla \phi\|_{\L2}^2,
\end{align*}
\begin{align*}
\int_{\Omega} \frac12 |\uu|^2 \rho'(\phi_1) \partial_t \phi_1 \,\d x
& \leq C \| \uu\|_{L^4(\Omega)}^2 \| \partial_t \phi_1 \|_{\L2}\\
& \leq \frac{\nu_\ast}{12}  \| D \uu\|_{\L2}^2+ C \| \uu\|_{\L2}^2,
\end{align*}
\begin{align*}
\int_{\Omega} (\uu_1 \cdot \nabla \phi) \, \Delta \phi \, \d x
& \leq \|\uu_1 \|_{L^\infty(\Omega)} \| \nabla \phi\|_{\L2} \| \Delta \phi\|_{\L2}\\
& \leq \frac{1}{6}  \| \Delta \phi\|_{\L2}^2+ C \|\uu_1 \|_{L^\infty(\Omega)}^2 \| \nabla \phi\|_{\L2}^2,
\end{align*}
\begin{align*}
\int_{\Omega}
&\left(  \rho'(\phi_1)\frac{|\uu_1|^2}{2}- \rho'(\phi_2)\frac{|\uu_2|^2}{2} \right) \Delta \phi \, \d x\\
&= \int_{\Omega} \Big( \rho'(\phi_1)-\rho'(\phi_2) \Big) \frac{|\uu_1|^2}{2} \Delta \phi \, \d x + \int_{\Omega} \frac{\rho'(\phi_2)}{2} \Big( \uu_1\cdot \uu+ \uu\cdot \uu_2 \Big) \Delta \phi \, \d x\\
&\leq C \| \phi\|_{L^4(\Omega)}\| \uu_1\|_{L^8(\Omega)}^2 \| \Delta \phi\|_{\L2} + C \| \uu\|_{\L2} (\| \uu_1\|_{L^\infty(\Omega)}+\| \uu_2\|_{L^\infty(\Omega)}) \| \Delta \phi\|_{L^2(\Omega)}\\
&\leq \frac{1}{6}  \| \Delta \phi\|^2_{L^2(\Omega)} + C\| \nabla \phi\|_{\L2}^2+
C (\| \uu_1\|_{L^\infty(\Omega)}^2+\| \uu_2\|_{L^\infty(\Omega)}^2) \| \uu\|_{\L2}^2.
\end{align*}
Next, using the generalized Young's inequality \eqref{Young0} and the standard Young's inequality, for $x>0$, $y>0$, $z>0$ with $Cz>y$, we obtain
\begin{align}
x^2 y^2 \ln \left( \frac{Cz}{y} \right)
&\leq xy^2\left(x\ln x+\frac{Cz}{y}\right) \notag \\
& \leq \varepsilon z^2+  x^2 y^2 \ln x +C^2\varepsilon ^{-1}x^2y^2, \quad \forall \, \varepsilon>0.
\label{ineqy}
\end{align}
Then from \eqref{BGI} and \eqref{ineqy}, we find that
\begin{align*}
& \int_{\Omega} (F'(\phi_1)-F'(\phi_2)) \Delta \phi \, \d x\\
&\quad =  \int_{\Omega} \int_0^1 F''(\tau\phi_1+(1-\tau)\phi_2) \, \d \tau\, \phi \Delta \phi \, \d x\\
&\quad \leq C\left( \| F''(\phi_1)\|_{\L2}+ \| F''(\phi_2)\|_{\L2} \right) \| \phi\|_{L^\infty(\Omega)}\| \Delta \phi\|_{\L2}\\
&\quad \leq C\left( \| F''(\phi_1)\|_{\L2}+ \| F''(\phi_2)\|_{\L2} \right) \| \nabla \phi\|_{\L2} \ln^\frac12 \left( C \frac{\|\Delta \phi \|_{\L2}}{\| \nabla \phi\|_{\L2}} \right) \| \Delta \phi\|_{\L2}\\
&\quad \leq \frac{1}{12} \| \Delta \phi\|_{\L2}^2 +C
\left(  \| F''(\phi_1)\|_{\L2}^2+ \| F''(\phi_2)\|_{\L2}^2 \right) \| \nabla \phi\|_{\L2}^2 \ln \left( C \frac{\|\Delta \phi \|_{\L2}}{\| \nabla \phi\|_{\L2}} \right)\\
 &\quad \leq \frac{1}{6} \| \Delta \phi\|_{\L2}^2 +
C \| F''(\phi_1)\|_{\L2}^2 \left( 1+ \ln \left(  \| F''(\phi_1)\|_{\L2} \right) \right)
\| \nabla \phi\|_{\L2}^2 \\
&\qquad  + C \| F''(\phi_2)\|_{\L2}^2 \left( 1+ \ln \left(  \| F''(\phi_2)\|_{\L2} \right) \right) \| \nabla \phi\|_{\L2}^2 .
\end{align*}

Collecting the above estimates, we deduce the following differential inequality
\begin{align}
&\ddt \left( \int_{\Omega} \frac{\rho(\phi_1)}{2} |\uu|^2 \, \d x +
\int_{\Omega} \frac12 |\nabla \phi|^2 \,\d x \right)
+ \frac{\nu_\ast}{2} \int_{\Omega} |D \uu|^2 \, \d x +\frac12 \|\Delta \phi \|_{\L2}^2 \notag  \\
&\quad \leq W_1(t) \int_{\Omega} \frac{\rho(\phi_1)}{2} |\uu|^2 \, \d x+ W_2(t) \| \nabla \phi\|_{\L2}^2,
\label{D-DI}
\end{align}
where
\begin{align*}
W_1(t)= C \left( 1+ \| \uu_1\|_{L^\infty(\Omega)}^2+  \| \uu_2\|_{L^\infty(\Omega)}^2+ \| \partial_t \uu_2 + \uu_2 \cdot \nabla \uu_2\|_{\L2}^2\right),
\end{align*}
and
\begin{align*}
W_2(t)&=  C \left( 1+ \| \Delta \phi_1\|_{L^4(\Omega)}^2+\| \partial_t \uu_2 + \uu_2 \cdot \nabla \uu_2\|_{\L2}^2+ \| D\uu_2\|_{L^4(\Omega)}^2+
\| \uu_1\|_{L^\infty(\Omega)}^2 \right) \\
 &\quad + C \| F''(\phi_1)\|_{\L2}^2  \ln \left(  \| F''(\phi_1)\|_{\L2} \right) +
 C \| F''(\phi_2)\|_{\L2}^2  \ln \left(  \| F''(\phi_2)\|_{\L2} \right).
\end{align*}
Here we have used that $\rho(s)\geq \rho_\ast$ for all $s \in (-1,1)$.

In order to apply Gronwall's lemma, we are left to show that
\begin{equation}
\label{F''2log}
\int_0^T  \| F''(\phi_i)\|_{\L2}^2  \ln \left( \| F''(\phi_i)\|_{\L2} \right) \, \d \tau\leq C(T), \quad i=1,2.
\end{equation}
To this aim, we introduce the function
$$
g(s)= s \ln ( C^\ast s), \quad \forall \, s\in (0,\infty),
$$
where $C^\ast$ is a positive constant. It is easily seen that $g$ is continuous and convex $\left( g''(s)= \frac{1}{s}>0\right)$.  By applying Jensen's inequality, we have
$$
g \left( \frac{1}{|\Omega|}\int_{\Omega} |F''(\phi)|^2 \, \d x \right)
\leq \frac{1}{|\Omega|} \int_{\Omega} g \left( |F''(\phi)|^2 \right) \, \d x.
$$
Using the explicit form of $g$, this is equivalent to
\begin{align*}
\frac{1}{|\Omega|}\|F''(\phi)\|_{\L2}^2  \ln \left( \frac{C^\ast}{|\Omega|}
\|F''(\phi)\|_{\L2}^2  \right) \leq \frac{1}{|\Omega|}
\int_{\Omega} |F''(\phi)|^2 \ln (C^\ast |F''(\phi)|^2)  \, \d x.
\end{align*}
Taking $C^\ast= |\Omega|$ and integrating the above inequality over $[0,T]$, we find
\begin{equation}
\label{jensen}
\int_0^T \|F''(\phi)\|_{\L2}^2  \ln \big(  \|F''(\phi)\|_{\L2}  \big) \, \d \tau \leq  \int_{0}^T\! \int_{\Omega}  |F''(\phi)|^2 \ln \left(  |\Omega| |F''(\phi)|^2 \right)  \, \d x \d \tau.
\end{equation}
Then, \eqref{F''2log} immediately follows from the entropy bounds \eqref{EB3} and \eqref{jensen}. As a consequence, both $W_1$ and $W_2$ belong to $L^1(0,T)$, for any $T>0$.
Finally, an application of the Gronwall lemma entails the uniqueness of strong solutions.

The proof of Theorem \ref{strong-D} is complete.
\hfill$\square$
\medskip

\begin{remark}[Entropy Estimates in $L^p$, $p>1$]
Notice that the entropy estimate in $L^1(\Omega)$ proved in Theorem \ref{strong-D}-(2) can be generalized to the $L^p(\Omega)$ case with $p>1$. More precisely, for any $p\in \mathbb{N}$, there exists $\eta_p>0$ with the latter depending on the norms of the initial data and on the parameters of the system
$$
\eta_p=\eta_p(E(\uu_0,\phi_0), \| \uu_0\|_{\V_\sigma}, \| \phi_0\|_{H^2(\Omega)},\| F'(\phi_0)\|_{\L2},\theta,\theta_0)
$$
such that, if
$$\|\rho'\|_{L^\infty(-1,1)}\leq \eta_p\quad \text{and}\quad F''(\phi_0)\in L^p(\Omega),$$ then for any $T>0$, we have
\begin{align*}
F''(\phi)\in L^\infty(0,T;L^p(\Omega)),\quad |F''(\phi)|^{p-1}F'''(\phi) F'(\phi)\in L^1(\Omega\times (0,T)).
\end{align*}
These results follow from the above proof by replacing $\ddt \int_{\Omega} F''(\phi) \, \d x$ by $\ddt \int_{\Omega} (F''(\phi))^p \, \d x$, and the observation that, for any $p>2$, there exist two positive constants $C^{(1)}_p$ and $C^{(2)}_p$ such that
$$
|(F''(s))^{p-1}F'''(s)| \ln \big( |(F''(s))^{p-1}F'''(s)| \big) \leq
C^{(1)}_p+C^{(2)}_p (F''(s))^{p-1}F'''(s) F'(s), \quad \forall \, s \in (-1,1).
$$
\end{remark}
\smallskip

\subsection{Proof of Theorem \ref{Proreg-D}}
We now prove the propagation of entropy bound as stated
in Theorem \ref{Proreg-D}.
For every strong solution given by Theorem \ref{strong-D}-(1), we have the following estimates
\begin{align*}
\|\uu\cdot \nabla \phi\|_{H^1(\Omega)}&\leq \|\uu\|_{L^4(\Omega)}\|\nabla \phi\|_{L^4}+\|\nabla \uu\|_{L^4(\Omega)}\|\nabla \phi\|_{L^4(\Omega)}+\|\uu\|_{L^\infty(\Omega)}\|\phi\|_{H^2(\Omega)}\nonumber\\
&\leq C+C\|\uu\|_{H^2(\Omega)}^\frac12\|\phi\|_{H^2(\Omega)}^\frac12+C\|\phi\|_{H^2(\Omega)},
\end{align*}
and
\begin{align*}
\|\rho'(\phi)|\uu|^2\|_{H^1(\Omega)}
&\leq C\|\uu\|_{L^4(\Omega)}^2+C\|\nabla \phi\|_{L^\infty(\Omega)}\|\uu\|_{L^4(\Omega)}^2+C\|\nabla \uu\|_{L^4(\Omega)}\|\uu\|_{L^4(\Omega)}\nonumber\\
&\leq C+C\|\phi\|_{W^{2,3}(\Omega)}+C\|\uu\|_{H^2(\Omega)},
\end{align*}
which imply that
$$
\int_t^{t+1} \| \uu(\tau)\cdot \nabla \phi(\tau)\|_{H^1(\Omega)}^2+ \left\|\rho'(\phi(\tau))\frac{|\uu(\tau)|^2}{2} \right\|_{H^1(\Omega)}^2 \, \d \tau \leq C, \quad \forall \, t \geq 0,
$$
for some positive constant $C$ independent of $t$. In light of \eqref{NSAC4}, it follows that $$
\int_t^{t+1} \|-\Delta \phi(\tau) +F'(\phi(\tau))\|_{H^1(\Omega)}^2 \, \d \tau \leq C, \quad \forall \, t \geq 0.
 $$
By using \cite[Lemma 7.4]{GGW2018}, we infer that, for any $p\geq 1$, there exists some $C=C(p)>0$ such that
\begin{equation}
\label{propF''}
\| F''(\phi)\|_{L^p(\Omega)}\leq  C\left( 1+{e}^{C\| -\Delta \phi +F'(\phi)\|_{H^1(\Omega)}^2} \right) \quad \text{a.e. in}\ (0,T),
\end{equation}
for any $T>0$. Notice that we are not able to conclude that the right-hand side of \eqref{propF''} belongs to $L^1(0,T)$. Nevertheless, since an integrable function is finite almost everywhere, the above inequality entails that there exists some $\zeta \in (0,1)$ (actually $\zeta$ can be taken arbitrarily small but positive) such that
\begin{align}
F''(\phi(\zeta))\in L^p(\Omega) \quad\text{with}\quad \| F''(\phi(\zeta))\|_{L^p(\Omega)}\leq C(p,\zeta),\quad \forall\,p \in [1,\infty).\label{FttLp}
\end{align}
Then, under the condition \eqref{Hyp} but without the additional assumption $F''(\phi_0)\in L^1(\Omega)$ on the initial datum, we are able to deduce that the previous estimates \eqref{EB1}-\eqref{EB3}  hold for all $t\geq \zeta>0$. More precisely, we have
\begin{equation}
\label{EB-sig}
\int_t^{t+1}\! \int_{\Omega} (F''(\phi))^2 \ln ( 1+F''(\phi)) \,\d x \d \tau \leq C(\zeta), \quad \forall \,  t \geq \zeta.
\end{equation}

Differentiating \eqref{NSAC-D}$_1$ with respect to time and testing the resultant by $\partial_t\uu$, integrating over $\Omega$, we have
\begin{align*}
& \frac12 \int_{\Omega} \rho(\phi) \partial_t |\partial_t \uu|^2 \, \d x
+\int_{\Omega} \rho(\phi) \big( \partial_t \uu \cdot \nabla \uu + \uu \cdot \nabla \partial_t \uu\big)\cdot \partial_t \uu \, \d x \\
&\qquad + \int_{\Omega} \rho'(\phi)\partial_t \phi (\partial_t \uu + \uu \cdot \nabla \uu) \cdot \partial_t \uu \, \d x + \int_{\Omega} \nu(\phi) |D \partial_t \uu|^2 \, \d x+ \int_{\Omega} \nu'(\phi) \partial_t \phi D \uu : D \partial_t \uu \, \d x\\
&\quad =  \int_\Omega \partial_t(\nabla \phi\otimes \nabla \phi):\nabla \partial_t\uu \, \d x.
\end{align*}
Since
\begin{align*}
&\frac12  \int_{\Omega} \rho(\phi) \partial_t |\partial_t \uu|^2 \, \d x
= \frac12 \ddt \int_{\Omega} \rho(\phi)|\partial_t \uu|^2 \,\d x - \frac12 \int_{\Omega}  \rho'(\phi)\partial_t \phi |\partial_t \uu|^2\, \d x,
\end{align*}
we find that
\begin{align}
&\frac12\frac{\d}{\d t} \int_{\Omega} \rho(\phi)|\partial_t \uu|^2 \, \d x
 +\int_\Omega \nu(\phi)|D\partial_t \uu|^2 \, \d x\nonumber\\
&\quad =-\int_\Omega \rho(\phi)(\partial_t\uu\cdot \nabla \uu+ \uu\cdot \nabla \partial_t \uu)\cdot \partial_t\uu \, \d x
- \frac12 \int_{\Omega}  \rho'(\phi)\partial_t \phi |\partial_t \uu|^2\, \d x \notag \\
&\qquad -\int_{\Omega} \rho'(\phi)\partial_t \phi (\uu \cdot \nabla \uu) \cdot \partial_t \uu \, \d x - \int_\Omega \nu'(\phi)\partial_t \phi D\uu: \nabla \partial_t\uu \notag \\
&\qquad + \int_\Omega \partial_t(\nabla \phi\otimes \nabla \phi):\nabla \partial_t\uu \, \d x. \label{DDD1}
\end{align}
In view of \eqref{str-1}, by using \eqref{LADY}, we have
\begin{align}
-\int_{\Omega} \rho(\phi)(\partial_t \uu \cdot \nabla \uu) \cdot \partial_t \uu \, \d x &\leq C \| \partial_t \uu\|_{L^4(\Omega)}^2 \|\nabla \uu\|_{\L2} \notag \\
&\leq \frac{\nu_\ast}{16} \| D \partial_t \uu\|_{\L2}^2+ C\| \partial_t \uu\|_{\L2}^2,
\nonumber
\end{align}
and
\begin{align*}
-\int_{\Omega} \rho(\phi)(\uu \cdot \nabla \partial_t \uu) \cdot \partial_t \uu \, \d x &\leq C \| \uu\|_{L^4(\Omega)} \| \nabla \partial_t \uu\|_{\L2} \| \partial_t \uu\|_{L^4(\Omega)}\\
&\leq  \frac{\nu_\ast}{16} \| D \partial_t \uu\|_{\L2}^2+ C\| \partial_t \uu\|_{\L2}^2.
\end{align*}
In a similar manner, we obtain
\begin{align*}
-\frac12 \int_{\Omega} \rho'(\phi)\partial_t \phi |\partial_t \uu|^2\, \d x
&\leq C \| \partial_t \phi\|_{\L2} \| \partial_t \uu \|_{L^4(\Omega)}^2\\
&\leq \frac{\nu_\ast}{16} \| D \partial_t \uu\|_{\L2}^2+ C\| \partial_t \uu\|_{\L2}^2,
\end{align*}
and
\begin{align*}
 -\int_{\Omega} \rho'(\phi)\partial_t \phi (\uu \cdot \nabla \uu) \cdot \partial_t \uu \, \d x
& \leq  C \| \partial_t \phi\|_{L^4(\Omega)} \| \uu \|_{L^4(\Omega)}
\| \nabla \u\|_{L^4(\Omega)} \| \partial_t \uu\|_{L^4(\Omega)}\\
&  \leq C \| \nabla \partial_t \phi\|_{\L2}^\frac12 \| \uu\|_{H^2(\Omega)}^\frac12 \| \partial_t \uu\|_{\L2}^\frac12 \| D \partial_t \uu\|_{\L2}^\frac12\\
&  \leq \frac{\nu_\ast}{16} \| D \partial_t \uu\|_{\L2}^2
+ C\| \partial_t \uu\|_{\L2}^2 +C \| \nabla \partial_t \phi\|_{\L2}^2+C\| \uu\|_{H^2(\Omega)}^2.
\end{align*}
Besides, by means of \eqref{Agmon2d}, we can deduce that
\begin{align*}
 -\int_{\Omega} \nu'(\phi) \partial_t \phi D \uu : D \partial_t \uu \, \d x
& \leq C \| \partial_t \phi\|_{L^\infty(\Omega)} \| D \uu\|_{L^2(\Omega)} \| D \partial_t \uu \|_{\L2}  \\
&  \leq  \frac{\nu_\ast}{16} \| D \partial_t \uu\|_{\L2}^2 +C \| \partial_t \phi\|_{\L2} \|  \partial_t \phi\|_{H^2(\Omega)}   \\
&  \leq  \frac{\nu_\ast}{16} \| D \partial_t \uu\|_{\L2}^2 + \frac{1}{14} \| \Delta \partial_t \phi\|_{\L2}^2 +C \| \nabla \partial_t \phi\|_{\L2} ^2,
\end{align*}
and
\begin{align*}
& \int_\Omega \partial_t(\nabla \phi\otimes \nabla \phi):\nabla \partial_t\uu \, \d x  \\ &\quad \leq
\|\nabla \phi\|_{L^4(\Omega)}\|\nabla \partial_t \phi\|_{L^4(\Omega)}\|D \partial_t\uu\|_{L^2(\Omega)}\\
&\quad  \leq \frac{\nu_\ast}{16} \| D \partial_t \uu\|_{\L2}^2
 +C\|\nabla \partial_t\phi\|_{L^2(\Omega)}\|\nabla \partial_t \phi\|_{H^1(\Omega)}\\
&\quad  \leq  \frac{\nu_\ast}{16} \| D \partial_t \uu\|_{\L2}^2
+ \frac{1}{14} \| \Delta \partial_t \phi\|_{\L2}^2 +C \| \nabla \partial_t \phi\|_{\L2} ^2.
\end{align*}

Next, we differentiate \eqref{NSAC-D}$_3$ with respect to time, multiply the resultant by $-\Delta \partial_t \phi$, and integrate over $\Omega$ to obtain
\begin{align}
& \frac12 \ddt  \| \nabla \partial_t \phi\|_{\L2}^2+ \| \Delta \partial_t \phi\|_{\L2}^2
 \notag \\
&= \theta_0 \|\nabla  \partial_t \phi\|_{\L2}^2+ \int_{\Omega} F''(\phi) \partial_t \phi \Delta \partial_t \phi \, \d x+\int_{\Omega} (\partial_t \uu \cdot \nabla \phi) \Delta \partial_t \phi\, \d x \notag  \\
&\quad +  \int_{\Omega} (\uu \cdot \nabla \partial_t \phi) \Delta \partial_t \phi\, \d x +\frac12 \int_{\Omega} \rho''(\phi) \partial_t \phi |\uu|^2  \Delta \partial_t \phi \, \d x \notag  \\
&\quad + \int_{\Omega} \rho'(\phi)( \uu \cdot \partial_t \uu)  \Delta \partial_t \phi \, \d x.\label{DDD2}
\end{align}
Here we have used that $\overline{\Delta \partial_t \phi}=0$ since $\partial_\n \partial_t \phi=0$ almost everywhere on $\partial \Omega\times(0,T)$.
Exploiting \eqref{BGI}, we get
\begin{align*}
&\int_{\Omega} F''(\phi) \partial_t \phi \Delta \partial_t \phi \, \d x\\
&\quad \leq \| F''(\phi)\|_{\L2} \| \partial_t \phi\|_{L^\infty(\Omega)} \| \Delta \partial_t \phi \|_{\L2}\\
&\quad \leq  \| F''(\phi)\|_{\L2} \| \nabla \partial_t \phi\|_{\L2} \ln^\frac12
\left( C\frac{\| \Delta \partial_t \phi\|_{\L2}}{\| \nabla \partial_t \phi\|_{\L2}} \right)\| \Delta \partial_t \phi \|_{\L2}\\
&\quad \leq \frac{1}{28} \| \Delta \partial_t \phi \|_{\L2}^2 +C \| F''(\phi)\|_{\L2}^2 \| \nabla \partial_t \phi\|_{\L2}^2 \ln \left( C\frac{\| \Delta \partial_t \phi\|_{\L2}}{\| \nabla \partial_t \phi\|_{\L2}}\right).
\end{align*}
Recalling \eqref{ineqy}, we find
\begin{align}
\int_{\Omega} F''(\phi) \partial_t \phi \Delta \partial_t \phi \, \d x  \leq \frac{1}{14} \| \Delta \partial_t \phi \|_{\L2}^2 +C \| F''(\phi)\|_{\L2}^2  \ln \left( C\| F''(\phi)\|_{\L2} \right)\| \nabla \partial_t \phi\|_{\L2}^2.
\label{FFF}
\end{align}
Besides, using \eqref{LADY} and \eqref{str-1}, we get
\begin{align}
\int_{\Omega} (\partial_t \uu \cdot \nabla \phi) \Delta \partial_t \phi\, \d x
&\leq \|\partial_t \uu \|_{L^4(\Omega)}\|\nabla \phi\|_{L^4(\Omega)}\|\Delta \partial_t \phi\|_{\L2}\nonumber\\
&\leq  \frac{\nu_\ast}{16} \| D \partial_t \uu\|_{\L2}^2
+  \frac{1}{14}\| \Delta \partial_t \phi\|_{\L2}^2 + C \|\partial_t \uu \|_{L^2(\Omega)}^2, \nonumber
\end{align}
and
\begin{equation}
\label{unpt}
\begin{split}
\int_{\Omega} (\uu \cdot \nabla \partial_t \phi) \Delta \partial_t \phi\, \d x& \leq \| \uu\|_{L^4(\Omega)} \| \nabla \partial_t \phi\|_{L^4(\Omega)}
\| \Delta \partial_t \phi\|_{\L2} \\
&\leq \frac{1}{14} \| \Delta \partial_t \phi\|_{\L2}^2 + C \| \nabla \partial_t \phi\|_{\L2}^2.
\end{split}
\end{equation}
Finally, in a similar manner, we obtain that
\begin{align*}
\frac12 \int_{\Omega} \rho''(\phi) \partial_t \phi |\uu|^2  \Delta \partial_t \phi \, \d x
&\leq C \| \partial_t \phi\|_{L^4(\Omega)} \| \uu\|_{L^8(\Omega)}^2 \| \Delta \partial_t \phi\|_{\L2}\\
&\leq \frac{1}{14} \| \Delta \partial_t \phi\|_{\L2}^2 + C \| \nabla \partial_t \phi\|_{\L2}^2,
\end{align*}
and
\begin{align*}
 \int_{\Omega} \rho'(\phi)( \uu \cdot \partial_t \uu)  \Delta \partial_t \phi \, \d x
 &\leq C \| \uu\|_{L^4(\Omega)} \| \partial_t \uu \|_{L^4(\Omega)} \|  \Delta \partial_t \phi\|_{\L2}\\
 &\leq \frac{\nu_\ast}{16} \| D \partial_t \uu\|_{\L2}^2
+  \frac{1}{14}\| \Delta \partial_t \phi\|_{\L2}^2 + C \|\partial_t \uu \|_{L^2(\Omega)}^2.
\end{align*}

Combining the above estimates, we deduce from \eqref{DDD1} and \eqref{DDD2} that
\begin{align}
\frac{\d}{\d t}L(t)+ \frac{\nu_*}{2}\|D\partial_t \uu\|_{\L2}^2
+\frac{1}{2}\| \Delta \partial_t \phi\|_{\L2}^2  \leq CK(t)L(t)+ C \| \uu\|_{H^2(\Omega)}^2,
\end{align}
where
\begin{align*}
&L(t)=\frac12 \int_{\Omega} \rho(\phi) |\partial_t\uu(t)|^2 \, \d x+\frac12\|\nabla \partial_t \phi(t)\|_{\L2}^2,\\
&K(t)=1+\| F''(\phi)\|_{\L2}^2  \ln \big( C\| F''(\phi)\|_{\L2} \big).
\end{align*}
Recalling estimates \eqref{NSAC4} and \eqref{str-3}, we have
$$
\int_{t}^{t+1} L(\tau) + \| \uu(\tau)\|_{H^2(\Omega)}^2 \, \d \tau \leq C, \quad \forall \, t \geq 0,
$$
where $C>0$ is independent of $t$.
As a consequence, there exists $\zeta \in (0,1)$ ($\zeta$ can be chosen arbitrary small but positive) such that
\begin{align}
L(\zeta)\leq C(\zeta).\label{Lini}
\end{align}
Notice that, without loss of generality, this value of $\zeta$ can be chosen equal to the one in \eqref{FttLp}.
Then, by exploiting \eqref{EB-sig} and the Jensen inequality (cf. \eqref{jensen}), we obtain
$$
\int_t^{t+1} K(\tau) \, \d \tau \leq C, \quad \forall \, t \geq \zeta,
$$
where $C>0$ depends on $\zeta$, but is independent of $t$. Thus, by using Gronwall's lemma on the time interval $[\zeta,1]$ and the uniform Gronwall lemma for $t\geq 1$, we deduce that
$$
L(t)+ \int_t ^{t+1}  \|D\partial_t \uu\|_{\L2}^2
+\| \Delta \partial_t \phi\|_{\L2}^2 \, \d \tau  \leq C(\zeta),\quad \forall\, t\geq \zeta.
$$
Hence we have
\begin{align*}
& \partial_t\uu\in L^\infty(\zeta, T; \H_\sigma)\cap L^2(\zeta, T; \V_\sigma),\\
& \partial_t \phi\in L^\infty(\zeta, T; H^1(\Omega))\cap L^2(\zeta, T; H^2(\Omega)).
\end{align*}
In light of \eqref{est-uw1p} and \eqref{NSAC3h-D}, we infer that
$$
\uu \in L^\infty(\zeta,T;\mathbf{W}^{1,p}(\Omega)), \quad \forall \, p \in (2,\infty).
$$
An immediate consequence of the above regularity results is that $$\widetilde{\mu}=-\Delta \phi+F'(\phi) \in L^2(\zeta,T;L^\infty(\Omega)).$$  Thanks to \cite[Lemma 7.2]{GGW2018}, we deduce that $F'(\phi)\in L^2(\zeta,T;L^\infty(\Omega))$. This property entails that there exists $\zeta' \in (\zeta,\zeta+1)$ such that
\begin{equation}
\label{F'sigma}
\|F'(\phi(\zeta'))\|_{L^\infty(\Omega)}\leq C(\zeta).
\end{equation}
Note that $\zeta'$ can also be chosen arbitrarily close to $\zeta$.

Now, we rewrite \eqref{NSAC-D} as follows
$$
\partial_t \phi + \uu \cdot \nabla \phi - \Delta \phi +F'(\phi) = U(x,t),
$$
with
$$
U=\theta_0 \phi- \rho'(\phi) \frac{|\uu|^2}{2}+\xi.
$$
Thanks to the above regularity, it easily seen that $U \in L^\infty(0,T;L^\infty(\Omega))$. In particular,
$$\sup_{t\geq \zeta} \| U(t)\|_{L^\infty(\Omega)}\leq C(\zeta).$$
For any $p\geq 2$, we compute
\begin{align*}
\frac{1}{p} \ddt \int_{\Omega} |F'(\phi)|^p \, \d x
&= \int_{\Omega} |F'(\phi)|^{p-2} F'(\phi) F''(\phi) \partial_t \phi \, \d x\\
&= \int_{\Omega} |F'(\phi)|^{p-2} F'(\phi) F''(\phi)
 \left( - \uu\cdot \nabla \phi + \Delta \phi -F'(\phi) + U \right) \, \d x.
\end{align*}
Since
$$
\int_{\Omega}  |F'(\phi)|^{p-2} F'(\phi) F''(\phi) \uu \cdot \nabla \phi \, \d x=
\int_{\Omega} \uu \cdot \nabla \left( \frac{1}{p} |F'(\phi)|^p \right) \, \d x=0,
$$
we deduce that
\begin{align*}
& \frac{1}{p} \ddt \int_{\Omega} |F'(\phi)|^p \, \d x
+ \int_{\Omega} \Big( (p-1) |F'(\phi)|^{p-2} F''(\phi)^2 +
|F'(\phi)|^{p-1} F'(\phi) F'''(\phi) \Big) |\nabla \phi|^2 \, \d x\\
&\quad +\int_{\Omega} |F'(\phi)|^{p} F''(\phi) \, \d x
= \int_{\Omega} |F'(\phi)|^{p-2} F'(\phi) F''(\phi)
 U \, \d x.
\end{align*}
We notice that the second term on the left-hand side of the above inequality is non-negative.
Next, we observe that
$$
F''(s)\leq \theta \mathrm{e}^{\frac{2}{\theta} |F'(s)|}, \quad \forall \, s \in (-1,1).
$$
Owing to the above inequality, and using the fact that $s \leq \mathrm{e}^s$ for $s\geq 0$, we deduce that
\begin{align*}
 \ln \left( |F'(s)|^{p-1} F''(s) \right) \leq \ln (\theta)+ \left(1+\frac{2}{\theta} \right) (p-1) |F'(s)|, \quad \forall \, s \in (-1,1).
\end{align*}
Thus, we get
\begin{equation}
|F'(s)|^{p-1} F''(s) \ln \left( |F'(s)|^{p-1} F''(s) \right) \leq C_1 p |F'(s)|^p F''(s) + C_2, \quad \forall \, s \in (-1,1),
\end{equation}
for some $C_1,C_2>0$ independent of $p$. Recalling the inequality
\begin{equation*}
xy \leq \varepsilon x \ln x + {e}^{\frac{y}{\varepsilon}},\quad \forall \, x>0,\ y>0,\   \varepsilon \in (0,1),
\end{equation*}
and taking $\varepsilon = \frac{1}{2C_1 p}$, we arrive at
\begin{align*}
\frac{1}{p} \ddt  \int_{\Omega} |F'(\phi)|^p \, \d x
+ \frac12 \int_{\Omega} |F'(\phi)|^{p} F''(\phi) \, \d x
\leq  \frac{C_2 |\Omega|}{2C_1}+ \int_{\Omega} {e}^{2C_1 p |U| }\, \d x.
\end{align*}
Since $U$ is globally bounded, we obtain
$$
\frac{C_2 |\Omega|}{2C_1}+ \int_{\Omega} {e}^{2C_1 p |U| }\, \d x
\leq \frac{C_2 |\Omega|}{2C_1} + |\Omega| {e}^{2C_3 p} \leq C_4
{e}^{C_5 p},
$$
for some $C_4,C_5>0$ independent of $p$ and $t$.
Observing that $ F''(s)\geq \theta$ for all $s\in (-1,1)$, we rewrite the above differential inequality for $p\geq 2$ as follows
\begin{equation}
\ddt  \int_{\Omega} |F'(\phi)|^p \, \d x
+\theta \int_{\Omega} |F'(\phi)|^{p} \, \d x
\leq C_4   p  {e}^{C_5 p}.\nonumber
\end{equation}
By applying Gronwall's lemma on the time interval $[\zeta', \infty)$, we infer that
\begin{equation}
\| F'(\phi(t))\|_{L^p(\Omega)}^p \leq \| F'(\phi (\zeta'))\|_{L^p(\Omega)}^p e^{-\theta (t-\zeta')} + \frac{C_4 p {e}^{C_5 p}}{\theta}, \quad \forall \, t \geq \zeta'.
\end{equation}
We recall the elementary inequality for $q\in[0,1]$
$$
(x+y)^q\leq x^q+y^q, \quad \forall \,  x>0,\ y>0.
$$
Choosing $q=\frac{1}{p}$, with $p\geq 2$, we find
\begin{equation}
\label{di-sp}
\| F'(\phi(t))\|_{L^p(\Omega)} \leq \| F'(\phi(\zeta'))\|_{L^p(\Omega)} {e}^{-\frac{\theta (t-\zeta')}{p}} +\left( \frac{C_4 p}{\theta}\right)^{\frac{1}{p}} {e}^{C_5}, \quad \forall \, t \geq \zeta'.\nonumber
\end{equation}
Recalling  \eqref{F'sigma} and taking the limit as $p\rightarrow +\infty$, we then deduce that
\begin{equation}
\label{sp1}
\| F'(\phi(t))\|_{L^\infty(\Omega)} \leq \| F'(\phi (\zeta'))\|_{L^\infty(\Omega)} + {e}^{C_5}, \quad \forall \, t \geq \zeta'. \nonumber
\end{equation}
As a result, there exists $\delta=\delta(\zeta)>0$ such that
$$
-1+\delta \leq \phi(x,t) \leq 1-\delta, \quad \forall \, x \in \overline{\Omega},\ t \geq \zeta'.
$$
The proof of Theorem \ref{Proreg-D} is complete.
\hfill $\square$

\section{Mass-conserving Euler-Allen-Cahn System in Two Dimensions}
\label{EAC-sec}
\setcounter{equation}{0}

In this section, we study the dynamics of ideal two-phase flows in a bounded domain $\Omega \subset \mathbb{R}^2$ with smooth boundary, which is described by the following mass-conserving Euler-Allen-Cahn system:
\begin{equation}
 \label{EAC}
\begin{cases}
\partial_t \uu + \uu \cdot \nabla \uu + \nabla P= - \div(\nabla \phi \otimes \nabla \phi),\\
\div \uu=0,\\
\partial_t \phi +\uu\cdot \nabla \phi + \mu= \overline{\mu}, \\
\mu= -\Delta \phi + \Psi' (\phi),
\end{cases}
\quad \text{ in } \Omega \times (0,T).
\end{equation}
The above system corresponds to the inviscid NSAC system \eqref{NSAC-D} (i.e., $\nu\equiv 0$) with matched densities (i.e., $\rho \equiv 1$). The coupled system \eqref{EAC} is subject to the following boundary conditions
\begin{equation}  \label{boundaryE}
\uu\cdot \n =0,\quad  \partial_{\n} \phi =0 \quad \text{ on } \partial\Omega
\times (0,T),
\end{equation}
and initial conditions
\begin{equation}
\label{ICE}
\uu(\cdot, 0)= \uu_0, \quad \phi(\cdot, 0)=\phi_0 \quad \text{ in } \Omega.
\end{equation}

The main result of this section is as follows:

\begin{theorem}
\label{Th-EAC}
Let $\Omega$ be a bounded domain in $\mathbb{R}^2$ with smooth boundary $\partial\Omega$.
\begin{itemize}
\item[(1)] Assume that $\uu_0 \in \H_\sigma\cap \mathbf{H}^1(\Omega)$, $\phi_0\in H^2(\Omega)$ such that
$F'(\phi_0)\in \L2$, $\| \phi_0\|_{L^\infty(\Omega)}\leq 1$,
$|\overline{\phi}_0|<1$ and $\partial_\n \phi_0=0$ almost everywhere on $\partial \Omega$. Then, there exists a global solution $(\uu,\phi)$ that satisfies the problem \eqref{EAC}-\eqref{ICE} in the sense of distribution on $\Omega \times (0,\infty)$, and for all $T>0$,
\begin{align*}
&\uu \in L^\infty(0,T;\H_\sigma \cap \mathbf{H}^1(\Omega)), \\
& \phi \in L^\infty(0,T; H^2(\Omega))\cap L^2(0,T; W^{2,p}(\Omega)),\\
& \partial_t \phi  \in L^\infty(0,T;L^2(\Omega)) \cap L^2(0,T;H^1(\Omega)),\\
&\phi \in L^\infty(\Omega\times (0,T)) : |\phi(x,t)|<1 \ \ \text{a.e. in} \  \Omega\times(0,T),
\end{align*}
for any $p\in (2,\infty)$. Moreover, $\partial_\n \phi=0$ almost everywhere on $\partial\Omega\times(0,\infty)$, and $\uu|_{t=0}=\uu_0$, $\phi|_{t=0}=\phi_0$ in $\Omega$.

\medskip
\item[(2)] Assume that $
\uu_0 \in \H_\sigma\cap \mathbf{W}^{1,p}(\Omega)$, $p \in (2,\infty)$, $\phi_0\in H^2(\Omega)
$ such that $F'(\phi_0)\in \L2$,
$F''(\phi_0) \in L^1(\Omega)$, $\| \phi_0\|_{L^\infty(\Omega)}\leq 1$,  $ |\overline{\phi}_0|<1$, $\partial_\n \phi_0=0$ almost everywhere on $\partial \Omega$, and in addition $\nabla \mu_0= \nabla ( -\Delta \phi_0+F'(\phi_0)) \in \mathbf{L}^2(\Omega)$. Then, there exists a global solution $(\uu,\phi)$ that satisfies the problem \eqref{EAC}-\eqref{ICE} almost everywhere in $\Omega \times (0,\infty)$, and for all $T>0$,
\begin{align*}
&\uu \in L^\infty(0,T;\H_\sigma \cap \mathbf{W}^{1,p}(\Omega)), \\
&\phi \in L^\infty(0,T;  W^{2,p}(\Omega)),\\
&\partial_t \phi  \in L^\infty(0,T;H^1(\Omega)) \cap L^2(0,T;H^2(\Omega)),\\
&\phi \in L^\infty(\Omega\times (0,T)) : |\phi(x,t)|<1 \ \ \text{a.e. in} \  \Omega\times(0,T).
\end{align*}
for any $p\in (2,\infty)$. Moreover, $\partial_\n \phi=0$ almost everywhere on $\partial\Omega\times(0,\infty)$, and $\uu|_{t=0}=\uu_0$, $\phi|_{t=0}=\phi_0$ in $\Omega$. In addition, for any $\zeta>0$, there exists $\delta=\delta(\zeta)>0$ such that
$$
-1+\delta \leq \phi(x,t) \leq 1-\delta, \quad \forall \, x \in \overline{\Omega}, \ t \geq \zeta.
$$
\end{itemize}
\end{theorem}
\begin{proof} To prove Theorem \ref{Th-EAC}, we derive formal \textit{a priori} estimates leading to the required estimates of solutions. Then the existence results can be proved by a suitable approximation scheme together with the fixed point argument and then passing to the limit, which is standard owing to uniform estimates obtained in the first step. Hence, below we only focus on the \textit{a priori} estimates and omit further details.
\smallskip

\textbf{Case 1.}
Let us first consider initial datum $(\uu_0,\phi_0)$ such that
$$
\uu_0 \in \H_\sigma\cap \mathbf{H}^1(\Omega), \quad \phi_0\in H^2(\Omega), \quad \partial_\n \phi_0=0\ \ \text{a.e. on}\ \partial\Omega,
$$
with
$$
\| \phi_0\|_{L^\infty(\Omega)}\leq 1, \quad  |\overline{\phi}_0|<1 \quad \text{and} \quad F'(\phi_0)\in L^2(\Omega).
$$

\textbf{Lower-order estimate.} As in the previous section, we have the conservation of mass
$$
\overline{\phi}(t)= \overline{\phi}_0, \quad \forall \, t \geq 0.
$$
By the same argument for \eqref{BEL-D}, we can deduce the energy equality
\begin{equation}
\label{EE2}
\ddt E(\uu, \phi) + \|\partial_t \phi + \uu \cdot \nabla \phi\|_{L^2(\Omega)}^2=0.
\end{equation}
Integrating the above relation on $[0,t]$, we find
$$
 E(\uu(t),\phi(t))+  \int_0^t \| \partial_t \phi + \uu \cdot \nabla \phi\|_{\L2}^2 \, \d \tau = E(\uu_0,\phi_0), \quad \forall \, t \geq 0.
$$
This implies that for all $T>0$,
\begin{equation}
\label{E1}
\uu \in L^\infty(0,T; \H_\sigma), \quad \phi\in L^\infty(0,T;H^1(\Omega)), \quad
 \partial_t \phi + \uu \cdot \nabla \phi \in L^2(0,T;L^2(\Omega)),
\end{equation}
where the last property also implies $\mu-\overline{\mu}\in L^2(0,T;L^2(\Omega))$. In addition, it follows from the estimates \eqref{H2-D} and \eqref{mubar} that for all $T>0$
\begin{align*}
\phi \in L^2(0,T;H^2(\Omega)),\quad \mu\in L^2(0,T; L^2(\Omega)), \quad F'(\phi)\in L^2(0,T;L^2(\Omega)).
\end{align*}
Like before, the singularity of $F'$ entails that $\phi \in L^\infty(\Omega \times (0,T))$ with
$|\phi(x,t)|<1$ almost everywhere in $\Omega\times(0,T)$.
  This fact combined with $\phi \in L^2(0,T;H^2(\Omega))$ implies that $\phi \in L^\infty(0,T;L^\infty(\Omega))$ and $$\|\phi(t)\|_{L^\infty(\Omega)}\leq 1,\quad \text{for a.a.}\ t\in [0,T].$$
In comparison with the viscous case, at this stage it is not possible to prove that $\partial_t \phi\in L^2(0,T;L^2(\Omega))$.
\medskip

\textbf{Higher-order estimates.}
In the two dimensional case, it is convenient to consider the equation for the vorticity $\omega= \partial_{x_1} u_2-\partial_{x_2} u_1$ that reads as follows
\begin{equation}
\label{vort-eq}
\partial_t \omega+ \uu \cdot \nabla \omega = \nabla \mu \cdot (\nabla \phi)^\perp,
\end{equation}
where $\vv^\perp= (v_2,-v_1)$ for any two dimensional vector $\vv=(v_1,v_2)$. Multiplying \eqref{vort-eq} by $\omega$ and integrating over $\Omega$, we obtain
\begin{equation}
\label{Vor1}
\frac12 \ddt \| \omega\|_{\L2}^2= \int_{\Omega} \nabla \mu \cdot (\nabla \phi)^\perp \, \omega \, \d x.
\end{equation}
On the other hand, differentiating \eqref{EAC}$_3$ with respect to time, multiplying the resultant by $\partial_t \phi$ and integrating over $\Omega$, we find
\begin{align}
\frac12 \ddt \| \partial_t \phi\|_{\L2}^2+ \int_{\Omega} \partial_t \uu \cdot \nabla \phi \, \partial_t \phi \, \d x  + \| \nabla \partial_t \phi\|_{\L2}^2
+\int_{\Omega} F''(\phi) |\partial_t \phi|^2 \, \d x= \theta_0 \| \partial_t \phi\|_{\L2}^2.
\label{TestAC2}
\end{align}
Here we have used the following equalities
$$
\int_{\Omega} \uu\cdot \nabla \partial_t \phi \, \partial_t \phi \, \d x=
\int_{\Omega} \uu\cdot \nabla \left( \frac12 |\partial_t \phi|^2 \right) \, \d x=0 \quad
\text{and}\quad
\int_{\Omega} \partial_t \phi \, \d x=0.
$$
Define
$$
H(t) = \frac12  \| \omega\|_{\L2}^2+ \frac12 \| \partial_t \phi\|_{\L2}^2.
$$
Then adding together \eqref{Vor1} and \eqref{TestAC2}, we infer from the convexity of $F$ (i.e., $F''>0$) that
\begin{equation}
\label{In-EAC}
\ddt H(t) +  \| \nabla \partial_t \phi\|_{\L2}^2 \leq  \int_{\Omega} \nabla \mu \cdot (\nabla \phi)^\perp \, \omega \, \d x - \int_{\Omega} \partial_t \uu \cdot \nabla \phi \, \partial_t \phi \, \d x +  \theta_0 \| \partial_t \phi\|_{\L2}^2.
\end{equation}

Before proceeding to control the terms on the right-hand side of \eqref{In-EAC}, we rewrite the second one using the Euler equation.
We first observe that
$$
\partial_t \uu = \P \big( \mu \nabla \phi - \uu \cdot \nabla \uu \big),
$$
where $\mathbb{P}$ is the Leray projection operator.
Thus, we have
\begin{align*}
\int_{\Omega} \partial_t \uu \cdot \nabla \phi \, \partial_t \phi \, \d x
&= \int_{\Omega}  \P \big( \mu  \nabla \phi - \uu \cdot \nabla \uu \big) \cdot \nabla \phi \, \partial_t \phi \, \d x \\
&= \int_{\Omega}  \mu  \nabla \phi \cdot \P \big( \nabla \phi \, \partial_t \phi\big) \, \d x - \int_{\Omega} (\uu \cdot \nabla \uu)  \cdot \P \big( \nabla \phi \, \partial_t \phi\big) \, \d x \\
&= -\int_{\Omega}  \mu  \nabla \phi \cdot \P \big( \phi \,\nabla \partial_t \phi\big) \, \d x - \int_{\Omega} \div ( \uu \otimes \uu ) \cdot \P \big( \nabla \phi \, \partial_t \phi\big) \, \d x \\
&= -\int_{\Omega}  \mu  \nabla \phi \cdot \P \big( \phi \,\nabla \partial_t \phi\big) \, \d x + \int_{\Omega}  (\uu \otimes \uu)  : \nabla \P \big( \nabla \phi \, \partial_t \phi\big) \, \d x \\
&\quad - \int_{\partial \Omega} (\uu\otimes \uu) \P\big( \nabla \phi \, \partial_t \phi\big) \cdot \n \, \d S \\
&= -\int_{\Omega}  \mu  \nabla \phi \cdot \P \big( \phi \,\nabla \partial_t \phi\big) \, \d x + \int_{\Omega}  (\uu \otimes \uu)  : \nabla \P \big( \nabla \phi \, \partial_t \phi\big) \, \d x \\
&\quad - \int_{\partial \Omega} (\uu \cdot \n) \big(\uu \cdot \P\big( \nabla \phi \, \partial_t \phi\big) \big) \, \d S \\
&= -\int_{\Omega}  \mu  \nabla \phi \cdot \P \big( \phi \,\nabla \partial_t \phi\big) \, \d x + \int_{\Omega}  (\uu \otimes \uu)  : \nabla \P \big( \nabla \phi \, \partial_t \phi\big) \, \d x.
\end{align*}
Here we have used that $\P( \nabla v)=0$ for any $v \in H^1(\Omega)$, the relation $\div (S^t \vv)= S^t : \nabla \vv+ \div S \cdot \vv$ for any $d \times d$ tensor $S$ and vector $\vv$, and the no-normal flow condition $\uu \cdot \n =0$ at the boundary $\partial\Omega$.
As a consequence, we rewrite \eqref{In-EAC} as follows
\begin{equation}
\label{In-EAC2}
\begin{split}
\ddt H(t) +  \| \nabla \partial_t \phi\|_{\L2}^2
 & \leq  \int_{\Omega} \nabla \mu \cdot (\nabla \phi)^\perp \, \omega \, \d x + \int_{\Omega} \mu \nabla \phi \cdot \P \big(  \phi \, \nabla \partial_t \phi\big) \, \d x  \\
&\quad -  \int_{\Omega}  (\uu \otimes \uu)  : \nabla \P \big( \nabla \phi \, \partial_t \phi\big) \, \d x +  \theta_0 \| \partial_t \phi\|_{\L2}^2.
\end{split}
\end{equation}

We now turn to estimate the right-hand side of \eqref{In-EAC2}.
By H\"{o}lder's inequality, we have
\begin{equation}
\label{I1}
\int_{\Omega} \nabla \mu \cdot (\nabla \phi)^\perp \, \omega \, \d x
\leq \| \nabla \mu\|_{\L2} \| \nabla \phi\|_{L^\infty(\Omega)} \| \omega\|_{\L2}.
\end{equation}
By taking the gradient of \eqref{EAC}$_3$, we observe that
$$
\| \nabla \mu\|_{\L2}
\leq \| \nabla \partial_t \phi\|_{\L2}+ \|\nabla^2 \phi\,  \uu \|_{\L2} + \| \nabla \uu \, \nabla \phi\|_{\L2}.
$$
Recalling the elementary inequality
$$
\| \vv\|_{H^1(\Omega)}\leq C \left(\| \vv\|_{\L2} +
\| \div \vv\|_{\L2} + \| \curl \vv\|_{\L2}+ \| \vv\cdot \n\|_{H^\frac12(\partial \Omega)}\right), \quad \forall \, \vv \in \mathbf{H}^1(\Omega),
$$
and exploiting Lemma \ref{result1} as well as \eqref{BGW}, we find that
\begin{align*}
\| \nabla \mu\|_{\L2}
&\leq \| \nabla \partial_t \phi\|_{\L2}
+ C \| \uu\|_{H^1(\Omega)} \| \nabla^2 \phi\|_{\L2} \ln^\frac12
\left( C \frac{\| \nabla^2 \phi\|_{L^{p}(\Omega)}}{\| \nabla^2 \phi\|_{\L2}} \right) \\
&\quad
+ \| \nabla \uu \|_{\L2} \|\nabla \phi\|_{L^\infty(\Omega)}\\
&\leq \| \nabla \partial_t \phi\|_{\L2}
+C (1+ \| \omega\|_{\L2})  \| \nabla^2 \phi\|_{\L2} \ln^\frac12
\left( C \frac{\| \nabla^2 \phi\|_{L^{p}(\Omega)}}{\| \nabla^2 \phi\|_{\L2}} \right)\\
&\quad + C (1+ \| \omega\|_{\L2})
\|\nabla \phi\|_{H^1(\Omega)}\ln^\frac12
\left( C \frac{\| \nabla \phi\|_{W^{1,p}(\Omega)}}{\| \nabla \phi\|_{H^1(\Omega)}} \right),
\end{align*}
for some $p>2$. Using \eqref{ineq0}, we rewrite the above estimate as follows
\begin{align*}
\| \nabla \mu\|_{\L2}
&\leq \| \nabla \partial_t \phi\|_{\L2}
+C \left(1+ \| \omega\|_{\L2}\right)  \left( \| \nabla \phi\|_{H^1(\Omega)}
\ln^\frac12 \left( C \|\nabla \phi\|_{W^{1,p}(\Omega)} \right) + 1 \right).
\end{align*}
Then, using again the inequality \eqref{BGW}, \eqref{I1} can be controlled as follows
\begin{align*}
&\int_{\Omega} \nabla \mu \cdot (\nabla \phi)^\perp \, \omega \, \d x\\
&\quad \leq \| \nabla \partial_t \phi\|_{\L2} \| \omega\|_{\L2} \|\nabla \phi\|_{H^1(\Omega)}\ln^\frac12
\left( C \frac{\| \nabla \phi\|_{W^{1,p}(\Omega)}}{\| \nabla \phi\|_{H^1(\Omega)}} \right) \\
&\qquad + C \| \omega\|_{\L2} \left( 1+ \| \omega\|_{\L2} \right)  \Big( \|\nabla  \phi\|_{H^1(\Omega)}
\ln^\frac12 \big( C \| \nabla \phi\|_{W^{1,p}(\Omega)} \big) + 1 \Big)\\
&\qquad\quad  \times
\|\nabla \phi\|_{H^1(\Omega)}\ln^\frac12
\left( C \frac{\| \nabla \phi\|_{W^{1,p}(\Omega)}}{\| \nabla \phi\|_{H^1(\Omega)}} \right)\\
&\quad \leq \| \nabla \partial_t \phi\|_{\L2} \| \omega\|_{\L2} \Big( \|\nabla  \phi\|_{H^1(\Omega)}
\ln^\frac12 \big( C \| \nabla \phi\|_{W^{1,p}(\Omega)}\big) +1\Big)\\
&\qquad + C \left( 1+\| \omega\|_{\L2}^2 \right) \Big( \| \nabla \phi\|_{H^1(\Omega)}^2
\ln  \big( C \| \nabla \phi\|_{W^{1,p}(\Omega)} \big) + 1 \Big)\\
&\quad \leq \frac16  \| \nabla \partial_t \phi\|_{\L2}^2+
C \left( 1+\| \omega\|_{\L2}^2 \right) \Big( \| \phi\|_{H^2(\Omega)}^2
\ln  \big( C \| \phi\|_{W^{2,p}(\Omega)} \big) + 1 \Big),
\end{align*}
for some $p>2$.
Next, since $\phi$ is globally bounded, we have
\begin{align*}
& \int_{\Omega} \mu \nabla \phi \cdot \P \big(  \phi \, \nabla \partial_t \phi\big) \, \d x\\
&\quad \leq C \| \mu\|_{\L2} \| \nabla \phi\|_{L^\infty(\Omega)} \| \phi \nabla \partial_t \phi \|_{\L2}\\
&\quad \leq C  \| \mu\|_{\L2} \|\nabla \phi\|_{H^1(\Omega)}\ln^\frac12
\left( C \frac{\| \nabla \phi\|_{W^{1,p}(\Omega)}}{\| \nabla \phi\|_{H^1(\Omega)}} \right) \| \phi\|_{L^\infty(\Omega)} \| \nabla \partial_t \phi\|_{\L2}\\
&\quad \leq C \| \mu\|_{\L2} \Big( \| \phi\|_{H^2(\Omega)}
\ln^\frac12 \big( C \| \phi\|_{W^{2,p}(\Omega)} \big) + 1 \Big) \| \nabla \partial_t \phi\|_{\L2},
\end{align*}
for some $p>2$. In order to estimate the $L^2$-norm of $\mu$, we notice that
\begin{align*}
\| \mu-\overline{\mu}\|_{L^2(\Omega)}
&\leq \| \partial_t \phi\|_{\L2}+ \| \uu\cdot \nabla \phi\|_{\L2}\\
&\leq  \| \partial_t \phi\|_{\L2}+ \|\uu \|_{L^4(\Omega)} \| \nabla \phi\|_{L^4(\Omega)} \\
&\leq  \| \partial_t \phi\|_{\L2}+ C \|\uu \|_{\L2}^\frac12 \| \u\|_{H^1(\Omega)}^\frac12 \| \nabla \phi\|_{\L2}^\frac12 \| \phi\|_{H^2(\Omega)}^\frac12 \\
&\leq \| \partial_t \phi\|_{\L2}+ C \left( 1+\| \omega\|_{\L2} \right)^\frac12 \left(1+\| \mu-\overline{\mu}\|_{\L2} \right)^\frac12 \nonumber\\
&\leq \| \partial_t \phi\|_{\L2}+ C \left( 1+\| \omega\|_{\L2}\right) +\frac12 \| \mu-\overline{\mu}\|_{L^2(\Omega)}.
\end{align*}
Here we have used the equation \eqref{EAC}$_3$, the Ladyzhenskaya inequality, and the estimates \eqref{H2-D}, \eqref{E1}. Since $\|\mu\|_{\L2}\leq C(1+\| \mu-\overline{\mu}\|_{\L2})$ (recalling \eqref{mubar}), we then infer that
$$
\| \mu\|_{\L2}\leq C \left( 1+ \| \partial_t \phi\|_{\L2}+ \| \omega\|_{\L2} \right).
$$
Thus, we can deduce that
\begin{align*}
\int_{\Omega} &\mu \nabla \phi \cdot \P \big(  \phi \, \nabla \partial_t \phi\big) \, \d x\\
&\leq \frac16 \| \nabla \partial_t \phi\|_{\L2}^2 + C \left(1+ \| \partial_t \phi\|_{\L2}^2 + \| \omega\|_{\L2}^2  \right) \left(  \| \phi\|_{H^2(\Omega)}^2
\ln  \left( C \| \phi\|_{W^{2,p}(\Omega)} \right) + 1 \right).
\end{align*}
Recalling that $\P$ is a bounded operator from $\mathbf{H}^1(\Omega)$ to $\H_\sigma \cap \mathbf{H}^1(\Omega)$, and using the inequalities \eqref{LADY} and \eqref{BGW}, the Poincar\'{e} inequality, and Lemma \ref{result1}, we have
\begin{align*}
& -\int_{\Omega}  (\uu \otimes \uu)  : \nabla \P \big( \nabla \phi \, \partial_t \phi\big) \, \d x\\
&\quad  \leq \| \uu\|_{L^4(\Omega)}^2 \|\P (\nabla \phi \, \partial_t \phi) \|_{H^1(\Omega)}\\
&\quad \leq C \| \uu\|_{\L2} \| \uu\|_{H^1(\Omega)} \| \nabla \phi \, \partial_t \phi\|_{H^1(\Omega)}\\
&\quad \leq C \left(1+\| \omega\|_{\L2}\right) \Big( \| \nabla \phi \, \partial_t \phi\|_{\L2} +
\| \nabla^2 \phi  \,\partial_t \phi \|_{\L2}+ \| \nabla \phi \, \nabla \partial_t \phi\|_{\L2} \Big) \\
&\quad \leq C \left(1+\| \omega\|_{\L2}\right) \Bigg[ \| \nabla \phi\|_{L^\infty(\Omega)} \| \nabla \partial_t \phi\|_{\L2}\\
&\qquad \quad + \| \nabla \partial_t \phi\|_{\L2} \| \nabla^2 \phi\|_{\L2} \ln^\frac12
\left( C \frac{\| \nabla^2 \phi\|_{L^{p}(\Omega)}}{\| \nabla^2 \phi\|_{\L2}} \right) \Bigg] \\
&\quad \leq C \left(1+\| \omega\|_{\L2}\right)  \| \nabla \partial_t \phi\|_{\L2}
\Big( \|\nabla \phi\|_{H^1(\Omega)}\ln^\frac12
\left( C \frac{\| \nabla \phi\|_{W^{1,p}(\Omega)}}{\| \nabla \phi\|_{H^1(\Omega)}} \right)\\
&\qquad +  \| \nabla^2 \phi\|_{\L2} \ln^\frac12
\left( C \frac{\| \nabla^2 \phi\|_{L^{p}(\Omega)}}{\| \nabla^2 \phi\|_{\L2}} \Big) \right)\\
&\quad \leq C \left(1+\| \omega\|_{\L2}\right)  \| \nabla \partial_t \phi\|_{\L2}
\Big(  \| \phi\|_{H^2(\Omega)} \ln^\frac12\big( C\| \phi\|_{W^{2,p}(\Omega)}\big)+1 \Big)\\
&\quad \leq \frac16 \| \nabla \partial_t \phi\|_{\L2}^2+
C\left(1+\|\omega \|_{\L2}^2\right) \Big( \| \phi\|_{H^2(\Omega)}^2 \ln\big( C\| \phi\|_{W^{2,p}(\Omega)}\big)+1 \Big),
\end{align*}
for some $p>2$.

Combining the above estimates together with \eqref{In-EAC2}, we arrive at the differential inequality
\begin{align}
\label{In-EAC3}
\ddt H(t) +  \frac12 \| \nabla \partial_t \phi\|_{\L2}^2
\leq C (1+H(t)) \Big( \| \phi\|_{H^2(\Omega)}^2 \ln\big( C \| \phi\|_{W^{2,p}(\Omega)}\big)+1 \Big).
\end{align}
In order to close the estimate, we are left to absorb the logarithmic term on the right-hand side of the above differential inequality. To this aim, we first multiply $\mu=-\Delta \phi+\Psi'(\phi)$ by $|F'(\phi)|^{p-2}F'(\phi)$, for some $p>2$, and integrate over $\Omega$. After integrating by parts and using the boundary condition for $\phi$, we obtain
$$
\int_{\Omega} (p-1)|F'(\phi)|^{p-2} F''(\phi) |\nabla \phi|^2 \, \d x+
\| F'(\phi)\|_{L^p(\Omega)}^p= \int_{\Omega} (\mu +\theta_0 \phi)|F'(\phi)|^{p-2}F'(\phi) \, \d x.
$$
By Young's inequality and the fact that $F''>0$, we deduce
$$
\| F'(\phi)\|_{L^p(\Omega)} \leq C \left( 1+ \|\mu\|_{L^p(\Omega)} \right).
$$
Using the elliptic regularity, together with the above inequality and \eqref{E1}, we obtain that (cf. \eqref{pw2p})
$$
\| \phi\|_{W^{2,p}(\Omega)} \leq C \left(1+\| \mu\|_{L^p(\Omega)} \right).
$$
On the other hand, we infer from equation \eqref{EAC}$_3$ that
\begin{align*}
 \|\mu-\overline{\mu} \|_{L^p(\Omega)}
 \leq \| \partial_t \phi\|_{L^p(\Omega)} + \| \uu \cdot \nabla \phi\|_{L^p(\Omega)}.
\end{align*}
Then by the Poincar\'{e} inequality and the Sobolev embedding theorem, we find
\begin{align*}
\| \mu\|_{L^p(\Omega)}
&\leq C \|\mu-\overline{\mu} \|_{L^p(\Omega)} +C |\overline{\mu}|\\
&\leq C \| \nabla \partial_t \phi\|_{\L2}
 +C  \| \uu\|_{H^1(\Omega)} \| \phi\|_{H^2(\Omega)}
 + C \left( 1+\| \mu-\overline{\mu}\|_{\L2} \right)\\
 &\leq C \| \nabla \partial_t \phi\|_{\L2}
 +C \left( 1+\| \omega\|_{\L2} \right) \left(1+ \| \mu-\overline{\mu}\|_{\L2} \right) \\
 &\leq C \| \nabla \partial_t \phi\|_{\L2}
 +C \left(1+\| \omega\|_{\L2}\right) \left(1+ \| \partial_t \phi\|_{\L2}+ \|\omega\|_{\L2}\right).
\end{align*}
Thus, for $p>2$, we reach
$$
\| \phi\|_{W^{2,p}(\Omega)} \leq C \left( 1+ \| \nabla \partial_t \phi\|_{\L2} +
H(t) \right),
$$
which, in turn, allows us to rewrite \eqref{In-EAC3} as
\begin{align}
& \ddt H(t) +  \frac12 \| \nabla \partial_t \phi\|_{\L2}^2\notag\\
&\quad \leq C (1+H(t)) \Big( \| \phi\|_{H^2(\Omega)}^2 \ln\Big( C\big( 1+ \| \nabla \partial_t \phi\|_{\L2} +
H(t)\big)\Big)+1 \Big).
\label{In-EAC4}
\end{align}
We now observe that, for any $\varepsilon>0$, the following inequality holds
$$
x\ln(C y)\leq \varepsilon y+ x\ln\Big(\frac{ C x}{\varepsilon}\Big) \quad \forall \, x, y >0.
$$
By using the above inequality with $x=1+H$, $y= 1+ \| \nabla \partial_t \phi\|_{\L2} +H $ and $\varepsilon=1$, we deduce that
\begin{align*}
 &\ddt H(t) +  \frac12 \| \nabla \partial_t \phi\|_{\L2}^2  \\
  &\quad \leq  \| \nabla \partial_t \phi\|_{\L2} \| \phi\|_{H^2(\Omega)}^2
  + C ( 1+\| \phi\|_{H^2(\Omega)}^2 ) (1+H(t)) \ln \big( C (1+H(t))\big).
\end{align*}
By Young's inequality, we obtain
\begin{align*}
\ddt H(t) +  \frac14 \| \nabla \partial_t \phi\|_{\L2}^2
\leq   \| \phi\|_{H^2(\Omega)}^4
  + C ( 1+\| \phi\|_{H^2(\Omega)}^2 ) (1+H(t)) \ln \big( C (1+H(t))\big).
\end{align*}
Recalling that $\| \phi\|_{H^2(\Omega)}^2 \leq C(1+H(t))$, we are finally led to the differential inequality
\begin{equation}
\label{In-EAC5}
\ddt H(t) +  \frac14 \| \nabla \partial_t \phi\|_{\L2}^2
\leq   C ( 1+\| \phi\|_{H^2(\Omega)}^2 ) (1+H(t)) \ln \big( C (1+H(t))\big).
\end{equation}
Since $\phi \in L^2(0,T;H^2(\Omega))$, then applying the generalized Gronwall lemma \ref{GL2}, we find the double exponential bound
\begin{align*}
\sup_{t \in [0,T]} &\Big( \|\partial_t \phi (t)\|_{\L2}^2+ \|\omega(t) \|_{\L2}^2 \Big)\\
&\leq C \left(1+  \| \uu_0\|_{H^1(\Omega)}^2 \| \phi_0\|_{H^2(\Omega)}^2
+\| \phi_0\|_{H^2(\Omega)}^2+ \| \Psi'(\phi_0)\|_{\L2}^2+ \| \uu_0\|_{H^1(\Omega)}^2 \right)^{{e}^{\int_0^T 1+\| \phi(s)\|_{H^2(\Omega)}^2 \, \d s}},
\end{align*}
for some constant $C>0$.
Here we have used that
$$
\| \partial_t \phi (0)\|_{\L2} \leq C \| \uu_0\|_{H^1(\Omega)} \| \phi_0\|_{H^2(\Omega)}
+C \| \phi_0\|_{H^2(\Omega)}+C \| \Psi'(\phi_0)\|_{\L2}.
$$
Hence, we get
\begin{equation}
\label{Reg-E1}
\partial_t \phi \in L^\infty(0,T;L^2(\Omega)) \cap L^2(0,T;H^1(\Omega)), \quad
\omega \in L^\infty(0,T;L^2(\Omega)), \quad \forall \, T>0,
\end{equation}
which, in turn, entail that
\begin{equation}
\label{Reg-E2}
\uu \in L^\infty(0,T;\mathbf{H}^1(\Omega)), \quad \phi \in L^\infty(0,T; H^2(\Omega))\cap L^2(0,T; W^{2,p}(\Omega)),  \quad \forall \, T>0,
\end{equation}
for any $p \in [2,\infty)$.
\medskip

\textbf{Case 2.}
We now consider an initial condition $(\uu_0,\phi_0)$ such that
$$
\uu_0 \in \H_\sigma\cap \mathbf{W}^{1,p}(\Omega) , \quad \phi_0\in H^2(\Omega),\quad\partial_\n \phi_0=0\ \ \text{a.e on}\ \partial\Omega,
$$
for $p \in (2,\infty)$, with $\| \phi_0\|_{L^\infty(\Omega)}\leq 1$, $|\overline{\phi}_0|<1$ and
$$
F'(\phi_0)\in \L2, \quad  F''(\phi_0)\in L^1(\Omega), \quad \nabla \mu_0= \nabla ( -\Delta \phi_0+F'(\phi_0)) \in L^2(\Omega).
$$
Thanks to the first part of Theorem \ref{Th-EAC}, we have a solution $(\uu,\phi)$ satisfying \eqref{Reg-E1} and \eqref{Reg-E2}. Moreover, repeating the same argument performed in Section \ref{S-STRONG}, we have (cf. \eqref{EntE4})
$$
\ddt \int_{\Omega} F''(\phi) \, \d x + \frac14 \int_{\Omega} F'''(\phi) F'(\phi) \, \d x
\leq C,
$$
for some positive constant $C$ only depending on $\Omega$ and the parameters of the system. Since $F''(\phi_0)\in L^1(\Omega)$, we learn, in particular, that (cf. \eqref{EB3})
\begin{equation}
\int_t^{t+1}\! \int_{\Omega} |F''(\phi)|^2 \ln ( 1+F''(\phi)) \,\d x \, \d \tau \leq
C, \quad \forall \, t \geq 0.
\end{equation}
Multiplying \eqref{vort-eq} by $|\omega|^{p-2}\omega$ ($p>2$) and integrating over $\Omega$, we obtain
$$
\frac{1}{p} \ddt \|\omega\|_{L^p(\Omega)}^p = \int_{\Omega} \nabla \mu \cdot (\nabla \phi)^\perp |\omega|^{p-2}\omega \, \d x.
$$
By H\"{o}lder's inequality, we easily get
$$
\frac{1}{p} \ddt  \|\omega\|_{L^p(\Omega)}^p \leq \| \nabla \mu \cdot (\nabla \phi)^\perp\|_{L^p(\Omega)} \|\omega\|_{L^p(\Omega)}^{p-1},
$$
which, in turn, implies
$$
\frac12 \ddt  \|\omega\|_{L^p(\Omega)}^2 \leq  \| \nabla \mu \cdot (\nabla \phi)^\perp\|_{L^p(\Omega)} \| \omega\|_{L^p(\Omega)}.
$$
Next, differentiating \eqref{EAC}$_3$ with respect time, then multiplying the resultant by $-\Delta \partial_t \phi$ and integrating over $\Omega$, we obtain
\begin{align*}
& \frac12 \ddt  \| \nabla \partial_t \phi\|_{\L2}^2+ \| \Delta \partial_t \phi\|_{\L2}^2\\
&\quad  = \theta_0 \| \nabla \partial_t \phi\|_{\L2}^2+ \int_{\Omega} F''(\phi) \partial_t \phi \Delta \partial_t \phi \, \d x+\int_{\Omega} (\partial_t \uu \cdot \nabla \phi) \Delta \partial_t \phi\, \d x \\
&\qquad +  \int_{\Omega} (\uu \cdot \nabla \partial_t \phi) \Delta \partial_t \phi\, \d x.
\end{align*}
Here we have used the fact that $\overline{\Delta \partial_t \phi}=0$ since $\partial_\n \partial_t \phi=0$ on $\partial \Omega$. Collecting the above two estimates, we find that
\begin{align*}
&\ddt \Big( \frac12  \|\omega\|_{L^p(\Omega)}^2+\frac12 \| \nabla \partial_t \phi\|_{\L2}^2 \Big)  + \| \Delta \partial_t \phi\|_{\L2}^2 \\
&\quad
\leq  \| \nabla \mu \cdot (\nabla \phi)^\perp\|_{L^p(\Omega)} \| \omega\|_{L^p(\Omega)} + \theta_0 \|\nabla  \partial_t \phi\|_{\L2}^2+ \int_{\Omega} F''(\phi) \partial_t \phi \Delta \partial_t \phi \, \d x\\
&\qquad
+\int_{\Omega} (\partial_t \uu \cdot \nabla \phi) \Delta \partial_t \phi\, \d x +  \int_{\Omega} (\uu \cdot \nabla \partial_t \phi) \Delta \partial_t \phi\, \d x.
\end{align*}
Notice that, by  \eqref{EAC}$_3$, we have the relation $\nabla \mu= \nabla \partial_t \phi + (\nabla \uu)^t \nabla \phi + (\uu\cdot \nabla ) \nabla \phi$. Exploiting this identity, we get
\begin{align*}
&\| \nabla \mu \cdot (\nabla \phi)^\perp\|_{L^p(\Omega)} \| \omega\|_{L^p(\Omega)}\\
&\quad \leq \big( \|\nabla \partial_t \phi\|_{L^p(\Omega)} + \|(\nabla \uu)^t \nabla \phi\|_{L^p(\Omega)} + \| (\uu\cdot \nabla ) \nabla \phi\|_{L^p(\Omega)} \big) \|\nabla \phi\|_{L^\infty(\Omega)} \| \omega\|_{L^p(\Omega)}.
\end{align*}
Using  the Gagliardo-Nirenberg inequality \eqref{GN2} and the following inequality for divergence-free vector fields satisfying the boundary condition \eqref{boundaryE}$_1$
\begin{equation}
\label{u-v}
\| \nabla \u\|_{L^p(\Omega)}\leq C(p) \| \omega\|_{L^p(\Omega)}, \quad p \in [2,\infty),
\end{equation}
we deduce that
\begin{align*}
&\| \nabla \mu \cdot (\nabla \phi)^\perp\|_{L^p(\Omega)} \| \omega\|_{L^p(\Omega)}\\
&\qquad  \leq C \| \nabla \partial_t \phi\|_{\L2}^\frac{2}{p} \| \Delta \partial_t \phi\|_{\L2}^{1-\frac{2}{p}} \| \nabla \phi\|_{L^\infty(\Omega)}  \| \omega\|_{L^p(\Omega)} \\
&\qquad\quad  + C \| \nabla \uu\|_{L^p(\Omega)} \| \nabla \phi\|_{L^\infty(\Omega)}^2  \| \omega\|_{L^p(\Omega)}\\
&\qquad\quad  + \| \uu\|_{L^\infty(\Omega)} \| \phi\|_{W^{2,p}(\Omega)}
 \| \nabla \phi\|_{L^\infty(\Omega)}  \| \omega\|_{L^p(\Omega)}\\
 & \qquad \leq \frac18 \| \Delta \partial_t \phi\|_{\L2}^2+ C \| \nabla \phi\|_{L^\infty(\Omega)}^{\frac{2p}{p+2}} \| \nabla \partial_t \phi\|_{\L2}^{\frac{4}{p+2}} \| \omega\|_{L^p(\Omega)}^\frac{2p}{p+2}\\
 &\qquad\quad  + C\big( \| \nabla \phi\|_{L^\infty(\Omega)}^2 + \| \phi\|_{W^{2,p}(\Omega)} \| \nabla \phi\|_{L^\infty(\Omega)} \big) \big(1+ \| \omega\|_{L^p(\Omega)}^2\big).
\end{align*}
Next, using \eqref{EAC}$_1$ together with the bounds \eqref{Reg-E1}, we have
\begin{align*}
& \int_{\Omega} \partial_t \uu \cdot \nabla \phi \Delta \partial_t \phi\, \d x\\
&\quad  \leq \int_{\Omega} \mathbb{P} \big( -\uu \cdot \nabla \uu -\Delta \phi \nabla \phi \big) \cdot \nabla \phi \Delta \partial_t \phi \, \d x\\
&\quad  \leq C \| \mathbb{P}\big( \uu \cdot \nabla \uu\big)\|_{\L2} \|\nabla \phi \Delta \partial_t \phi \|_{\L2} + C \| \mathbb{P}\big( \Delta \phi \nabla \phi\big)\|_{\L2} \| \nabla \phi \Delta \partial_t \phi\|_{\L2}\\
&\quad  \leq \frac{1}{8} \| \Delta\partial_t \phi \|_{\L2}^2
+C \| \uu\|_{L^\infty(\Omega)}^2 \| \nabla \uu\|_{\L2}^2 \| \nabla \phi\|_{L^\infty(\Omega)}^2+
C\| \Delta \phi \|_{\L2}^2 \| \nabla \phi\|_{L^\infty(\Omega)}^4\\
&\quad  \leq \frac{1}{8} \| \Delta\partial_t \phi \|_{\L2}^2
+C (1+\| \omega\|_{L^p(\Omega)}^2)  \| \nabla \phi\|_{L^\infty(\Omega)}^2+C\| \nabla \phi\|_{L^\infty(\Omega)}^4.
\end{align*}
Arguing as for \eqref{FFF} and \eqref{unpt}, we have
\begin{align}
 \int_{\Omega} F''(\phi) \partial_t \phi \Delta \partial_t \phi \, \d x  \leq \frac18 \| \Delta \partial_t \phi \|_{\L2}^2 +C \| F''(\phi)\|_{\L2}^2  \ln \big( C\| F''(\phi)\|_{\L2} \big)\| \nabla \partial_t \phi\|_{\L2}^2,\nonumber
\end{align}
and
\begin{align}
&\int_{\Omega} (\uu \cdot \nabla \partial_t \phi) \Delta \partial_t \phi\, \d x
\leq \frac{1}{8} \| \Delta \partial_t \phi\|_{\L2}^2 + C \| \nabla \partial_t \phi\|_{\L2}^2.
\nonumber
\end{align}

Collecting the above estimates and using Young's inequality, we arrive at the differential inequality
\begin{align*}
&\ddt \left(  \|\omega\|_{L^p(\Omega)}^2+ \| \nabla \partial_t \phi\|_{\L2}^2 \right) +  \| \Delta \partial_t \phi\|_{\L2}^2   \leq R_1(t) \left(   \|\omega\|_{L^p(\Omega)}^2+ \| \nabla \partial_t \phi\|_{\L2}^2 \right)  +R_2(t),
\end{align*}
where
$$
R_1= C \Big( 1+
\| \nabla \phi\|_{L^\infty(\Omega)}^2 + \| F''(\phi)\|_{\L2}^2 \ln \big( C\| F''(\phi)\|_{\L2}\big) \Big)
$$
and
$$
R_2 = C \| \phi\|_{W^{2,p}(\Omega)}^2 + C \big(\| \nabla \phi\|_{L^\infty(\Omega)}^4+1).
$$
By using \eqref{BGW}, and recalling \eqref{ineq0}, we see that
\begin{align*}
\| \nabla \phi\|_{L^\infty(\Omega)}^4
& \leq C \| \nabla^2 \phi\|_{\L2}^4 \ln^2 \big( \| \nabla^2 \phi\|_{L^p(\Omega)} \big) +1 \\
& \leq  C  \ln^2 \big( \| \phi\|_{W^{2,p}(\Omega)} \big) +1,
\end{align*}
for $p>2$. In light of \eqref{Reg-E2}, we infer that both $R_1$ and $R_2$ belong to $L^1(0,T)$. Thanks to Gronwall's lemma, we obtain
\begin{align*}
 &\|\omega(t)\|_{L^p(\Omega)}^2+ \| \nabla \partial_t \phi(t)\|_{\L2}^2 \\
 &\quad \leq  \Big( \|\omega(0)\|_{L^p(\Omega)}^2+ \| \nabla \partial_t \phi(0)\|_{\L2}^2 +  \int_0^T R_2(\tau)\, \d \tau\Big) {e}^{\int_0^T R_1(\tau)\, \d \tau},
\end{align*}
for any $t\in [0,T]$. Since $\| \omega (0)\|_{L^p(\Omega)}\leq \| \nabla \uu_0\|_{L^p(\Omega)}$ and
\begin{align*}
\| \nabla \partial_t \phi(0)\|_{\L2}
&\leq  \| (\nabla \uu_0)^t \nabla \phi_0\|_{\L2} + \| (\uu_0\cdot \nabla) \nabla \phi_0\|_{\L2} + \|\nabla \mu_0 \|_{\L2}\\
&\leq C  \| \nabla \uu_0\|_{L^p(\Omega)} \|\phi_0 \|_{H^2(\Omega)}+
C \| \uu_0\|_{L^\infty(\Omega)} \|\phi_0 \|_{H^2(\Omega)}+\|\nabla \mu_0 \|_{\L2}\\
&\leq C\| \uu_0\|_{W^{1,p}(\Omega)} \|\phi_0 \|_{H^2(\Omega)}+\|\nabla \mu_0 \|_{\L2},
\end{align*}
we deduce that for any $p\in (2,\infty)$
$$
\omega \in L^\infty(0,T;L^p(\Omega)),\quad \partial_t \phi \in L^\infty(0,T; H^1(\Omega))\cap L^2(0,T;H^2(\Omega)),  \quad \forall \, T>0.
$$
This, in turn, implies that
$$
\uu \in L^\infty(0,T;W^{1,p}(\Omega)), \quad \phi \in L^\infty(0,T;W^{2,p}(\Omega)), \quad \forall \, T>0.
$$
As a consequence, the above estimates yield that
$$
\widetilde{\mu}=-\Delta \phi+F'(\phi) \in L^2(0,T;L^\infty(\Omega)), \quad \forall \, T>0.
$$
The rest part of the proof is the same as the proof of Theorem \ref{Proreg-D} with the choice $\zeta>0$.

The proof of Theorem \ref{Th-EAC} is complete.
\end{proof}

\section{Conclusions and Future Developments}
\label{con}

In this paper, we present well-posedness results of two Diffuse Interface models that describe the evolution of incompressible binary fluid mixture having (possibly) different densities and viscosities. Our focus is on the mass-conserving Allen-Cahn relaxation of the transport equation with the physically relevant Flory-Huggins potential. We show the existence of global weak solutions in three dimensions and the existence of global strong solutions in two dimensions. For the latter, we discuss additional properties, such as uniqueness, regularity and the property of strict separation from pure states $\pm 1$.
On the other hand, there are several unsolved questions concerning the analysis of the  Navier-Stokes-Allen-Cahn and Euler-Allen-Cahn systems in the three dimensional case, which will be the subject of future investigations.

We conclude by mentioning some interesting open problems related to the results proved in this work:

$\bullet$ Two possible improvements of this work concern the Navier-Stokes-Allen-Cahn system \eqref{NSAC-D}-\eqref{IC-D}. The first question is whether the entropy estimates in Theorem \ref{strong-D} can be achieved for strong solutions with small initial data, but without restrictions on the parameters of the system, or even without any condition on the initial data. The second issue is to show the uniqueness of strong solutions given in Theorem \ref{strong-D}-(1), without relying on the entropy estimates in Theorem \ref{strong-D}-(2). Also, we mention the possibility of considering moving contact lines for the Navier-Stokes-Allen-Cahn system (see \cite{MCYZ2017} for numerical attempts).

$\bullet$ Interesting open issues regarding the Euler-Allen-Cahn system \eqref{EAC}-\eqref{ICE} are the existence and the uniqueness of solutions corresponding to an initial datum $\omega_0 \in L^\infty(\Omega)$ as well as the study of the inviscid limit on arbitrary time intervals (cf. \cite{ZGH2011} for results on short time intervals).
\smallskip

\section*{Acknowledgments}
\noindent
The authors wish to thank the anonymous referee for pointing out several references about the sharp interface limit of Diffuse Interface models and for the careful reading of the manuscript, which significantly improved the quality of our work.
Part of this work was carried out during the first and second authors' visit to School of Mathematical Sciences of Fudan University whose hospitality is gratefully acknowledged. M.~Grasselli is a member of the Gruppo Nazionale per l'Analisi Matematica, la Probabilit\`{a} e le loro applicazioni (GNAMPA) of the Istituto Nazionale di Alta Matematica (INdAM). H.~Wu is partially supported by NNSFC grant No. 12071084, 11631011 and the Shanghai Center for Mathematical Sciences at Fudan University.

\section*{Compliance with Ethical Standards}
\noindent The authors declare that they have no conflict of interest. The authors also confirm that the manuscript has not been submitted to more than one journal for simultaneous consideration and the manuscript has not been published previously (partly or in full).

\appendix
\section{Stokes System with Variable Viscosity}
\label{App-0}
\setcounter{equation}{0}
We prove an elliptic regularity result for the following Stokes  problem with concentration depending viscosity
\begin{equation}
\label{Stokes}
\begin{cases}
-\div (\nu(\phi)D \uu)+\nabla P=\f, \quad &\text{in } \Omega,\\
\div \uu=0,  \quad &\text{in } \Omega,\\
\uu=0, \ \quad &\text{on } \partial \Omega.
\end{cases}
\end{equation}
This result is a variant of \cite[Lemma 4]{A2009}.
\begin{theorem}
\label{Stokes-e}
Let $\Omega$ be a bounded domain of class $\mathcal{C}^2$ in $\mathbb{R}^d$, $d=2,3$. Assume that $\nu\in W^{1,\infty}(\mathbb{R})$ such that $0<\nu_\ast\leq \nu(\cdot)\leq \nu^\ast$  in $\mathbb{R}$, $\phi \in W^{1,r}(\Omega)$ with $r>d$, and $\f\in \mathbf{L}^p(\Omega)$ with $1<p<\infty$ if $d=2$ and $\frac65\leq p<\infty$ if $d=3$.
Consider the (unique) weak solution $\uu\in \V_\sigma$ to \eqref{Stokes} such that $$(\nu(\phi) D\uu,\nabla \ww)= (\,\f,\ww), \quad \forall\, \ww\in \V_\sigma.$$
\smallskip

\begin{itemize}
\item[1.] If $\frac{1}{p}=\frac{1}{2}+\frac{1}{r}$,
there exists $C=C(p,\Omega)>0$ such that
\begin{equation}
\label{regstokes}
\| \u\|_{\mathbf{W}^{2,p}(\Omega)} \leq
C\| \f\|_{L^p(\Omega)}+ C \| \nabla \phi\|_{L^r(\Omega)} \|D \u \|_{L^2(\Omega)}.
\end{equation}

\item[2.] Suppose that $\u \in \mathbf{V}_\sigma\cap \mathbf{W}^{1,s}(\Omega)$ with $s> 2$ such that
$$
\frac{1}{p}=\frac{1}{s}+\frac{1}{r}, \quad r\geq \frac{2s}{s-1}.
$$
Then, there exists $C=C(s,p,\Omega)>0$ such that
\begin{equation}
\label{regstokes2}
\| \u\|_{\mathbf{W}^{2,p}(\Omega)} \leq
C\| \f\|_{L^p(\Omega)}+ C \| \nabla \phi\|_{L^r(\Omega)} \|D \u \|_{L^s(\Omega)}.
\end{equation}
\end{itemize}
\end{theorem}
\begin{proof}
We denote by $B$ the Bogovskii operator. We recall that
$B: L^{q}_{(0)}(\Omega) \rightarrow W^{1,q}_0(\Omega)$, $1<q<\infty$,  such that $\div B f=f$.
It is well-known (see, e.g., \cite[Theorem III.3.1]{Galdi}) that, for all $1<q<\infty$,
\begin{equation}
\label{Bog-W1p}
\| Bf \|_{W^{1,q}(\Omega)}\leq C \| f\|_{L^q(\Omega)}.
\end{equation}
In addition, by \cite[Theorem III.3.4]{Galdi}, if $f = \div \, \g$, where $\g\in \mathbf{L}^q(\Omega)$, $1<q<\infty$, is such that $\div \g \in L^q(\Omega)$, and $\g\cdot \n=0$ almost everywhere on $\partial \Omega$, we have
\begin{equation}
\label{Bog-L2}
\| B f \|_{L^q(\Omega)}\leq C \| \g\|_{L^q(\Omega)}.
\end{equation}

For the sake of simplicity, we start proving the second part of Theorem \ref{Stokes-e}, and then we show the first part.
\smallskip

\textbf{Case 2}.
Let us take $\vv \in \C^\infty_{0,\sigma}(\Omega)$. As in \cite[Lemma 4]{A2009}, we define $$\ww=\frac{\vv}{\nu(\phi)}- B\left[\div \left(\frac{\vv}{\nu(\phi)}\right)\right].$$
We observe that $\ww \in \mathbf{W}_0^{1,r}(\Omega)$ with $\div \ww=0$. In particular, $\ww \in \V_\sigma$.
Taking $\ww$ in the weak formulation, we obtain
\begin{align*}
(D \uu,\nabla \vv)&=\left( \f, \frac{\vv}{\nu(\phi)}- B \left[ \div \left( \frac{\vv}{\nu(\phi)} \right) \right] \right) - \left( \nu(\phi) D \uu, \vv \otimes \nabla \left(\frac{1}{\nu(\phi)}\right) \right)\\
&\quad + \left( \nu(\phi) D \uu, \nabla B \left[ \div \left( \frac{\vv}{\nu(\phi)}
\right) \right] \right).
\end{align*}
Since $\frac{2s}{s-1}\leq r$, we deduce that $r\geq p'$ $(\frac{1}{p'}=1-\frac{1}{p})$. This implies that $\div \left( \frac{\vv}{\nu(\phi)} \right) \in L^{p'}(\Omega)$.
By using the assumptions on $\nu$ and the estimate \eqref{Bog-L2} with $q=p'$,
we find
\begin{align*}
\left| \left( \f, \frac{\vv}{\nu(\phi)}- B \left[ \div \left( \frac{\vv}{\nu(\phi)} \right) \right] \right) \right|
&\leq\|\, \f\|_{L^p(\Omega)} \left(  \frac{1}{\nu_\ast} \| \vv\|_{L^{p'}(\Omega)}+
\left\| B \left[ \div \left( \frac{\vv}{\nu(\phi)} \right) \right] \right\|_{L^{p'}(\Omega)} \right)\\
&\leq C \|\, \f\|_{L^p(\Omega)} \left( \frac{1}{\nu_\ast} \| \vv\|_{L^{p'}(\Omega)}+ \left\| \frac{\vv}{\nu(\phi)}\right\|_{L^{p'}(\Omega)}\right)\\
&\leq C \|\, \f\|_{L^p(\Omega)} \| \vv\|_{L^{p'}(\Omega)}.
\end{align*}
Also, we have
\begin{align*}
\left| \left( \nu(\phi) D \uu, \vv \otimes \nabla \left(\frac{1}{\nu(\phi)}\right) \right)\right|
&=\left| \left( D \uu, \frac{\nu'(\phi)}{\nu^2(\phi)} \vv \otimes \nabla \phi \right)\right|\\
&\leq C \| D \uu\|_{L^s(\Omega)} \| \nabla \phi\|_{L^r(\Omega)} \| \vv\|_{L^{p'}(\Omega)}.
\end{align*}
Recalling that $\div \vv=0$ and $r>s'$, by using \eqref{Bog-W1p} we obtain
\begin{align*}
\left| \left( \nu(\phi) D \uu, \nabla B \left[ \div \left( \frac{\vv}{\nu(\phi)} \right) \right] \right) \right|
&\leq  \| D \uu\|_{L^s(\Omega)} \left\| \nabla B \left[ \nabla \left( \frac{1}{\nu(\phi)} \right) \cdot \vv \right] \right\|_{L^{s'}(\Omega)}\\
&\leq C \| D \uu\|_{L^s(\Omega)} \left\| \nabla \left( \frac{1}{\nu(\phi)} \right) \cdot \vv \right\|_{L^{s'}(\Omega)}\\
&\leq  C \| D \uu\|_{L^s(\Omega)} \| \nabla \phi\|_{L^r(\Omega)} \| \vv\|_{L^{p'}(\Omega)}.
\end{align*}
Therefore, by the Riesz representation theorem and a density argument, we reach
$$
(D \uu, \nabla \vv)= (\,\widetilde{\f},\vv), \quad \forall \,  \vv \in \V_\sigma,
$$
where
$$
\|\, \widetilde{\f}\|_{L^p(\Omega)}\leq C \|\, \f\|_{L^p(\Omega)}+ C\| D \uu\|_{L^s(\Omega)} \| \nabla \phi\|_{L^r(\Omega)},
$$
for some $C$ depending on $s, p$ and $\Omega$.
By the regularity of the Stokes operator (see, e.g., \cite[Theorem IV.6.1]{Galdi}), the claim easily follows.
\smallskip

\textbf{Case 1}. We consider $\phi_n \in C_c^\infty(\overline{\Omega})$ such that $\phi_n \rightarrow \phi$ in $W^{1,r}(\Omega)$ as $n\rightarrow \infty$. For any $n \in \mathbb{N}$, we define $\uu_n$ as the solution to
$$
(\nu(\phi_n) D \uu_n, \nabla \ww)= (\f,\ww), \quad \forall \, \ww \in \V_\sigma.
$$
Since $\nu(\cdot)\geq \nu_\ast>0$, by taking $\ww=\uu_n$, it is easily seen that $\lbrace \uu_n\rbrace_{n\in \mathbb{N}}$ is bounded in $\V_\sigma$ independently of $n$. In addition, recalling that $W^{1,r}(\Omega)\hookrightarrow L^\infty(\Omega)$, we have $\nu(\phi_n)\rightarrow \nu(\phi)$ in $L^\infty(\Omega)$. By uniqueness of the weak solution $\uu$ to \eqref{Stokes}, we deduce that $\uu_n \rightharpoonup \uu$ weakly in $\V_\sigma$.

Let us take $$\ww=\frac{\vv}{\nu(\phi_n)}- B \left[\div \left( \frac{\vv}{\nu(\phi_n)}\right) \right]$$ with $\vv \in \C^\infty_{0,\sigma}(\Omega)$. Then we find
\begin{align*}
(D \uu_n,\nabla \vv)&=\left( \f, \frac{\vv}{\nu(\phi_n)}- B\left[\div \left(\frac{\vv}{\nu(\phi_n)}\right) \right] \right)- \left( \nu(\phi_n) D \uu_n, \vv \otimes \nabla \left(\frac{1}{\nu(\phi_n)}\right) \right) \\
&\quad + \left( \nu(\phi_n) D \uu_n, \nabla B \left[ \div \left(\frac{\vv}{\nu(\phi_n)}\right) \right] \right).
\end{align*}
Note that, by construction, $\frac{\vv}{\nu(\phi_n)} \in W^{1,q}(\Omega)$ for all $q \in [1,\infty]$. Therefore, by repeating the same computations carried out above with $s=2$, we arrive at
$$
(D \uu_n, \nabla \vv)= (\,\widetilde{\f},\vv), \quad \forall \,  \vv \in \V_\sigma,
$$
where
$$
\|\, \widetilde{\f}\|_{L^p(\Omega)}\leq C \|\, \f\|_{L^p(\Omega)}+ C\| D \uu_n\|_{L^2(\Omega)} \| \nabla \phi_n\|_{L^r(\Omega)},
$$
for some $C$ depending on $p$ and $\Omega$.
By the regularity theory of the Stokes operator, we infer
\begin{align*}
\| \uu_n\|_{W^{2,p}(\Omega)}\leq C \|\, \f\|_{L^p(\Omega)}+ C\| D \uu_n\|_{L^2(\Omega)} \| \nabla \phi_n\|_{L^r(\Omega)}.
\end{align*}
Since $\lbrace \uu_n\rbrace_{n\in \mathbb{N}}$ is bounded in $\V_\sigma$ and $\phi_n \rightarrow \phi$ in $W^{1,r}(\Omega)$, $\uu_n$ is bounded in $\W^{2,p}(\Omega)$ independently of $n$. By the choice of the parameters $r>d$ and $\frac{1}{p}=\frac12 +\frac{1}{r}$, $\W^{2,p}(\Omega)\cap \V_\sigma$ is compactly embedded in $\V_\sigma$. In particular, $\| D \uu_n\|_{L^2(\Omega)} \rightarrow \| D \uu\|_{L^2(\Omega)}$ as $n\rightarrow \infty$. As a consequence, by the lower semi-continuity of the norm with respect to the weak topology, the conclusion follows. The proof is complete.
\end{proof}

\section{Some Lemmas on ODE Inequalities}
\label{App}
\setcounter{equation}{0}
For convenience of the readers, we collect some useful results concerning ODE inequalities that have been used in this paper. First, we report the Osgood lemma.

\begin{lemma}
\label{Osgood}
Let $f$ be a measurable function from $[0,T]$ to $[0,a]$, $g \in L^1(0,T)$, and $W$ a continuous and nondecreasing function from $[0,a]$ to $\mathbb{R}^+$. Assume that, for some $ c \geq 0$, we have
$$
f(t)\leq c+ \int_0^t g(s) W(f(s)) \, \d s, \quad \text{for a.e.}\, t\in [0,T].
$$
\begin{itemize}
\item[-] If $c>0$, then for almost every $t \in [0,T]$
$$
-\mathcal{M}(f(t))+\mathcal{M}(c)\leq \int_0^T g(s)\, \d s, \quad \text{where} \quad
\mathcal{M}(s)=\int_s^a \frac{1}{W(s)} \, \d s.
$$
\item[-] If $c=0$ and $\int_0^a \frac{1}{W(s)} \, \d s= \infty$, then $f(t)=0$ for almost every $t\in [0,T]$.
\end{itemize}
\end{lemma}

Next, we report two generalizations of the classical Gronwall lemma and the uniform Gronwall lemma.

\begin{lemma}
\label{GL2}
Let $f$ be a positive absolutely continuous function on $[0,T]$ and $g$, $h$ be two summable functions on $[0,T]$ that satisfy the differential inequality
$$
\ddt f(t)\leq g(t)f(t)\ln\big( e+ f(t)\big)+h(t),
$$
for almost every $t\in [0,T]$.
Then, we have
$$
f(t)\leq  \big( e+f(0)\big)^{e^{\int_0^t g(\tau)\, \d \tau}} e^
{ \int_0^t e^{\int_\tau^t g(s)\, \d s} h(\tau) \, \d \tau},
\quad \forall
\, t \in [0,T].
$$
\end{lemma}
\begin{lemma}
\label{UGL2}
Let $f$ be an absolutely continuous positive function on $[0,\infty)$
and $g$, $h$ be two positive locally summable functions on $[0,\infty)$ which satisfy the differential inequality
$$
\ddt f(t) \leq g(t)f(t)\ln\big( e+ f(t)\big) +h(t),
$$
for almost every $t\geq 0$, and the uniform bounds
$$
\int_t^{t+r} f(\tau)\, \d \tau\leq a_1,
\quad \int_t^{t+r} g(\tau)\, \d \tau \leq a_2,
\quad \int_t^{t+r} h(\tau)\, \d \tau \leq a_3, \quad \forall \, t \geq 0,
$$
for some positive constants $r$, $a_1$, $a_2$, $a_3$.
Then, we have
$$
f(t)\leq e^{\big(\frac{r+a_1}{r}+a_3\big)e^{a_2}}, \quad \forall \, t \geq r.
$$
\end{lemma}


\end{document}